\documentclass[10]{article}

\usepackage{amsfonts}
\usepackage{subfigure,graphicx}
\usepackage{amsmath,amssymb,amsthm}
\usepackage[numbers]{natbib}
\usepackage{hyperref}
\usepackage{dsfont}
\usepackage{cases}
\usepackage{fullpage}
\usepackage{tikz}
\usetikzlibrary{shapes,arrows}
\usepackage{todonotes}

\newcommand{\PR}{\mathbb{P}}
\newcommand{\E}{\mathbb{E}}
\newcommand{\LL}{\mathbb{L}^2}

\newcommand{\Q}{\mathbb{Q}}
\newcommand{\R}{\mathbb{R}}
\newcommand{\N}{\mathbb{N}}
\newcommand{\Hs}{\mathcal{H}_s}

\newcommand{\cH}{\mathcal{H}}
\newcommand{\cA}{\mathcal{A}}
\newcommand{\cC}{\mathcal{C}}
\newcommand{\cD}{\mathcal{D}}
\newcommand{\cN}{\mathcal{N}}
\newcommand{\cP}{\mathcal{P}}

\newcommand{\bftheta}{\boldsymbol{\theta}}
\newcommand{\bfmu}{\boldsymbol{\mu}}
\newcommand{\bfX}{\mathbf{X}}

\newcommand{\dd}{\textrm{d}}
\newcommand{\indicator}[1]{\mathds{1}_{#1}}
\newcommand{\eqdef}{:=} 
\newcommand{\Id}{\mathbb{I}}

\renewcommand{\leq}{\leqslant}
\renewcommand{\geq}{\geqslant}
\renewcommand{\epsilon}{\varepsilon}
\renewcommand{\hat}{\widehat}
\renewcommand{\tilde}{\widetilde}
\newcommand{\paragraphnospace}[1]{\medskip \noindent\textbf{#1}~~}
\newcommand{\eg}{e.g.}
\newcommand{\ie}{i.e.}

\DeclareMathOperator{\KL}{KL}

\DeclareMathOperator{\Ber}{Ber}

\newcommand{\argmin}{\operatornamewithlimits{arg\,min}}

\newtheorem{lem}{Lemma}
\newtheorem{rem}{Remark}
\newtheorem{prop}{Proposition}
\newtheorem{thm}{Theorem}
\newtheorem*{theo}{Theorem}  

\title{Optimal functional supervised classification with separation condition}
\author{S\'{e}bastien Gadat, S\'{e}bastien Gerchinovitz, and Cl\'{e}ment Marteau}

\begin{document}
\maketitle

\abstract{We consider the binary supervised classification problem with the Gaussian functional model introduced in \cite{Cadre}. Taking advantage of the Gaussian structure, we design a natural plug-in classifier and derive a family of upper bounds on its worst-case excess risk over Sobolev spaces. These bounds are parametrized by a separation distance quantifying the difficulty of the problem, and are proved to be optimal (up to logarithmic factors) through matching minimax lower bounds. Using the recent works of \cite{chaudhuri} and \cite{GKM16} we also derive a logarithmic lower bound showing that the popular $k$-nearest neighbors classifier is far from optimality in this specific functional setting.}

\section{Introduction}
\label{sec:intro}

The binary supervised classification problem is perhaps one of the most common tasks in statistics and machine learning. Even so, this problem still fosters new theoretical and applied questions because of the large variety of the data encountered so far. 
We refer the reader to \cite{DGL96} and \cite{BBG2005} and to the references therein for a comprehensive introduction to binary supervised classification. This problem unfolds as follows. The learner has access to $n$ independent copies $(X_1,Y_1),\ldots, (X_n,Y_n)$ of a pair $(X,Y)$, where~$X$ lies in a measurable space $\mathcal{H}$ and $Y \in \{0,1\}$. The goal of the learner is to predict the label $Y$ after observing the new input $X$, with the help of the sample $\mathcal{S}_n := (X_i,Y_i)_{1 \leq i \leq n}$ to learn the unknown joint distribution $\PR_{X,Y}$ of the pair $(X,Y)$.

In some standard situations, $X$ lies in the simplest possible Hilbert space: $\mathcal{H} = \mathbb{R}^d$, which corresponds to the finite-dimensional binary classification problem. This setting has been extensively studied so far. Popular classification procedures that are now theoretically well understood include the ERM method \cite{MT99,AT07}, the $k$-nearest neighbors algorithm \cite{G78,chaudhuri,BD2015,GKM16}, support vector machines \cite{Steinwart}, or random forests \cite{BiauScornet}, just to name a few.
 
However there are situations where the inputs $X_i$ and $X$ are better modelled as functions; the set~$\mathcal{H}$ is then infinite-dimensional. Practical examples can be found, \eg, in stochastic population dynamics \cite{L03}, in signal processing \cite{C02}, or in finance \cite{LL96}. This binary supervised functional classification problem was tackled with nonparametric procedures such as kernel methods or the $k$-nearest neighbours algorithm. For example, \cite{CoH67} studied the nearest neighbour rule in any metric space, while \cite{KP95} analyzed the performances of the $k$-nearest neighbours algorithm in terms of a metric covering measure. Such metric entropy arguments were also used in \cite{CG06}, or with kernel methods in \cite{ABC03}.\\

\paragraphnospace{Our functional model.} In the present work, we focus on one of the most elementary diffusion classification model: we suppose that the input $X=(X(t))_{t \in [0,1]}$ is a continuous trajectory, solution to the stochastic differential equation
\begin{equation}\label{eq:model}
\forall t \in [0,1], \qquad 
dX(t) = Y f(t) dt + (1-Y)g(t) dt +  dW(t) \,,
\end{equation}
where $(W(t))_{0 \leq t \leq 1}$ is a standard Brownian motion, and where $Y$ is a Bernoulli $\mathcal{B}(1/2)$ random variable independent from $(W(t))_{0 \leq t \leq 1}$. In particular, in the sample $\mathcal{S}_n$, trajectories $X_i$ labeled with $Y_i=1$ correspond to observations of the signal $f$, while trajectories $X_i$ labeled with $Y_i=0$ correspond to $g$.  \\

The white noise model has played a key role in statistical theoretical developments; see, \eg, the seminal contributions of \cite{IH81} in nonparametric estimation and of \cite{L90} in adaptive nonparametric estimation.
In our supervised classification setting, the goal is not to estimate $f$ and $g$ but to predict the value of $Y$ given an observed continuous trajectory $(X(t))_{0 \leq t \leq 1}$. Of course, we assume that both functions $f$ and $g$ are unknown so that the joint distribution $\mathbb{P}_{X,Y}$ of the pair $((X(t))_{t\in[0,1]},Y)$ is unknown. Without any assumption on $f$ and $g$, there is no hope to solve this problem in general. 
However, learning the functions $f$ and $g$ (and thus $\mathbb{P}_{X,Y}$) from the sample $\mathcal{S}_n$ becomes statistically feasible when $f$ and $g$ are smooth enough.
 
The functional model considered in this paper is very close to the one studied by \cite{Cadre}. Actually our setting is less general since \cite{Cadre} considered more general diffusions driven by state-dependent drift terms $t \longmapsto f(t,X(t))$ and $t \longmapsto g(t,X(t))$. We focus on a simpler model, but derive refined risk bounds (with a different approach) that generalize the worst-case bounds of \cite{Cadre}, as indicated below.

\paragraphnospace{Some notation.} We introduce some notation and definitions in order to present our contributions below. In our setting, a \emph{classifier} $\Phi$ is a measurable function, possibly depending on the sample $\mathcal{S}_n$, that maps each new input $X=(X(t))_{t\in[0,1]}$ to a label in $\{0,1\}$. The \emph{risk} associated with each classifier $\Phi$ depends on $f$ and $g$ and is defined by: 
$$ \mathcal{R}_{f,g}(\Phi) := \PR (\Phi(X) \not = Y) \,,$$
where the expectation is taken with respect to all sources of randomness (\ie, both the sample $\mathcal{S}_n$ and the pair $(X,Y)$). The goal of the learner is to construct a classifier $\hat{\Phi}$ based on the sample $\mathcal{S}_n$ that mimics the \emph{Bayes classifier}
\begin{equation}
\Phi^\star = \argmin_\Phi \mathcal{R}_{f,g}(\Phi) \,,
\label{eq:oracle}
\end{equation}
where the infimum is taken over all possible classifiers (the oracle $\Phi^\star$ is impractical since $f$ and $g$ and thus $\mathbb{P}_{X,Y}$ are unknown). We measure the quality of $\hat{\Phi}$ through its \emph{worst-case excess risk}
\begin{equation}\label{eq:max}
 \sup_{(f,g) \in \mathcal{E}} \left\{\mathcal{R}_{f,g}(\Phi) - \inf_{\Phi} \mathcal{R}_{f,g}(\Phi)\right\}
\end{equation}
over some set $\mathcal{E}$ of pairs of functions. In the sequel we focus on Sobolev classes $\Hs(R)$ (see~\eqref{def:Hs}) and consider subsets $\mathcal{E} \subseteq \mathcal{H}_s(R)^2$ parametrized by a separation lower bound $\Delta$ on $\|f-g\|$.

\paragraphnospace{Main contributions and outline of the paper.} In Section~\ref{s:model} we first state preliminary results about the margin behavior that will prove crucial in our analysis. We then make three types of contributions:
\begin{itemize}
\item In Section~\ref{s:upper} we design a classifier $\hat \Phi_{d_n}$ based on a thresholding rule. We derive an excess risk bound that generalizes the worst-case results of \cite{Cadre} but also imply faster rates when the distance $\|f-g\|$ is large. This acceleration is a consequence of the nice properties of the margin (see also, \eg, \cite{AT07} and \cite{GKM16}).
\begin{theo}[A] The classifier $\hat \Phi_{d_n}$ defined in~\eqref{eq:classif} with $d_n \approx n^{\frac{1}{2s+1}}$ has an excess risk roughly bounded by (omitting logarithmic factors and constant factors depending only on $s$ and $R$): for $n \geq N_{R,s}$ large enough,
\[
\sup_{\substack{f,g \in \Hs(R) \\ \|f-g\| \geq \Delta}} \left\{ \mathcal{R}_{f,g}(\hat\Phi_{d_n}) - \inf_{\Phi} \mathcal{R}_{f,g}(\Phi)  \right\}  \lesssim \left\{\begin{array}{ll}
	n^{-\frac{s}{2s+1}} & \textrm{if } \Delta \lesssim n^{-\frac{s}{2s+1}} \\[0.2em]
	\displaystyle \frac{1}{\Delta} n^{-\frac{2s}{2s+1}} & \textrm{if } \Delta \gtrsim n^{-\frac{s}{2s+1}}
\end{array}\right.
\]
\end{theo}
\item In Section~\ref{s:minimaxfano} we derive a matching minimax lower bound (up to logarithmic factors) showing that the above worst-case bound cannot be improved by any classifier.

\begin{theo}[B]
For any number $n \geq N_{R,s}$ of observations, any classifier $\hat{\Phi}$ must satisfy (omitting again logarithmic factors and constant factors depending only on $s$ and $R$):
\[
\sup_{\substack{f,g \in \Hs(R) \\ \|f-g\| \geq \Delta}} \left\{ \mathcal{R}_{f,g}(\hat\Phi) - \inf_{\Phi} \mathcal{R}_{f,g}(\Phi)  \right\} \gtrsim \left\{\begin{array}{ll}
	n^{-\frac{s}{2s+1}} & \textrm{if } \Delta \lesssim n^{-\frac{s}{2s+1}} \\[0.2em]
	\displaystyle \frac{1}{\Delta} n^{-\frac{2s}{2s+1}} & \textrm{if } \Delta \gtrsim n^{-\frac{s}{2s+1}}
\end{array}\right.
\]
\end{theo}
\item Finally, in Section~\ref{s:knn}, we show that the well-known $k$-nearest neighbors rule tuned in a classical and optimal way (see, \eg, \cite{S12,GKM16}) is far from optimality in our specific functional setting.
\begin{theo}[D]
For any threshold (dimension) $\hat{d} \in \mathbb{N}^*$ based on a sample-splitting policy, and for the optimal choice $k_n^{opt}(\hat{d})=\lfloor n^{4/(4+\hat{d})}\rfloor$, the $\hat{d}$-dimensional $k_n^{opt}(\hat{d})$-nearest neighbors classifier $\Phi_{\textrm{NN}}$ suffers a logarithmic excess risk in the worst case:
$$\sup_{f,g \in \Hs(r)} \Bigl\{ \mathcal{R}_{f,g}(\Phi_{\textrm{NN}}) - \inf_{\Phi} \mathcal{R}_{f,g}(\Phi) \Bigr\} \gtrsim \log(n)^{-2s} \,.$$
\end{theo}
\end{itemize} 

\ \\
Most proofs are postponed to Appendix~\ref{s:A} (for the upper bounds) and to Appendix~~\ref{s:minimaxlowerbound} (for the lower bounds). \\

\paragraphnospace{Other useful notation.} We denote the joint distribution of the pair $((X(t))_{t\in[0,1]},Y)$ by $\mathbb{P}_{X,Y}$, and write $\mathbb{P}_{\otimes^n} = P_{X,Y}^{\otimes n}$ for the joint distribution of the sample $(X_i,Y_i)_{1 \leq i \leq n}$. For notational convenience, the measure $\mathbb{P}$ will alternatively stand for $\mathbb{P} = \mathbb{P}_{X,Y} \times \mathbb{P}_{\otimes^n}$ (we integrate over both the sample $\mathcal{S}_n$ and the pair $(X,Y)$) or for any other measure made clear by the context. The distribution of $(X(t))_{t\in[0,1]}$ will be denoted by $\mathbb{P}_X$, while the distribution of $(X(t))_{t\in[0,1]}$ \textit{conditionally} on the event $\lbrace Y=1 \rbrace$ (resp. $\lbrace Y=0 \rbrace$) will be written as $\mathbb{P}_f$ (resp. $\mathbb{P}_g$). 

We write $\LL([0,1])$ for the set of square Lebesgue-integrable functions $f$ on $[0,1]$, with $\LL$-norm $\|f\| = \bigl(\int_{0}^1 f^2(t) d t\bigr)^{1/2}$ and inner product $\langle f,g \rangle = \int_{0}^1 f(t) g(t) d t$. With a slight abuse of notation, when $X$ is a solution of~\eqref{eq:model} and $\varphi \in \LL([0,1])$, we set $\langle \varphi,X \rangle = \int_{0}^1 \varphi(t) dX(t)$.

Finally we write $\Ber(p)$ or simply $\mathcal{B}(p)$ for the Bernoulli distribution of parameter $p \in [0,1]$, as well as $\mathcal{B}(n,p)$ for the binomial distribution with parameters $n \in \N^*$ and $p \in [0,1]$. We also set $x \wedge y = \min\{x,y\}$ for all $x,y \in \R$.

\section{Preliminary results}
\label{s:model}

\subsection{Bayes classifier}
\label{sec:BayesRule}
We start by deriving an explicit expression for the optimal classifier $\Phi^\star$ introduced in (\ref{eq:oracle}). This optimal classifier is known as \emph{the Bayes classifier} of the classification problem
(see, \eg, \cite{G78,DGL96}).

Let $\mathbb{P}_0$ denote the Wiener measure on the set of continuous functions on $[0,1]$.
It is easy to check that the law of $X \vert Y$ is absolutely continuous with respect to $\mathbb{P}_0$ (see, \eg, \cite{IW_1989}). Indeed, for any continuous trajectory $X$, the Girsanov formula implies that the density of $\PR_f$ (\ie, of $X|\{Y=1\}$) with respect to the reference measure $\mathbb{P}_0$ is given by
\begin{equation}
q_f(X) := 
\frac{d \mathbb{P}_f}{d \mathbb{P}_0}(X) = \exp \left(\int_{0}^1 f(s) dX_s - \frac{1}{2}\|f\|^2 \right) \,.
\label{eq:qf}
\end{equation}
Similarly, the density of $\PR_g$ (\ie, of $X|\{Y=0\}$) with respect to $\mathbb{P}_0$ is
\begin{equation}
q_g(X) := \frac{d \mathbb{P}_g}{d \mathbb{P}_0}(X) = \exp \left(\int_{0}^1 g(s) dX_s - \frac{1}{2}\|g\|^2 \right) \,.
\label{eq:qg}
\end{equation}
In the sequel we refer to $q_f$ and $q_g$ as the likelihood ratios of the models $\PR_f$ and $\PR_g$ versus $\PR_0$. Now,  using the Bayes formula, we can easily see that the regression function $\eta$ associated with~(\ref{eq:model}) is given by
\begin{eqnarray}
\eta(X) &:=& \mathbb{E}[Y \vert X] = \PR(Y=1 \vert X) \nonumber \\
&= &\frac{\frac{d \PR_f}{d \PR_0}(X)}{ \frac{d \PR_f}{d \PR_0}(X)+\frac{d \PR_g}{d \PR_0}(X)} \nonumber \\
& = & \frac{\exp\left(\int_{0}^1 (f(s)-g(s)) dX_s - \frac{1}{2}\|f\|^2 + \frac{1}{2}\|g\|^2\right)}{1+\exp\left(\int_{0}^1 (f(s)-g(s)) dX_s - \frac{1}{2}\|f\|^2 + \frac{1}{2}\|g\|^2\right)} \;.
\label{eq:eta}
\end{eqnarray}
As an example, if we assume that we observe $dX(t)=f(t) dt$ with $X(0)=0$, then $\eta(X) = \frac{\exp \left( \frac{1}{2} \|f-g\|^2 \right)}{1+\exp \left( \frac{1}{2} \|f-g\|^2 \right)}$, which is larger than or equal to $1/2$ and gets closer to $1$ when $\|f-g\|$ increases. Roughly speaking, this means in that example that the distribution $\mathbb{P}_f$ is more likely than the distribution $\mathbb{P}_g$, which is consistent with the definition of the model given by~\eqref{eq:model}.

\ \\
The Bayes classifier $\Phi^\star$ of the classification problem is then given by
\begin{equation}\label{eq:bayes_rule}
\Phi^{\star}(X) := 
\mathds{1}_{\left\lbrace \eta(X)\geq \frac12 \right\rbrace} =
\mathds{1}_{\left\lbrace \int_{0}^1 (f(s)-g(s)) dX_s \geq  \frac{1}{2}\|f\|^2 - \frac{1}{2}\|g\|^2  \right\rbrace} \;.
\end{equation}
It is well known that the Bayes classifier $\Phi^{\star}$ corresponds to the optimal classifier of the considered binary classification problem (see, \eg, \cite{DGL96}) in the sense that it satisfies~\eqref{eq:oracle}. In particular, for any other classifier~$\Phi$, the excess risk of classification is given by
\begin{equation}\label{eq:excess_risk}
\mathcal{R}_{f,g}(\Phi)-\mathcal{R}_{f,g}(\Phi^{\star}) = \mathbb{E} \left[ \left|2 \eta(X)-1\right| \mathds{1}_{\Phi(X) \neq \Phi^{\star}(X)}\right] \,.
\end{equation}

\noindent
In our statistical setting, the functions $f$ and $g$ are unknown so that it is impossible to compute the oracle Bayes classifier~\eqref{eq:bayes_rule}. However, we can construct an approximation of it using the sample $(X_i,Y_i)_{1 \leq i \leq n}$. In Section \ref{s:upper} we design a plug-in estimator combined with a projection step, and analyze its excess risk under a smoothness assumption on $f$ and $g$. The next result on the margin will be a key ingredient of our analysis.

\subsection{Control of the margin in the functional model}
\label{s:margin}
As was shown in earlier works on binary supervised classification (see, \eg, \cite{MT99} or \cite{AT07}), the probability mass of the region where the regression function $\eta$ is close to $1/2$ plays an important role in the convergence rates. The behaviour of the function $\eta$ is classically described by a so-called \textit{margin assumption}: there exist $\alpha \geq 0$ and $\epsilon_0,C > 0$ such that, for all $0 < \epsilon \leq \epsilon_0$,
\begin{equation}
\PR_{X} \left( \left| \eta(X) - \frac{1}{2} \right| \leq \epsilon   \right) \leq C \epsilon^\alpha \,.
\label{eq:marginassumption}
\end{equation}

\ \\
We will show in Proposition~\ref{thm:margin} which parameters $\alpha,\epsilon_0,C > 0$ are associated with Model~(\ref{eq:model}). The role of~\eqref{eq:marginassumption} is easy to understand: classifying a trajectory $X$ for which $\eta(X)$ is close to $1/2$ is necessarily a challenging problem because the events $\{Y=1\}$ and $\{Y=0\}$ are almost equally likely. This not only makes the optimal (Bayes) classifier $\mathds{1}_{\eta(X)\geq 1/2}$ error-prone, but it also makes the task of mimicking the Bayes classifier difficult. Indeed, any slightly bad approximation of $\eta$ when $\eta(X) \simeq 1/2$ can easily lead to a prediction different from $\mathds{1}_{\eta(X)\geq 1/2}$. A large value of the margin parameter $\alpha$ indicates that most trajectories $X$ are such that $\eta(X)$ is far from $1/2$: this makes in a sense the classification problem easier. \\

Our first contribution, detailed in Proposition~\ref{thm:margin} below, entails that the margin parameter associated with Model~(\ref{eq:model}) crucially depends on the distance between the functions $f$ and $g$ of interest. The proof is postponed to Appendix~\ref{s:pmargin}. \\

\begin{prop}
\label{thm:margin}
Let $X$ be distributed according to Model~(\ref{eq:model}), and set $\Delta := \| f-g \|$. Then, for all $ 0 < \epsilon \leq 1/8$, we have
$$ \PR_{X} \left( \left| \eta(X) - \frac{1}{2} \right| \leq \epsilon   \right) \leq 1 \wedge \frac{10 \epsilon}{\Delta} \,.$$
\end{prop}

 In particular, if the distance $\| f-g \|$  is bounded from below by a positive constant, then~\eqref{eq:marginassumption} is satisfied with a margin parameter $\alpha=1$. If, instead, $\|f-g\|$ is allowed to be arbitrarily small, then nothing can be guaranteed about the margin parameter (except the obvious value $\alpha=0$ that always works).


\section{Upper bounds on the excess risk}
\label{s:upper}
In this section we construct a classifier with nearly optimal excess risk.
We detail its construction in Section~\ref{sec:pluginFinitedim} and analyze its approximation and estimation errors in Sections~\ref{sec:trunc} and~\ref{sec:estimationerror}. Our main result, Theorem \ref{prop:minmax}, is stated in Section \ref{subsection:upper_bound}. Nearly matching lower bounds will be provided in Section~\ref{s:lower}.

\subsection{A classifier in a finite-dimensional setting}
\label{sec:pluginFinitedim}

Our classifier---defined in Section~\ref{sec:defclassifier} below---involves a projection step with coefficients $\theta_j$ and $\mu_j$ introduced in Section~\ref{sec:HilbertWhitenoise} and estimated in Section~\ref{sec:estimatethetamu}.

\subsubsection{$\LL$ orthonormal basis $(\varphi_j)_{j \geq 1}$ and white noise model}
\label{sec:HilbertWhitenoise}

\textit{Orthonormal basis.} Let $(\varphi_j)_{j \in \mathbb{N}^*}$ be any orthonormal basis of $\LL([0,1])$, and $d \in \mathbb{N}^*$ be some dimension that will be chosen as a function of the size $n$ of the sample (for projection purposes). In the sequel, the coefficients $(c_j(h))_{j \geq 1}$ of any function $h \in \LL([0,1])$ w.r.t.~the basis $(\varphi_j)_{j \geq 1}$ are defined by
$$
c_j(h) := \langle \varphi_j, h \rangle = \int_{0}^{1} h(s) \varphi_j(s) ds \,, \qquad j \geq 1 \,,
$$
and its $\LL$-projection onto  $\mathrm{Span}\left(  \varphi_j , \ 1 \leq j \leq d  \right)$ is given by
\begin{equation}\label{eq:defPid}
\Pi_d(h) = \sum_{j=1}^d c_j(h) \varphi_j \,.
\end{equation}
In particular, we will pay a specific attention to the coefficients of $f$ and $g$ involved in~\eqref{eq:model},
\begin{equation}\label{eq:defckfckg}
\forall j \geq 1, \qquad 
\theta_j := c_j(f) \qquad \text{and} \qquad \mu_j:= c_j(g) \,,
\end{equation}
and to their $d$-dimensional projections $f_d := \Pi_d(f) = \sum_{j=1}^d \theta_j \varphi_j$ and $g_d:=\Pi_d(g) = \sum_{j=1}^d \mu_j \varphi_j$.

\ \\
\textit{White noise model.} We now make a few comments on the white noise model considered in~\eqref{eq:model} (see also \cite{IH81} for further details and its link with the infinite Gaussian sequence model). First note that, for all $\varphi \in \LL([0,1])$, almost surely,
$$
\int_{0}^1 \varphi(t) dX(t) = 
Y \int_{0}^1 f(t) \varphi(t) dt + (1-Y) \int_{0}^1 g(t) \varphi(t) dt + \int_{0}^1 \varphi(t) dW(t) \,.
$$
Recall that, with a slight abuse of notation, we write $\langle \varphi,X \rangle := \int_{0}^1 \varphi(t) dX(t)$. The above almost sure equality implies that the conditional distribution of $\langle \varphi,X \rangle$ given $Y$ is Gaussian with expectation $\langle \varphi,f \rangle \mathds{1}_{Y=1} + \langle \varphi,g \rangle \mathds{1}_{Y=0}$ and variance $\|\varphi\|^2$. Therefore, the distribution of $\langle \varphi , X\rangle$ is a mixture of two Gaussian distributions:
$$\frac{1}{2} \mathcal{N}\Bigl(\langle \varphi,f\rangle,\|\varphi\|^2\Bigr) + \frac{1}{2} \mathcal{N}\Bigl(\langle\varphi,g\rangle,\|\varphi\|^2\Bigr) \,.$$

\ \\
An important feature of the white noise model is that the coefficients $\bigl(\langle \varphi_j,W \rangle\bigr)_{j \geq 1}$ associated with different frequencies of the standard Brownian motion are independent. This is because they are jointly Gaussian, with a diagonal infinite covariance matrix:
$$\forall j \neq j',  \qquad \mathbb{E} \Big[\langle \varphi_j,W \rangle \langle \varphi_{j'},W \rangle\Big] = \int_{0}^1 \varphi_j(t) \varphi_{j'}(t) d t = 0 \,.$$
The above remarks imply together with~\eqref{eq:defckfckg} that
\begin{equation}
\label{eq:Gaussiancoeffs}
\forall j \geq 1, \qquad \mathcal{L}\left( \langle \varphi_j,X \rangle|Y=1\right) = \mathcal{N}(\theta_j,1) \qquad \text{and} \qquad  \mathcal{L}\left(\langle \varphi_j,X \rangle|Y=0\right) =  \mathcal{N}(\mu_j,1) \,,
\end{equation}
and that the coefficients $\bigl(\langle \varphi_j,X \rangle\bigr)_{j \geq 1}$ are conditionally independent given $Y$.

\subsubsection{Estimation of $(\theta_j)_{1 \leq j \leq d}$ and $(\mu_j)_{1 \leq j \leq d}$}
\label{sec:estimatethetamu}

In order to estimate the $\theta_j$ and 
$\mu_j$, we split the sample $(X_i,Y_i)_{1 \leq i \leq n}$ into two subsamples corresponding to either $Y_i=0$ or $Y_i=1$. More formally, we define the two subsamples $(X^0_i)_{1 \leq i \leq N_0}$ and $(X^1_i)_{1 \leq i \leq N_1}$ (one of which can be empty) by
\[
\left\{\begin{array}{l}
	X^0_i := X_{\tau^0_i} \,, \quad 1 \leq i \leq N_0 \\
	X^1_i := X_{\tau^1_i} \,, \quad 1 \leq i \leq N_1
\end{array}\right.
\]
where
\begin{equation}\label{eq:defN0N1}
 N_0 := \sum_{i=1}^n \mathds{1}_{\lbrace Y_i = 0\rbrace} \quad \mathrm{and} \quad N_1 := \sum_{i=1}^n \mathds{1}_{\lbrace Y_i = 1\rbrace} \,,
\end{equation}
and where $\tau^k_i$ is the index $t \in \{1,\ldots,n\}$ such that $Y_t = k$ for the $i$-th time, \ie, for all $k \in \{0,1\}$ and $i \in \{1,\ldots,N_k\}$,
\[
\tau^k_i := \min \left\{t \in \{1,\ldots,n\} : \; \sum_{t'=1}^t \mathds{1}_{\{Y_{t'}=k\}} \geq i \right\} \,.
\]
The sizes $N_0$ and $N_1$ are random variables; they satisfy $N_0 + N_1 = n$ and both have a binomial distribution $\mathcal{B}(n,1/2)$. In particular, both subsamples have (with high probability) approximately the same sizes.

\ \\
Note from~\eqref{eq:model} that the two subsamples $(X^0_i)_{1 \leq i \leq N_0}$ and $(X^1_i)_{1 \leq i \leq N_1}$ correspond to observations of the functions $g$ and $f$ respectively.  Following our comments from Section~\ref{sec:HilbertWhitenoise}, it is natural to define the random coefficients $(X_{i,j}^0)_{\substack{1 \leq i \leq N_0\\ 1 \leq j \leq d}}$ and $(X_{i,j}^1)_{\substack{1 \leq i \leq N_1\\ 1 \leq j \leq d}}$ by
\begin{equation}\label{eq:defX0X1}
\left\lbrace  
\begin{array}{l}
X_{i,j}^0 := \langle \varphi_j, X_i^0 \rangle  = \mu_j + \epsilon_{i,j}^0 \,, \quad i=1\dots N_0 \,, \\
X_{i,j}^1 := \langle \varphi_j, X_i^1 \rangle = \theta_j + \epsilon_{i,j}^1 \,, \quad i=1\dots N_1 \,,
\end{array}
\right.   \quad j\in \lbrace 1,\dots, d \rbrace ,
\end{equation}
where the dimension $d \in \mathbb{N}^*$ will be determined later (as a function of $n$), and where
$$
\epsilon_{i,j}^k = \int_{0}^1 \varphi_{j}(t) dW_{\tau^k_i}(t) \,.
$$
By independence of the random variables $Y_1,W_1,\ldots,Y_n,W_n$ used to generate the sample $(X_i,Y_i)_{1 \leq i \leq n}$ according to~\eqref{eq:model}, and by the comments made in~Section~\ref{sec:HilbertWhitenoise}, we have the following conditional independence property for the $\epsilon_{i,j}^k$.

\begin{rem}
\label{rmk:conditionallyGaussian}
Conditionally on $(Y_1,\ldots,Y_n)$, the $n d$ random variables (or any $(Y_1,\ldots,Y_n)$-measurable permutation of them)
\[
\begin{array}{l}
\epsilon^0_{1,1},\ldots,\epsilon^0_{1,d},\epsilon^0_{2,1},\ldots,\epsilon^0_{2,d},\ldots,\epsilon^0_{N_0,1},\ldots,\epsilon^0_{N_0,d}, \\
\epsilon^1_{1,1},\ldots,\epsilon^1_{1,d},\epsilon^1_{2,1},\ldots,\epsilon^1_{2,d},\ldots,\epsilon^1_{N_1,1},\ldots,\epsilon^1_{N_1,d}
\end{array}
\]
are i.i.d.~$\cN(0,1)$. As a consequence, on the event $\{N_0>0\} \cap \{N_1>0\}$, the random variables $N_k^{-1/2} \sum_{i=1}^{N_k} \epsilon^k_{i,j}$, $1 \leq j \leq d$, $k \in \{0,1\}$, are i.i.d.~$\cN(0,1)$ conditionally  on $(Y_1,\ldots,Y_n)$.
\end{rem}

\ \\
For every $j\in \lbrace 1, \dots, d \rbrace$, we use the coefficients defined in \eqref{eq:defX0X1} to estimate the coefficients $\mu_j$ and $\theta_j$ by
\begin{equation}
\label{eq:defhatthethamu}
\hat \mu_j = \mathds{1}_{\{N_0>0\}}\frac{1}{N_0} \sum_{i=1}^{N_0} X_{i,j}^0  \quad \mathrm{and} \quad \hat \theta_j = \mathds{1}_{\{N_1>0\}} \frac{1}{N_1} \sum_{i=1}^{N_1} X_{i,j}^1 \,.
\end{equation}
Note that we arbitrarily impose the value $0$ for $\hat{\mu}_j$ when $N_0=0$ or for $\hat{\theta}_j$ when $N_1=0$. This convention has a negligible impact, since with high probability $N_0$ and $N_1$ are both positive.

\subsubsection{A simple classifier}
\label{sec:defclassifier}

We now build a simple classifier using the estimators $\hat{\mu}_j$ and $\hat{\theta}_j$ defined in~\eqref{eq:defhatthethamu}. After observing a new trajectory $X = (X(t))_{0 \leq t \leq 1}$, we construct the vector $\mathbf{X}_d \in \R^d$ defined by 
$$
\mathbf{X}_d := \Bigl(\langle \varphi_1, X \rangle, \ldots, \langle \varphi_d, X \rangle\Bigr) \,.
$$ 
Then, we assign the label $1$ to the trajectory $X$ if $\mathbf{X}_d$ is closer to $\hat{\theta} := (\hat\theta_1,\dots,\hat\theta_d)$ than to $\hat{\mu}:=(\hat\mu_1,\dots,\hat\mu_d)$, and the label $0$ otherwise. More formally, our classifier $\hat\Phi_d$ is defined for all trajectories $X$ by
\begin{equation}
 \hat \Phi_d(X) = \left\lbrace
\begin{array}{lcr}
1 & \mathrm{if} &  \| \mathbf{X}_d - \hat \theta \|_d \leq \| \mathbf{X}_d- \hat \mu \|_d  \\
0 & \mathrm{if} & \| \mathbf{X}_d - \hat \theta \|_d > \| \mathbf{X}_d- \hat \mu \|_d 
\end{array}
\right.  ,
\label{eq:classif}
\end{equation}
where $\| x \|_d = \sqrt{\sum_{j=1}^d x_j^2}$ denotes the Euclidean norm in $\mathbb{R}^d$; we also write $\langle x, y \rangle_d = \sum_{j=1}^d x_j y_j$ for the associated inner product. 

\paragraphnospace{Reinterpretation as a plug-in classifier}
We now explain why $\hat\Phi_d$ can be reinterpreted as a plug-in classifier in a truncated space. Recall the expression~\eqref{eq:eta} for the regression function $\eta$.
It is thus natural to consider the 'truncated' regression function $\eta_d$ by replacing $f$ and $g$ with their projections $\Pi_d(f) = \sum_{j=1}^d \theta_j \varphi_j$ and $\Pi_d(g) = \sum_{j=1}^d \mu_j \varphi_j$, \ie,
\begin{align}
\eta_d(X) & := \frac{\exp\left(\int_{0}^1 (\Pi_d(f)(s)-\Pi_d(g)(s)) dX_s - \frac{1}{2}\|\Pi_d(f)\|^2 + \frac{1}{2}\|\Pi_d(g)\|^2\right)}{1+\exp\left(\int_{0}^1 (\Pi_d(f)(s)-\Pi_d(g)(s)) dX_s - \frac{1}{2}\|\Pi_d(f)\|^2 + \frac{1}{2}\|\Pi_d(g)\|^2\right)} \label{eq:eta_d} \\
& = \frac{\exp\left( \langle \bftheta_d - \bfmu_d , \bfX_d \rangle_d -\frac{1}{2}\|\bftheta_d\|_d^2 + \frac{1}{2}\|\bfmu_d\|_d^2 \right)}{1+\exp\left( \langle \bftheta_d - \bfmu_d, \bfX_d \rangle_d -\frac{1}{2}\|\bftheta_d\|_d^2 + \frac{1}{2}\|\bfmu_d\|_d^2 \right)} \nonumber \,,
\end{align}
where $\bftheta_d := (\theta_j)_{1 \leq j \leq d}$, and  $\bfmu_d := (\mu_j)_{1 \leq j \leq d}$. We also define the associated oracle classifier (that still depends on the unknown functions $f$ and $g$):
\begin{equation}\label{eq:bayes_rule_d}
\Phi_d^{\star}(X) := 
\mathds{1}_{\left\lbrace \eta_d(X)\geq\frac12 \right\rbrace} = \mathds{1}_{\left\lbrace 
\int_{0}^1 (\Pi_d(f)(s)-\Pi_d(g)(s)) dX_s \geq \frac{1}{2}\|\Pi_d(f)\|^2 - \frac{1}{2}\|\Pi_d(g)\|^2 \right\rbrace} \,.
\end{equation}
As shown in Remark~\ref{rmk:truncatedSpace} (see Section~\ref{sec:estimationerror} below), $\eta_d$ and $\Phi_d^{\star}$ correspond to the regression function and the Bayes classifier of the restricted classification problem where the learner has only access to the projected input $\bfX_d \in \R^d$, rather than the whole trajectory $X$. \\

We are now ready to reinterpret $\hat\Phi_d$ as a plug-in classifier. Note that
\begin{eqnarray*}
\| \mathbf{X}_d - \hat \theta \|_d \leq \| \mathbf{X}_d- \hat \mu \|_d, 
& \Longleftrightarrow & \frac{\exp\left( \frac{1}{2} \lbrace\| \mathbf{X}_d - \hat \mu \|_d^2 - \| \mathbf{X}_d- \hat \theta \|_d^2 \rbrace \right)}{1+\exp\left( \frac{1}{2} \lbrace\| \mathbf{X}_d - \hat \mu \|_d^2 - \| \mathbf{X}_d- \hat \theta \|_d^2 \rbrace\right)} \geq  \frac{1}{2} \;, \\
& \Longleftrightarrow & \hat\eta_d(X)  \geq  \frac{1}{2} \;,
\end{eqnarray*}
where the estimated regression function $\hat{\eta}_d$ is defined by
\begin{equation}\label{eq:etadestim}
\hat\eta_d(X):= \frac{\exp\left( \langle \hat\theta - \hat\mu , \bfX_d \rangle_d -\frac{1}{2}\|\hat\theta\|_d^2 + \frac{1}{2}\|\hat\mu\|_d^2 \right)}{1+\exp\left( \langle \hat\theta - \hat\mu, \bfX_d \rangle_d -\frac{1}{2}\|\hat\theta\|_d^2 + \frac{1}{2}\|\hat\mu\|_d^2 \right)} \;.
\end{equation} 
In other words, our classifier $\hat\Phi_d$ can be rewritten as $\hat\Phi_d(X) = \mathds{1}_{\lbrace \hat\eta_d(X) \geq 1/2  \rbrace}$ where $\hat\eta_d$ is an estimator of the 'truncated' regression function $\eta_d$ introduced in~\eqref{eq:eta_d}.

\paragraphnospace{Proof strategy.} In the next sections we upper bound the excess risk of $\hat\Phi_d$. We use the following classical decomposition (all quantities below are defined in Section~\ref{sec:intro}, \eqref{eq:classif}, and~\eqref{eq:bayes_rule_d}):
\[
\mathcal{R}_{f,g}(\hat{\Phi}_{d}) - \mathcal{R}_{f,g}(\Phi^\star) = \underbrace{\mathcal{R}_{f,g}(\hat{\Phi}_{d}) - \mathcal{R}_{f,g}(\Phi_d^{\star})}_{\textrm{estimation error}} +
\underbrace{\mathcal{R}_{f,g}(\Phi_d^{\star}) - \mathcal{R}_{f,g}(\Phi^\star)}_{\textrm{approximation error}} \,.
\]
The first term of the right-hand side (estimation error) measures how close $\hat{\Phi}_{d}$ is to the oracle $\Phi_d^{\star}$ in the trunctated space; we analyse it in Section~\ref{sec:estimationerror} below. The second term (approximation error) quantifies the statistical loss induced by the $d$-dimensional projection; we study it in Section~\ref{sec:trunc}.

\subsection{Approximation error}\label{sec:trunc}

We first upper bound the approximation error $\mathcal{R}_{f,g}(\Phi_d^{\star}) - \mathcal{R}_{f,g}(\Phi^\star)$, where the two oracle classifiers  $\Phi_d^{\star}$ and $\Phi^\star$ are defined by~\eqref{eq:bayes_rule_d} and~\eqref{eq:bayes_rule} respectively. Comparing the definitions of $\eta$ and $\eta_d$ in~\eqref{eq:eta} and~\eqref{eq:eta_d}, we can expect that, for $d$ large enough, $\Pi_d(f) \approx f$ and $\Pi_d(g) \approx g$, so that $\eta_d(X) \approx \eta(X)$ and therefore  $\mathcal{R}_{f,g}(\Phi_d^\star) \approx \mathcal{R}_{f,g}(\Phi^\star)$.

 Lemma~\ref{prop:CO} below quantifies this approximation. The proof is postponed to Appendix~\ref{s:pCO}. We recall that, for notational convenience, we write $f_d=\Pi_d(f)$ and $g_d=\Pi_d(g)$.

\begin{lem}
\label{prop:CO}
Let $X$ be distributed according to Model~(\ref{eq:model}), and recall that $\Delta := \|f-g\|$. Let $0<\epsilon \leq 1/8$ and $d\in \mathbb{N}^\star$ such that
\begin{equation}
\max \left( \| f-f_d\|^2,\| g-g_d \|^2   \right) \leq \frac{\epsilon^2}{512 \ln(1/\epsilon^2)}\,.
\label{eq:lemma1conditioneps}
\end{equation}
Then, the two oracle classifiers $\Phi_d^\star$ and $\Phi^\star$ defined by~\eqref{eq:bayes_rule_d} and~\eqref{eq:bayes_rule} satisfy
$$ \mathcal{R}_{f,g}(\Phi_d^\star) - \mathcal{R}_{f,g}(\Phi^\star) \leq 12 \epsilon^2 + 2 \epsilon \left( 1 \wedge \frac{10\epsilon}{\Delta} \right) \,. $$
\end{lem}

We stress that the distance $\Delta$ between $f$ and $g$ has a strong influence on the approximation error. In particular, if $\Delta$ is bounded from below independently from $n$, then the approximation error is at most of the order of $\epsilon^2$, while it can only be controlled by $\epsilon$ if $\Delta \lesssim \epsilon$. This key role of $\Delta$ is a consequence of the margin behavior analyzed in Proposition \ref{thm:margin} (Section~\ref{s:margin}) and will also appear in the estimation error.

\paragraphnospace{A smoothness assumption.} When $d \in \N^*$ is fixed, we can minimize the bound of Lemma~\ref{prop:CO} in~$\epsilon$. Unsurprisingly the resulting bound involves the distances $ \| f-f_d\|$ and $\| g-g_d \| $ of $f$ and $g$ to their projections $f_d$ and $g_d$. In the sequel, we assume that the functions $f$ and $g$ are smooth in that their (Fourier) coefficients w.r.t.~the basis $(\varphi_j)_{j \geq 1}$ decay sufficiently fast. More precisely, we assume that, for some parameters $s,R>0$, the functions $f$ and $g$ belong to the set
\begin{equation}\label{def:Hs}
\Hs(R) := \left\{h \in \LL([0,1]) : \; \sum_{j=1}^{+\infty} c_j(h)^2 j^{2s} \leq R^2 \right\} \,.
\end{equation}
The set $\Hs(R)$ corresponds to a class of smooth functions with smoothness parameter $s$: when $s=0$, we simply obtain  the $\LL([0,1])$-ball of radius $R$. For larger $s$, for example $s=1$, we obtain a smaller Sobolev space of functions such that $f' \in \LL([0,1])$ with $\|f'\|_2 \leq R$.

\ \\
Under the above assumption on the tail of the spectrum of $f$ and $g$, the loss of accuracy induced by the projection step is easy to quantify. Indeed, for all $f \in \Hs(R)$ we have
	\[
	\|f-f_{d}\|^2 = \sum_{j=d+1}^{+\infty} c_j(f)^2 \leq d^{-2s} \sum_{j=d+1}^{+\infty} c_j(f)^2 j^{2s} \leq R^2 d^{-2s} \,,
	\]
so that, omitting logarithmic factors, $\epsilon$ can be chosen of the order of $R d^{-s}$ in the statement of Lemma~\ref{prop:CO}.

\subsection{Estimation error}
\label{sec:estimationerror}

We now upper bound the estimation error $\mathcal{R}_{f,g}(\hat{\Phi}_{d}) - \mathcal{R}_{f,g}(\Phi_d^{\star})$ of our classifier $\hat\Phi_d$. To that end, we first reinterpret $\eta_d$ and $\Phi_d^{\star}$; this will be useful to rewrite the estimation error as an excess risk (as in~\eqref{eq:excess_risk}) in the truncated space. The next remark follows from direct calculations.

\begin{rem}
\label{rmk:truncatedSpace}
Denote by $\bfX_d := \bigl(\langle\varphi_j,X\rangle\bigr)_{1 \leq j \leq d}$, $\bftheta_d = (\theta_j)_{1 \leq j \leq d}$, and  $\bfmu_d = (\mu_j)_{1 \leq j \leq d}$ the versions of $X$, $\theta$, and $\mu$ in the truncated space. Then,
\[
\eta_d(X) = \frac{\frac{1}{2} q_{f_d}(X)}{\frac{1}{2} q_{f_d}(X) + \frac{1}{2} q_{g_d}(X)} \qquad \textrm{and} \qquad q_{f_d}(X) = e^{\frac{1}{2} \|\bfX_d\|^2} e^{-\frac{1}{2}\|\bfX_d-\bftheta_d\|^2} \;.
\]
Since the conditional distribution of $\bfX_d$ is $\cN(\bftheta_d,I_d)$ given $Y=1$ and $\cN(\bfmu_d,I_d)$ given $Y=0$, this entails that $\eta_d(X) = \PR(Y=1|\bfX_d)$ almost surely.

In other words, $\eta_d$ is the regression function of the restricted classification problem where the learner has only access to the projected trajectory $\bfX_d \in \R^d$, instead of the whole trajectory $X$. The function $\Phi^*_d = \mathds{1}_{\{\eta_d \geq 1/2\}}$ is the associated Bayes classifier.
\end{rem}

\ \\
We are now ready to compare the risk of our classifier $\hat\Phi_d$ to that of the $d$-dimensional oracle $\Phi_d^\star$. The proof of the next lemma is postponed to Appendix~\ref{s:pproj}. (The value of $4608$ could most probably be improved.) We recall that $f_d=\Pi_d(f)$ and $g_d=\Pi_d(g)$.

\begin{lem}
\label{lem:proj}
We consider Model~(\ref{eq:model}). Let $d\in \mathbb{N}^*$ and set $\Delta_d := \| f_d - g_d \|$. Let $0 < \epsilon \leq 1/8$ and $n \geq 27$ such that
\begin{equation}
\left(\Delta_d + 2 \sqrt{\frac{d\log(n)}{n}} \, \right) \sqrt{\frac{d\log(n)}{n}} \leq \frac{\epsilon}{48} \;.
\label{eq:lemma2conditioneps}
\end{equation}
Then, the classifiers $\hat \Phi_d$ and $\Phi_d^\star$ defined by~\eqref{eq:classif} and~\eqref{eq:bayes_rule_d} satisfy
\[
 \mathcal{R}_{f,g}(\hat\Phi_d) - \mathcal{R}_{f,g}(\Phi_d^\star) \leq 2 \epsilon \left( 1\wedge \frac{10 \epsilon}{\Delta_d}  \right) + 6 \exp \left(  - \frac{n \epsilon^2}{4608 \, d \log n } \right) + \frac{13}{n} \;.
\]
\end{lem}

In the same vein as for the approximation error, the estimation error bound above strongly depends on the distance between the two functions $f_d$ and $g_d$ of interest. This is again a consequence of the  margin behavior analyzed in Proposition \ref{thm:margin} (Section~\ref{s:margin}).

More precisely, when $\epsilon$ is chosen at least of the order of $\sqrt{d/n} \log(n)$ (in order to kill the exponential term), the estimation error bound above is roughly of the order of $\min\bigl\{\epsilon, \epsilon^2/\Delta_d\bigr\}$. In particular, if $\Delta_d$ is bounded from below, then the estimation error is at most of the order of $\epsilon^2 \approx d \log^2(n)/n$. On the other hand, if no lower bound is available for $\Delta_d$, then the only estimation error bound we get is a slower rate of the order of $\epsilon \approx \sqrt{d/n} \log(n)$.

\subsection{Convergence rate under a smoothness assumption}
\label{subsection:upper_bound}

We now state the main result of this paper. We upper bound the excess risk $\mathcal{R}_{f,g}(\hat\Phi_d) - \mathcal{R}_{f,g}(\Phi^\star)$ of our classifier when $f$ and $g$ belong to subsets of the Sobolev ball $\Hs(R)$ defined in~\eqref{def:Hs}. These subsets are parametrized by a separation distance $\Delta$: a larger value of $\Delta$ makes the classification problem easier, as reflected by the non-increasing bound below.

\begin{thm}
\label{prop:minmax}
There exist an absolute constant $c>0$ and a constant $N_{s,R} \geq 86$ depending only on $s$ and $R$ such that the following holds true. For all $s,R > 0$ and all $n \geq N_{s,R}$, the classifier $\hat \Phi_{d_n}$ defined by~\eqref{eq:classif} with $d_n = \lfloor (R^2 n)^{\frac{1}{2s+1}} \rfloor $ satisfies
\[
\sup_{\substack{f,g \in \Hs(R) \\ \|f-g\| \geq \Delta}} \left\{ \mathcal{R}_{f,g}(\hat\Phi_{d_n}) - \inf_{\Phi} \mathcal{R}_{f,g}(\Phi)  \right\}  \leq \left\{\begin{array}{ll}
	c \, R^{\frac{1}{2s+1}} n^{-\frac{s}{2s+1}} \log(n) & \textrm{if } \Delta < R^{\frac{1}{2s+1}} n^{-\frac{s}{2s+1}} \log(n) \\[0.2em]
	\displaystyle \frac{c}{\Delta} R^{\frac{2}{2s+1}} n^{-\frac{2s}{2s+1}} \log^2(n) & \textrm{if } \Delta \geq R^{\frac{1}{2s+1}} n^{-\frac{s}{2s+1}} \log(n)
\end{array}\right.
\]
\end{thm}

\ \\
The proof is postponed to Appendix~\ref{s:pminmax} and combines Lemmas~\ref{prop:CO} and~\ref{lem:proj} from the previous sections. A proof sketch is also provided below.

Note that the two bounds of the right-hand side coincide when $\Delta = R^{1/(2s+1)} \, n^{-s/(2s+1)} \log(n)$. Therefore, there is a continuous transition from a slow rate (when $\Delta$ is small) to a fast rate (when $\Delta$ is large). This leads to the following remark.

\begin{rem}[Novelty of the bound]
\ \\[-0.4cm]
\begin{itemize}
	\item Taking $\Delta=0$, we recover the worst-case bound of \cite[Corollary~4.4]{Cadre} (where $u=1/s$) up to logarithmic factors. As shown by Theorem~\ref{thm:minimaxlowerbound} below, this slow rate is unimprovable for a small distance~$\| f- g \|$.
\item However, we obtain a significantly faster rate when $\| f- g \|$ is bounded from below. This improved rate is a consequence of the margin behavior (see, \eg, \cite{AT07,GKM16}), but not of the choice of $d_n$ that is oblivious to $\| f- g \|$.
	\item Continuous transitions from slow rates to faster rates were already derived in the past. For instance, for any supervised classification problem where the margin $\big|2\eta(X)-1\big| \geq h$ is almost surely bounded from below, \cite[Corollary~3]{MaNe-06-RiskBounds} showed that the excess risk w.r.t.~a class of VC-dimension $V$ varies continuously from $\sqrt{V/n}$ to $V/(nh)$ (omitting log factors) as a function of the margin parameter~$h$. In a completely different setting,  \cite[Theorem~5]{RaSrTs-13-MinimaxRegretRisk} analyzed the minimax excess risk for nonparametric regression with well-specified and misspecified models. They showed a continuous transition from slow to faster rates when the distance of the regression function to the statistical model decreases to zero.
\end{itemize}
\end{rem}

We finally note that, though the choice of the parameter $d_n$ does not depend on $\Delta$, it still depends on the (possibly unknown) smoothness parameter $s$. Though designing an adaptive classifier is beyond the scope of this paper, it might be addressed via the Lepski method (see, \eg, \cite{L90}) after adapting it to the classification setting.

\ \\
\textit{Sketch of the proof.} We outline the main ingredients. For all $d \in \N^*$ and $0 < \epsilon_1,\epsilon_2 \leq 1/8$ satisfying the assumptions of Lemmas~\ref{lem:proj} and~\ref{prop:CO},
\begin{eqnarray}
\mathcal{R}_{f,g}(\hat\Phi_d) - \mathcal{R}_{f,g}(\Phi^\star) 
& = & \mathcal{R}_{f,g}(\hat\Phi_d) - \mathcal{R}_{f,g}(\Phi_d^\star) + \mathcal{R}_{f,g}(\Phi_d^\star) - \mathcal{R}_{f,g}(\Phi^\star) \nonumber \\
& \lesssim &\epsilon_1 \left( 1\wedge \frac{\epsilon_1}{\| f_d -g_d \|}  \right) + \exp \left(  - \frac{n \epsilon_1^2}{4608 \, d \log n } \right)  + \frac{1}{n} + \epsilon_2^2 + 
\epsilon_2  \left( 1\wedge \frac{\epsilon_2}{\| f -g \|}  \right) \nonumber \\
& \lesssim & \min\left\{\epsilon_1+\epsilon_2, \frac{\epsilon_1^2 + \epsilon_2^2}{\| f-g \|} \right\} + \exp \left(  - \frac{n \epsilon_1^2}{4608 \, d \log n } \right) + \frac{1}{n} + \epsilon_2^2 \nonumber \\
& \lesssim & \min\left\{\sqrt{\frac{d}{n}} \log(n)+\epsilon_2, \frac{d \log^2(n)/n + \epsilon_2^2}{\| f-g \|} \right\} + \frac{1}{n} + \epsilon_2^2 \nonumber 
\end{eqnarray}
for $d$ large enough, and for the choice of $\epsilon_1 \approx \sqrt{d/n} \log(n)$. Following the comments at the end of Section~\ref{sec:trunc}, we also choose $\epsilon_2 \approx R d^{-s}$ (up to logarithmic factors). We obtain, omitting constant but also logarithmic factors:
\begin{align*}
\mathcal{R}_{f,g}(\hat\Phi_d) - \mathcal{R}_{f,g}(\Phi^\star)  & \lesssim \min\left\{\sqrt{\frac{d}{n}} + d^{-s} , \frac{d/n + d^{-2s}}{\| f-g \|} \right\} \lesssim \min\left\{n^{-\frac{s}{2s+1}} , \frac{n^{-\frac{2s}{2s+1}}}{\| f-g \|} \right\} \;,
\end{align*}
where the last inequality is obtained with the optimal value of $d \approx n^{\frac{1}{2s+1}}$. A detailed and more formal proof (with, \eg, a rigorous treatment of the approximation $\| f_d -g_d \| \approx \| f -g \|$) can be found in Appendix~\ref{s:pminmax}. \hfill $\square$


\section{Lower bounds on the excess risk}
\label{s:lower}
\label{sec:minimaxlowerbounds}

In this section we derive two types of excess risk lower bounds.

The first one decays polynomially with $n$ and applies to \textit{any} classifier. This minimax lower bound indicates that, up to logarithmic factors, the excess risk of Theorem~\ref{prop:minmax} cannot be improved in the worst case. This result is derived via standard nonparametric statistical tools (\eg, Fano's inequality) and is stated in Section~\ref{s:minimaxfano}.

Our second lower bound is of a different nature: it decays logarithmically with $n$ and only applies to the nonparametric $k$-nearest neighbors algorithm evaluated on projected trajectories $\bfX_{i,d} \in \R^d$ and $\bfX_d \in \R^d$. We allow $d$ to be chosen adaptively via a sample-splitting strategy, and we consider $k$ tuned (optimally) as a function of $d$. Our logarithmic lower bound indicates that this popular algorithm is not fit for our particular model; see Section~\ref{s:knn} below.  

\subsection{A general minimax lower bound}\label{s:minimaxfano}

We provide a lower bound showing that the excess risk bound of Theorem~\ref{prop:minmax} is minimax optimal up to logarithmic factors. The proof is postponed to Appendix~\ref{s:minimaxlowerbound}.

\begin{thm}
\label{thm:minimaxlowerbound}
Consider the statistical model~\eqref{eq:model} and the set $\cH_s(R)$ defined in~\eqref{def:Hs}, where $s,R>0$ and where $(\varphi_j)_{j \leq 1}$ is any Hilbert basis of $\LL([0,1])$. Then, every classifier $\hat{\Phi}$ satisfies, for any number $n \geq \max\bigl\{R^{1/s}, (32 \log(2) + 2)^{2s+1}/(3R^2/4)\bigr\}$ of i.i.d.~observations $(X_i,Y_i)_{1 \leq i \leq n}$ from~\eqref{eq:model} and all $\Delta \in (0,R/2]$,
\[
\sup_{\substack{f,g \in \Hs(R) \\ \|f-g\| \geq \Delta}} \left\{ \mathcal{R}_{f,g}(\hat\Phi) - \inf_{\Phi} \mathcal{R}_{f,g}(\Phi)  \right\} \geq \left\{ \begin{array}{ll}
\displaystyle c e^{-2 R^{2/(2s+1)}} \, R^{1/(2s+1)} \, n^{-s/(2s+1)} & \textrm{if $\Delta < R^{1/(2s+1)} \, n^{-s/(2s+1)}$} \\[0.5cm]
\displaystyle \frac{c e^{-2 \Delta^2}}{\Delta} \, R^{2/(2s+1)} \, n^{-2s/(2s+1)}  & \textrm{if $\Delta \geq R^{1/(2s+1)}\, n^{-s/(2s+1)}$}
\end{array}\right.
\]
for some absolute constant $c>0$.
\end{thm}

We note two minor differences between the upper and lower bounds: Theorem~\ref{prop:minmax} involves extra logarithmic factors, while Theorem~\ref{thm:minimaxlowerbound} involves an extra term of $e^{-2\Delta^2}$. Fortunately both terms have a minor influence (note that $e^{-2\Delta^2} \geq e^{-8R^2}$ since $f,g \in \Hs(R)$). We leave the question of identifying the exact rate for future work. \footnote{Possible solutions include: slightly improving Proposition~\ref{thm:margin} via a tighter Gaussian concentration bound (to gain a factor of nearly $e^{-\Delta^2/8}$), and optimizing the constant appearing in the exponential term of Lemma~\ref{lem:GaussianMassCone} (Appendix~\ref{lem:GaussianMassCone}).}

\ \\
If we omit logarithmic factors and constant factors depending only on $s$ and $R$, Theorems~\ref{prop:minmax} and~\ref{thm:minimaxlowerbound} together imply that, for $n \geq N_{s,R}$ large enough:
\begin{itemize}
	\item when $\Delta \lesssim R^{1/(2s+1)} \, n^{-s/(2s+1)}$, the optimal worst-case excess risk is of the order of $n^{-s/(2s+1)}$;
	\item when $\Delta \gtrsim R^{1/(2s+1)} \, n^{-s/(2s+1)}$, the optimal worst-case excess risk is of the order of $n^{-2s/(2s+1)}/\Delta$.
\end{itemize}

\subsection{Lower bound for the $k$-NN classifier}
\label{s:knn}

In this section, we focus on the $k$-nearest neighbor (kNN) classifier. This classification rule has been intensively studied over the past fifty years. In particular, this method provides interesting theoretical and practical properties. It is quite easy to handle and implement. Indeed, given a sample $\mathcal{S} = \lbrace (X_1,Y_1),\dots, (X_n,Y_n) \rbrace$, a number of neighbors $k$, a norm $\|.\|$ and a new incoming observation, the kNN classifier is defined as 
\begin{equation}
\Phi_{n,k}(X) = \mathds{1}_{\lbrace \frac{1}{k} \sum_{j=1}^k Y_{(j)}(X)>1/2  \rbrace},
\label{eq:knn}
\end{equation}
where the $Y_{(j)}$ correspond to the label of the $X_{(j)}$ re-arranged according to the ordering
 $$\| X_{(1)} - X \| \leq ... \leq \| X_{(n)} - X \|.$$ We refer the reader, \eg, to \cite{G78}, \cite{DGL96} or \cite{BD2015} for more details. \\

We are interested below in the performances of the classifier $ \Phi_{n,k}$ in this functionnal setting. For this purpose, we will use the recent contribution of \cite{chaudhuri} that provides a lower bound of the misclassification rate of the kNN classifier in a very general framework. This lower bound is expressed as the measure of an uncertain set around $\eta \simeq 1/2$. We emphasize that we want to understand if a truncation strategy associated to a non parametric supervised classification approach is suitable for this kind of problem.

\subsubsection{Finite-dimensional case}

\paragraph{Smoothness parameter $\alpha$}

We shall consider first a finite $d$-dimensional case for our Gaussian translation model. In that case, Remark~\ref{rmk:truncatedSpace} in Section~\ref{sec:estimationerror} reveals that the truncation approach problem we are studying is, without loss of generality, equivalent to a supervised classification in $\mathbb{R}^d$ where conditionally on the event $\lbrace Y=0\rbrace$ (resp. $\lbrace Y=1 \rbrace$), $X_d$ is a standard Gaussian variable (resp. a Gaussian random variable with mean $m$ and variance $1$). 
If $\gamma_d$ refers to the Gaussian density:
$$
\forall x \in \mathbb{R}^d \qquad \gamma_d(x) := (2\pi)^{-d/2} e^{-\|x\|^2/2},
$$
then in that case, the Bayes classifier in $\mathbb{R}^d$ is:
$$
\Phi_d^\star(x) = \mathds{1}_{\lbrace \eta_d(x) \geq 1/2 \rbrace } \qquad \mathrm{with} \qquad \eta_d(x) = \frac{\gamma_d(x)}{\gamma_d(x)+\gamma_d(x-m)} \quad \forall x\in \mathbb{R}^d,
$$

In the following, to simplify the notations, we will drop the subscript $d$ in all these terms and will write $\gamma, \eta$ instead of $\gamma_d,\eta_d$. Following \cite{chaudhuri}, the rate of convergence of the kNN depends on a smoothness parameter $\alpha$ involved in the next inequality:
\begin{equation}\label{eq:smooth}
\forall x \in \mathbb{R}^d \qquad |\eta(B(x,r)) - \eta(x)| \leq L \mu(B(x,r))^{\beta}
\end{equation}
where $\eta(B(x,r))$ refers to the mean value of $\eta$ on $B(x,r)$ w.r.t. the distribution of the design $X$ given by $\mu=\frac{1}{2} \gamma(.)+\frac{1}{2} \gamma(.-m)$.
Therefore, our first task is to determine the value of $\beta$ in our Gaussian translation model. We begin with a simple proposition that entails that the value of $\beta$ corresponding to our Gaussian translation model in $\mathbb{R}^d$ is $2/d$. The proof of Proposition \ref{prop:gaussian_alpha}  is postponed to Appendix \ref{proof:gaussian_alpha}. 

\begin{prop}
\label{prop:gaussian_alpha} 
Assume that $\|x\| \leq R$ for some $R \in \mathbb{R}^+$. Then an explicit constant $L_R$ exists such that
$$
\forall r \leq \frac{1}{R} \qquad 
|\eta(B(x,r))-\eta(x)| \leq L_R \mu(B(x,r))^{2/d}.
$$
\end{prop}

An important point given in the previous proposition is that when we are considering design points $x$ such that $\|x\| \leq R/2$ and $\|m\| \leq R/2$, we then have
$$
\forall r \leq 1 \quad \forall x \in B(0,R/2) \qquad 
\left| \eta(B(x,r)) - \eta(x)\right|\leq  60 \pi e d R^2 e^{R^2/d} \mu(B(x,r))^{2/d},
$$
so that the constant $L_R$ involved in the statement of Proposition \ref{prop:gaussian_alpha} can be chosen as:
\begin{equation}\label{eq:value_LR}
L_R =  60 \pi e d R^2 e^{R^2/d}
\end{equation}

According to inequality (\ref{eq:smooth}) and thanks to Proposition \ref{prop:gaussian_alpha}, the smoothness of the Gaussian translation model is given by:
$$\beta_d=2/d.$$
 Now, we slightly modify the approach of \cite{chaudhuri} 
to obtain a lower bound on the excess risk that involves the margin of the classification problem. As pointed above, in the Gaussian translation model, when the two classes are well separated (meaning that the center of the two classes are separated with a distance independent on $n$), the margin parameter is equal to $1$ (see Theorem \ref{thm:margin}).\\

\paragraph{Optimal calibration of the kNN}

Before giving our first result on the rate of convergence of the kNN classifier, we remind first some important facts regarding the choice of the number of neighbors $k$ for the kNN classifier. The ability of the kNN to produce a universally consistent classification rule highly depends on the choice of the bandwidth parameter $k_n$. In particular, this bandwidth parameter must satisfy $k_n \longrightarrow + \infty$ and $k_n/n \longrightarrow 0$ as $n \longrightarrow + \infty$   to  produce an asymptotically vanishing variance and bias (see, \eg,  \cite{DGL96} for details). However, to obtain an optimal rate of convergence, $k_n$ has to be chosen to produce a nice trade-off between the bias and the variance of the excess risk. It is shown in \cite{chaudhuri} that, when the marginal law of $X$ is compactly supported, the optimal calibration $k^{opt}_n$ is:
\begin{equation}\label{eq:optimal_k}
\frac{1}{\sqrt{k^{opt}_n}(d)} = c \left( \frac{k^{opt}_n(d)}{n}\right)^{\beta_d} \quad \Leftrightarrow \quad k^{opt}_n(d) \sim n^{\frac{4}{4+d}}
\end{equation}
where $c$ refers to any non negative constant and $\beta_d = 2/d$ refers to the smoothness parameter of the model involved in Inequality \eqref{eq:smooth}. On the other hand, when $\beta_d=1/d$, it is shown in \cite{GKM16} that (almost) optimal rates of convergence can be obtained in the non-compact case, choosing for instance
$$ k_n \sim n^{\frac{2}{2+ d + \tau}},$$
for some positive $\tau$.  The following results provides a lower bound on the convergence rate with a number of neighbor $k$ contained in a range of values .

\begin{prop}
\label{theo:knngauss} 
For any $d\in \mathbb{N}$, a constant $C_1$ exists such that
$$
\mathcal{R}_{f,g}(\Phi_{k,n,d})-\mathcal{R}(\Phi_d^{\star}) \geq \frac{C_1}{k_n}
$$
when $k \in  \mathcal{K}_n$ where 
$$ \mathcal{K}_n = \left\lbrace \ell \in \mathbb{N}\ s.t. \ \frac{1}{\sqrt{\ell}} \geq d \left( \frac{\ell}{n}  \right)^{2/d} \quad \text{and} \quad \ell \leq n \right\rbrace.$$
\end{prop}

The proof of this result is given in Appendix \ref{proof:knngauss}. 

\begin{rem} Proposition \ref{theo:knngauss} is an important intermediary result to understand the behaviour of kNN with functional data. We briefly comment on this result below.
\begin{itemize}
\item
The set $\mathcal{K}_n$ contains all the integers from $1$ to an integer equivalent to $n^{4/(4+d)} d^{-2d/(4+d)}$. In particular, the ``optimal" standard calibration of $k_n$ given by Equation \eqref{eq:optimal_k} is included in the set $\mathcal{K}_n$ and Proposition \ref{theo:knngauss} applies in particular for such a calibration.
\item
Proposition \ref{theo:knngauss} entails that tuning the kNN classifier in an ``optimal way"  cannot produce faster rates of convergence than $n^{-4/(d+4)}$, even with some additional informations on the considered model (here the Gaussian distribution of the conditional distributions):
$$
\mathcal{R}_{f,g}(\Phi_{k^{opt}_n(d),n,d})-\mathcal{R}(\Phi_d^{\star}) \geq C_1 n^{-\frac{4}{d+4}}.
$$
 These performances have to be compared to those obtained with our procedure that explicitly exploits the additional knowledge of Gaussian conditional distributions (see, \eg, Lemma \ref{lem:proj}). 

\item The last important point is that the lower bound in the statement of Proposition \ref{theo:knngauss} appears to be seriously damaged when $d$ increases. This is a classical feature of the curse of dimensionality. For us, it invalidates any approach that will jointly associate a truncation strategy with a kNN plug-in classifier:  we will be led to choose $d$ large with $n$ to avoid too much loss of information but in the same time this will harms the statistical misclassification.
\end{itemize}
\end{rem}

\subsubsection{Lower bound of the misclassification rate with truncated strategies}

As pointed by Proposition \ref{theo:knngauss}, the global behavior of the kNN classifier heavily depends on the choice of the dimension $d$. In the same time, the size of $d$ is important to obtain a truncated Bayes classifier $\Phi_d^{\star}$ close to the Bayes classifier $\Phi^{\star}$. To assess the performance of kNN, we consider a sample splitting strategy $\mathcal{S}=\mathcal{S}_1 \cup \mathcal{S}_2$ where $(\mathcal{S}_1,\mathcal{S}_2)$ is a partition of $\mathcal{S}$. Then, $\mathcal{S}_1$ is used to choose a dimension $\hat{d}$, then we apply an optimal kNN classifier method based on the samples of $\mathcal{S}_2$ on the truncated spaces with $\Pi_{\hat{d}}$ with $k_n^{opt}(\hat{d})$ chosen as in Equation \eqref{eq:optimal_k}. It is important to note that the sample splitting strategy produces a choice $\hat{d}$ independent on the samples in $\mathcal{S}_2$. 

Theorem \ref{thm:lwr_bound_knn} below shows that any sample splitting strategy, every choice of $\hat{d}$ will lead to bad performances of classification on model \eqref{eq:model}. 
The proof is postponed to Appendix~\ref{proof:lwr_bound_knn}.
\begin{thm}
\label{thm:lwr_bound_knn}
In the functional Gaussian translation model, any kNN classifier $\Phi_{k_n^{opt},n,\hat{d}}$ satisfies
$$
\inf_{\hat{d}\in \mathbb{N}}\sup_{f,g \in \Hs(r)} \mathcal{R}_{f,g}(\Phi_{k_n^{opt}(\hat{d}),n,\hat{d}})-R_n(\Phi^{\star})  \gtrsim \log(n)^{-2s}.
$$
\end{thm}

The main conclusion of this section and of Theorem \ref{thm:lwr_bound_knn} is that the kNN rule based on a truncation strategy does not lead to satisfying rates of convergence, regardless the choice of the dimension $\hat{d}$ is. We stress that this result is only valid for a specific choice of $k_n^{opt}(\hat{d})$. Although this choice appears to be classic regarding the existing literature, obtaining a global lower bound (\ie, for any choice of $k$) remains an open (and difficult) problem. Even though we suspect that such a logarithmic lower bound also holds for some more general procedures (without sample splitting and with a more general possible choice of $k_n$), we do not have any proof of such a result.
However, it should be kept in mind that the misclassification of the classifier $\hat \Phi_{d_n}$ proposed in Equation \eqref{eq:classif} attains a polynomial rate of convergence, which is an important encouragement  for its use instead of the use of a nonparametric classifier associated with a threshold rule.

\appendix

\section{Proof of the upper bounds}\label{s:A}
The goal of this section is to prove the polynomial upper bound of Theorem~\ref{prop:minmax} together with the intermediate results of Proposition \ref{thm:margin} and Lemma~\ref{lem:proj}. We will pay a specific attention to the acceleration (in terms of the number $n$ of samples) obtained when the functions $f$ and $g$ appearing in \eqref{eq:model} are well separated.

\subsection{Proof of Proposition \ref{thm:margin} (control of the margin)}
\label{s:pmargin}
We start by proving Proposition \ref{thm:margin}, \ie, we analyze the margin behavior in Model \eqref{eq:model}. This result is a key ingredient to derive our excess risk upper bounds.

\begin{proof}[Proof of Proposition \ref{thm:margin}]

We use the Girsanov Equations (\ref{eq:qf}) and (\ref{eq:qg}) that define the likelihood ratio $q_f$ and $q_g$. We therefore deduce that
\begin{eqnarray}
\lefteqn{ \PR_{X} \left( \left| \eta(X) - \frac{1}{2} \right| \leq \epsilon   \right) }\nonumber \\
 &= &  \PR_{X} \left( \frac{|q_f(X) - q_g(X) |}{2(q_f(X) + q_g(X))} \leq \epsilon \right) \nonumber \\
& = &   \PR_{X} \left( \left\lbrace \frac{|q_f(X) - q_g(X) |}{2(q_f(X) + q_g(X))} \leq \epsilon \right\rbrace \cap \left\lbrace q_f(X) \leq q_g(X) \right\rbrace \right)+\PR_{X} \left( \left\lbrace \frac{|q_f(X) - q_g(X) |}{2(q_f(X) + q_g(X))} \leq \epsilon \right\rbrace \cap \left\lbrace  q_f(X)>q_g(X) \right\rbrace \right) \nonumber \\
  & \leq &   \PR_{X} \left( \left\lbrace \frac{|q_f(X) - q_g(X) |}{4 q_g(X)} \leq \epsilon \right\rbrace \cap \left\lbrace q_f(X) \leq q_g(X) \right\rbrace \right) +  \PR_{X} \left(\left\lbrace \frac{|q_f(X) - q_g(X) |}{ 4 q_f(X)} \leq \epsilon \right\rbrace \cap \left\lbrace q_f(X)>q_g(X) \right\rbrace\right) \nonumber \\ 
  & \leq &  \PR_{X} \left( \left|\frac{q_f(X)}{q_g(X)} -1 \right|  \leq 4 \epsilon  \right) +  \PR_{X} \left( \left|\frac{q_g(X)}{q_f(X)} -1 \right|  \leq 4 \epsilon  \right) \,. \label{eq:LRcontrol}
\end{eqnarray}
The two terms of the last line are handled similarly, and we only deal with the first one. We note that
$$
\frac{q_f(X)}{q_g(X)} = \exp\left(\int_{0}^1 (f-g)(s) dX_s - \frac{1}{2} [\|f\|^2-\|g\|^2] \right) \,.
$$
Using the fact that $Y \sim \mathcal{B}(1/2)$ and conditionning by $Y=1$ and $Y=0$, we can see that
\begin{eqnarray}
\lefteqn{ \PR_{X} \left( \left|\frac{q_f(X)}{q_g(X)} -1 \right|  \leq 4 \epsilon  \right)} \nonumber \\
&= & \PR \left( \left| e^{\int_{0}^1 (f-g)(s) f(s) ds +\int_{0}^1 (f-g)(s)  dW_s   - \frac{1}{2} [\|f\|^2-\|g\|^2]} - 1 \right| \leq 4 \epsilon \right) \PR(Y=1) \nonumber\\
 & & \hspace{2cm}+   \PR\left( \left| e^{\int_{0}^1 (f-g)(s) g(s) ds +\int_{0}^1 (f-g)(s)  dW_s   - \frac{1}{2} [\|f\|^2-\|g\|^2]} - 1 \right| \leq 4 \epsilon \right)\PR(Y=0) \nonumber\\
  & & =  \frac{1}{2}   \PR \left( \left|  e^{\frac{1}{2} \|f-g\|^2 +\int_{0}^1 (f-g)(s)  dW_s} - 1 \right| \leq 4 \epsilon \right) +
  \frac{1}{2}   \PR \left( \left|  e^{-\frac{1}{2} \|f-g\|^2 +\int_{0}^1 (f-g)(s)  dW_s} - 1 \right| \leq 4 \epsilon \right) \nonumber\\
  & & =  \frac{1}{2}   \PR \left( \left|  e^{\frac{1}{2} \Delta^2 + \Delta \xi } - 1 \right| \leq 4 \epsilon \right) +
  \frac{1}{2}   \PR \left( \left|  e^{-\frac{1}{2} \Delta^2 +\Delta \xi} - 1 \right| \leq 4 \epsilon \right) \,, \label{eq:intervalle_gaussienne}
\end{eqnarray}
where $\Delta := \|f-g\|$ and $\xi \sim \mathcal{N}(0,1)$ because $\int_{0}^1 [f(s)-g(s)] dW_s \sim \mathcal{N}(0,\Delta^2)$.\\

Using the inequalities $\ln(1+ 4\epsilon) \leq 4\epsilon$ and $\ln(1-4\epsilon)\geq -8\epsilon$ when $\epsilon \leq 1/8$, the above probability can be upper bounded as
\begin{eqnarray*}
 \PR_{X} \left( \left|\frac{q_f(X)}{q_g(X)} -1 \right|  \leq  4 \epsilon  \right) 
 & \leq  & \frac{1}{2} \PR \left(-\frac{8 \epsilon}{\Delta} - \frac{\Delta}{2} \leq \xi \leq \frac{4 \epsilon}{\Delta} - \frac{\Delta}{2}\right)   
 + \frac{1}{2} \PR \left(-\frac{8 \epsilon}{\Delta} + \frac{\Delta}{2} \leq \xi \leq \frac{4 \epsilon}{\Delta} + \frac{\Delta}{2} \right), \\
 & \leq & \frac{5 \epsilon}{\Delta},
\end{eqnarray*}
where the last inequality follows from $\PR(a \leq \xi \leq b) \leq (b-a)/\sqrt{2\pi}$ and $12/\sqrt{2 \pi}\leq 5$.
Inverting the roles of $f$ and $g$, we get by symmetry of the problem that the second term of~\eqref{eq:LRcontrol} is also upper bounded by $5 \epsilon / \Delta$. This concludes the proof. 
\end{proof}

\begin{rem}\label{rem:pid}
Following the same proof strategy, it is easy to check that the same result hold in the truncated space, \ie, replacing $\eta$ with $\eta_d$ and $\Delta$ with $\Delta_d := \|\Pi_d(f-g)\|$. Namely, for all $d \in \N^*$ and all $ 0 < \epsilon \leq 1/8$,
\[
\PR_{X} \left( \left| \eta_d(X) - \frac{1}{2} \right| \leq \epsilon   \right) \leq 1 \wedge \frac{10 \epsilon}{\Delta_d} \,.
\]
In particular, Equation \eqref{eq:intervalle_gaussienne} holds with $q_{f_d}(X)/q_{g_d}(X)$ on the left-hand side and with $\Delta_d$ on the right-hand side because $\int_{0}^1 \Pi_d(f-g)(s) dW_s \sim \mathcal{N}(0,\Delta_d^2)$.
\end{rem}

\subsection{Proof of Lemma~\ref{prop:CO} (control of the approximation error)}
\label{s:pCO}

One key ingredient of the proof is to control the excess risk $\mathcal{R}_{f,g}(\Phi^{\star}_d) - \mathcal{R}_{f,g}(\Phi^{\star})$ in terms of the closeness of $f_d$ and $g_d$ to $f$ and $g$ respectively. To do so, we set
$$
\delta_d := \|f-f_d\| = \sqrt{\|f\|^2-\|f_d\|^2} \qquad \text{and} \qquad \tilde{\delta}_d := \|g-g_d\| = \sqrt{\|g\|^2-\|g_d\|^2} \;.
$$

\begin{proof}[Proof of Lemma~\ref{prop:CO}]

We start with the well-known formula on the excess risk of any classifier (see, \eg, \cite{G78}):
$$
\mathcal{R}_{f,g}(\Phi^{\star}_d) - \mathcal{R}_{f,g}(\Phi^{\star}) = \mathbb{E} \left[  \left| 2 \eta(X) - 1 \right|
\mathds{1}_{\left\lbrace \Phi_d^{\star}(X) \neq \Phi^{\star}(X) \right\rbrace}
 \right].
$$
Then, following a classical control of the excess risk (see, \eg, 
\cite{GKM16}),
\begin{eqnarray}
\mathcal{R}_{f,g}(\Phi^{\star}_d) - \mathcal{R}_{f,g}(\Phi^{\star}) &=& \mathbb{E} \left[  \left| 2 \eta(X) - 1 \right|
\mathds{1}_{\left\lbrace \Phi_d^{\star}(X) \neq \Phi^{\star}(X) \right\rbrace} \left[ 
\mathds{1}_{\left\lbrace |\eta(X)-1/2| \leq \epsilon \right\rbrace}+\mathds{1}_{\left\lbrace |\eta(X)-1/2| > \epsilon \right\rbrace}
\right]  \right] \nonumber \\
& \leq &  \underbrace{2 \epsilon \PR \left( |\eta(X)-1/2| \leq \epsilon \right)}_{:=T_1} +
\underbrace{\PR\bigl(
\{ \Phi_d^{\star}(X) \neq \Phi^{\star}(X) \}  \cap \{|\eta(X)-1/2| > \epsilon \}\bigr)}_{:=T_2} \,.
\label{eq:inter1}
\end{eqnarray}
Note that, up to the quantity $2\epsilon$,  the term $T_1$ corresponds to the margin behavior discussed in Section~\ref{s:margin} above. By Proposition \ref{thm:margin} (note that $0<\epsilon \leq 1/8$), we have
$$ T_1 := 2 \epsilon \PR\bigl( |\eta(X)-1/2| \leq \epsilon \bigr) \leq 2\epsilon \left( 1 \wedge \frac{10 \epsilon}{\Delta}   \right).$$

\noindent
To control the second term $T_2$, we note (classically) that $\Phi^{\star}(X) = \mathds{1}_{\eta(X) \geq 1/2}$ and $\Phi_d^{\star}(X) = \mathds{1}_{\eta_d(X) \geq 1/2}$ together imply that
$$ T_2 := \PR\bigl(
\{ \Phi_d^{\star}(X) \neq \Phi^{\star}(X) \}  \cap \{|\eta(X)-1/2| > \epsilon \}\bigr) \leq \mathbb{P} \bigl( |\eta_d(X) - \eta(X) | > \epsilon \bigr) \,. $$
Using $Y \sim \mathcal{B}(1/2)$ and the conditional distribution of $X \vert Y$, we have
$$
T_2 \leq \frac{1}{2} \underbrace{\PR_{f} \bigl(|\eta_d(X)-\eta(X)|> \epsilon\bigr)}_{:=T_{2,1}} +  \frac{1}{2} \underbrace{\PR_{g} \bigl(|\eta_d(X)-\eta(X)|> \epsilon\bigr) }_{:=T_{2,2}}.
$$
For the sake of brevity, we only study $T_{2,1}$ (the second term $T_{2,2}$ can be upper bounded similarly by symmetry of the problem and by inverting the roles of $f$ and $g$).
 To alleviate the notation, we set $f_d:=\Pi_d(f)$ and $g_d:=\Pi_d(g)$. 
Recall from~\eqref{eq:qf} that $q_{f}$ denotes the likelihood ratio of the model $\PR_{f}$. Next we decompose $\eta-\eta_d $ using the four (a.s. positive) likelihood ratios $q_f$, $q_g$, $q_{f_d}$, and $q_{g_d}$:
$$
 \eta-\eta_d = \frac{q_f}{q_f+q_g} - \frac{q_{f_d}}{q_{f_d}+q_{g_d}} = \frac{q_{f}-q_{f_d}}{q_{f}+q_g} + q_{f_d} \left(\frac{1}{q_f+q_g}-\frac{1}{q_{f_d}+q_{g_d}}\right).
$$
In order to upper bound $T_{2,1}$, we use the triangle inequality three times in the decomposition above, we note that
\[
\left|\frac{1}{q_f+q_g}-\frac{1}{q_{f_d}+q_{g_d}}\right| = \left|\frac{q_{f_d}-q_f+q_{g_d}-q_g}{(q_f+q_g)(q_{f_d}+q_{g_d})}\right| \leq \frac{\left|q_{f_d}-q_f\right|}{q_f q_{f_d}} + \frac{\left|q_{g_d}-q_g\right|}{q_g q_{f_d}} \;,
\]
and we use the inclusion $\{Z_1+Z_2+Z_3>\epsilon\} \subseteq \{Z_1>\epsilon/2\} \cup \{Z_2>\epsilon/4\} \cup \{Z_3>\epsilon/4\}$ valid vor any random variables $Z_1, Z_2, Z_3$. We get:
\begin{eqnarray*}
\lefteqn{T_{2,1}}\\
 & \leq& \PR_f \left( |q_{f_d}(X)-q_f(X)| > \frac{\epsilon}{2} |q_f(X)+q_g(X)| \right)   +  \PR_f \left( q_{f_d}(X)|q_f(X)-q_{f_d}(X)| > \frac{\epsilon}{4} q_f(X) q_{f_d}(X)| \right) \\&&+ 
 \PR_f \left( q_{f_d}(X)|q_g(X)-q_{g_d}(X)| > \frac{\epsilon}{4} q_g(X) q_{f_d}(X)| \right)  \\
 & \leq &  \PR_f \left( |q_{f_d}(X)-q_{f}(X)| > \frac{\epsilon}{2} q_f(X) \right) + \PR_f \left( |q_{f_d}(X)-q_f(X)| > \frac{\epsilon}{4} q_f(X) \right)  + \PR_f \left( |q_{g_d}(X)-q_g(X)| > \frac{\epsilon}{4} q_g(X) \right) \\
 & \leq & 2\PR_f \left( |q_{f_d}(X)-q_f(X)| > \frac{\epsilon}{4} q_f(X) \right)  + \PR_f \left( |q_{g_d}(X)-q_g(X)| > \frac{\epsilon}{4} q_g(X) \right) \\
\end{eqnarray*}
Taking the logarithm, we can see that:
$$\PR_f \left( \left|\frac{q_{f_d}(X)}{q_f(X)}-1\right| > \frac{\epsilon}{4}  \right)  = \PR_f \left(\log \left( \frac{q_{f_d}}{q_f} \right) <  \log (1-\epsilon/4) \right) +\PR_f \left( \log \left( \frac{q_{f_d}}{q_f} \right)  > \log(1+\epsilon/4) \right)$$
Using the inequalities $\log(1+\epsilon/4) \geq \epsilon/8$ and $\log(1-\epsilon/4) \leq -\epsilon/4$ (that hold at least for all $0 < \epsilon \leq 1$) we obtain:
\begin{equation}
T_{2,1} \leq 2  \underbrace{\PR_f \left(  \log \frac{q_{f_d}(X)}{q_f(X)}  > \epsilon/8 \right) }_{:=S_1} + 
2 \underbrace{\PR_f  \left(\log  \frac{q_{f_d}(X)}{q_f(X)} < -\frac{\epsilon}{4} \right)}_{:=S_2} +  \underbrace{\PR_f \left(  \log \frac{q_{g_d}(X)}{q_g(X)}  > \frac{\epsilon}{8} \right)}_{:=S_3}  + 
 \underbrace{\PR_f  \left(\log  \frac{q_{g_d}(X)}{q_g(X)} < -\frac{\epsilon}{4} \right)}_{:=S_4} .
\label{eq:T21}
\end{equation}
The Girsanov formula makes it possible to write
$
\log \frac{q_{f_d}(X)}{q_f(X)} = \int_{0}^1 (f_d-f)(s) dX_s - \frac{1}{2} [\|f_d\|^2-\|f\|^2].
$
We study $S_1$ and remark that under $\PR_f$, $dX_s = f(s) ds + dW_s$ for all $s\in [0,1]$ so that
\begin{eqnarray*}
S_1= \PR_f \left(  \log \frac{q_{f_d}(X)}{q_f(X)}  > \frac{\epsilon}{8} \right) 
& = & \PR\left( \langle f_d-f,f \rangle + \int_{0}^1 (f_d-f)(s) dW_s - \frac{1}{2}\left[ \|f_d\|^2-\|f\|^2\right] >  \frac{\epsilon}{8}\right), \\
& = &  \PR\left(  \int_{0}^1 (f_d-f)(s) dW_s > \frac{1}{2}\left[ \|f\|^2-\|f_d\|^2\right] +  \frac{\epsilon}{8}\right)\\
& \leq & \PR \left( \xi >  \frac{\epsilon}{8}\right) \,,
\end{eqnarray*}
where  $\xi \sim \mathcal{N}\bigl(0,\|f_d-f\|^2\bigr) = \mathcal{N}(0,\delta_d^2)$. But, by a classical (sub)Gaussian tail bound stated, \eg, in \cite[p.22]{BLM}, we get
\begin{equation}
S_1  \leq \exp\left(-\frac{(\epsilon/8)^2}{2 \delta_d^2}\right) = \exp\left(-\frac{\epsilon^2}{128 \delta_d^2}\right) \,.
\label{eq:contS1}
\end{equation}
Combining the last inequality with the assumption
\begin{equation*}
\delta_d^2 \leq \frac{\epsilon^2}{512 \ln(1/\epsilon^2)} \leq \frac{\epsilon^2}{128 \ln(1/\epsilon^2)} \,,
\end{equation*}
we finally obtain $S_1  \leq \epsilon^2$.

\ \\
The second term $S_2$ introduced in (\ref{eq:T21}) can be dealt similarly, except that we can no longer neglect the positive term $\bigl(\|f\|^2-\|f_d\|^2\bigr)/2 = \delta_d^2/2$: considering again $\xi \sim \mathcal{N}(0,\delta_d^2)$, we have
\begin{align*}
S_2 & = \PR_f \left(  \log \frac{q_{f_d}(X)}{q_f(X)}  < -\frac{\epsilon}{4} \right) 
 = \PR \left( \xi < \frac{\delta_d^2}{2} - \frac{\epsilon}{4}\right) 
 \leq \PR \left( \xi < - \frac{\epsilon}{8}\right) \leq \epsilon^2 \,,
\end{align*}
where the last inequality follows from the same Gaussian concentration argument as in~\eqref{eq:contS1}, and where the inequality before last is because $\delta_d^2/2 \leq \epsilon/8$. Indeed, by the assumptions of Lemma~\ref{prop:CO},
\begin{equation}
\epsilon \geq \sqrt{512 \ln(8^2)} \max\bigl\{\delta_d,\tilde{\delta}_d\bigr\} \geq 24 \max\bigl\{\delta_d^2,\tilde{\delta}_d^2\bigr\} \,,
\label{eq:epsilondelta}
\end{equation}
where the second inequality follows from $\max\{\delta_d, \tilde{\delta}_d\} \leq \epsilon/24 \leq 1$ (as a result of the first inequality and $\epsilon \leq 1$). Therefore, $\delta_d^2/2 \leq \epsilon/48 \leq \epsilon/8$ as claimed above.

\ \\
We now focus on $S_3$: noting that
$
\log \frac{q_{g_d}(X)}{q_g(X)} = \int_{0}^1 (g_d-g)(s) dX_s - \frac{1}{2} [\|g_d\|^2-\|g\|^2],
$
we get
\begin{align*}
S_3 = \PR_f \left(  \log \frac{q_{g_d}(X)}{q_g(X)}  > \frac{\epsilon}{8} \right) & = \PR \left(
\langle g_d-g,f\rangle + \int_{0}^1 (g_d-g)(s) dW_s - \frac{1}{2} \left[ \|g_d\|^2-\|g\|^2 \right]> \frac{\epsilon}{8} \right) \\
& = \PR \left(
\langle g_d-g,f-f_d\rangle + \int_{0}^1 (g_d-g)(s) dW_s - \frac{1}{2} \left[ \|g_d\|^2-\|g\|^2 \right]> \frac{\epsilon}{8} \right) \,,
\end{align*}
where we used the fact that $g_d-g$ and $f_d$ are orthogonal.
Recall now that $\delta_d = \|f-f_d\|$ and $\tilde{\delta}_d = \|g-g_d\| = \sqrt{\|g\|^2-\|g_d\|^2}$.
If  $\tilde \xi  \sim \mathcal{N}(0,\tilde \delta_d^2)$, the last equality entails
\begin{equation*}
S_3 \leq \PR \left( \tilde \xi  > \frac{\epsilon}{8} - \frac{\tilde{\delta}_d^2}{2} - \tilde{\delta}_d \delta_d  \right) \leq \PR \left( \tilde \xi  > \frac{\epsilon}{8} - \frac{\tilde{\delta}_d^2}{2} - \frac{\tilde{\delta}_d^2}{2} - \frac{\delta_d^2}{2} \right) \leq \PR \left( \tilde \xi  > \frac{\epsilon}{16} \right) \leq \epsilon^2 \,,
\end{equation*}
where the second inequality follows from $\tilde{\delta}_d \delta_d \leq \bigl(\tilde{\delta}_d+\delta_d^2\bigr)/2$, where the third inequality is because $\max\bigl\{\delta_d^2,\tilde{\delta}_d^2\bigr\} \leq \epsilon/24$ (by~\eqref{eq:epsilondelta}), and where the last inequality follows from the same Gaussian tail bound as the one used in~\eqref{eq:contS1} and from the assumption $\tilde{\delta}_d^2 \leq \epsilon^2/\bigl(512 \ln(1/\epsilon^2)\bigr)$.

\ \\
A similar analysis shows that the last term $S_4$ introduced in (\ref{eq:T21}) also satisfies $S_4 \leq \epsilon^2$. Putting everything together, we finally get
\begin{equation*}
T_{2,1} \leq 6\epsilon^2 \,.
\end{equation*} 
By symmetry of the problem and by inverting the roles of $f$ and $g$, we can also see that $T_{2,2} \leq 6 \epsilon^2$. Summing the bounds on $T_1$, $T_{2,1}$, and $T_{2,2}$ concludes the proof.
\end{proof}
%
%
%
%

\subsection{Proof of Lemma \ref{lem:proj} (control of the estimation error)}
\label{s:pproj}

Though we now focus on the estimation error, most of the proof follows similar arguments as for Lemma~\ref{prop:CO} above: comparison of two regression functions, and Gaussian-type concentration inequalities.

\begin{proof}[Proof of Lemma \ref{lem:proj}]
Recall from Remark~\ref{rmk:truncatedSpace} (Section~\ref{sec:estimationerror}) that $\eta_d$ and $\Phi^*_d = \mathds{1}_{\{\eta_d \geq 1/2\}}$ correspond to the regression function and the Bayes classifier of the classification problem when the learner has only access to the projected input $\bfX_d := (\langle\varphi_j,X\rangle)_{1 \leq j \leq d}$. Since $\hat{\Phi}_d(X)$ only depends on $X$ through $\bfX_d$, its excess risk can be rewritten as
$$ \mathcal{R}_{f,g}(\hat\Phi_d) - \mathcal{R}_{f,g}(\Phi_d^\star) = \E \left[ |2\eta_d(X)-1| \mathds{1}_{\lbrace \hat\Phi_d(X) \not = \Phi_d^\star(X)  \rbrace}   \right] \,, $$
where the expectation is with respect to both the sample $(X_i,Y_i)_{1 \leq i \leq n}$ and the new input $X$. Now, for all $\epsilon >0$,
\begin{eqnarray*}
\mathcal{R}_{f,g}(\hat\Phi_d) - \mathcal{R}_{f,g}(\Phi_d^\star)  &=& 
\E\!\left[ |2\eta_d(X)-1| \mathds{1}_{\lbrace \hat\Phi_d(X) \not = \Phi_d^\star(X)  \rbrace} \mathds{1}_{
\lbrace{|\eta_d(X)-1/2| \leq \epsilon  \rbrace}}    \right] \nonumber\\
&&
+ \, \E \left[ |2\eta_d(X)-1| \mathds{1}_{\lbrace \hat\Phi_d(X) \not = \Phi_d^\star(X)  \rbrace} \mathds{1}_{
\lbrace{|\eta_d(X)-1/2| > \epsilon  \rbrace}}    \right]
\nonumber\\
&\leq & 2 \epsilon \PR_{X} \left( |\eta_d(X)-1/2| \leq \epsilon \right)  + \mathbb{P} ( | \hat \eta_d(X) - \eta_d(X) | > \epsilon) \,,
\end{eqnarray*}
where the last inequality follows from the inclusion $\bigl\{ \hat\Phi_d(X) \not = \Phi_d^\star(X) \bigr\}  \cap \bigl\{|\eta_d(X)-1/2| > \epsilon \bigr\} \subseteq \bigl\{|\hat{\eta}_d(X)-\eta_d(X)| > \epsilon  \bigr\}$ (because $\hat{\Phi}_d(X) = \mathds{1}_{\hat{\eta}_d(X) \geq 1/2}$ and $\Phi^*_d(X) = \mathds{1}_{\eta_d(X) \geq 1/2}$).
We can now apply the adaptation of Proposition~\ref{thm:margin} to the truncated space (see Remark~\ref{rem:pid}) to get
\begin{equation}\label{eq:toto}
\mathcal{R}_{f,g}(\hat\Phi_d) - \mathcal{R}_{f,g}(\Phi_d^\star) \leq 2 \epsilon \left( 1 \wedge\frac{10 \epsilon}
{\Delta_d}\right)
+ \mathbb{P} ( | \hat \eta_d(X) - \eta_d(X) | > \epsilon) \,.
\end{equation}
\noindent
Using $Y \sim \mathcal{B}(1/2)$ and the conditional distribution of $X$ given $Y$, we have:
$$ \mathbb{P} ( | \hat \eta_d(X) - \eta_d(X) | > \epsilon) = \frac{1}{2}  \underbrace{\mathbb{P}_f ( | \hat \eta_d(X) - \eta_d(X) | > \epsilon)}_{:=T_1} + \frac{1}{2} \underbrace{\mathbb{P}_g ( | \hat \eta_d(X) - \eta_d(X) | > \epsilon)}_{:=T_2} \,,$$
where, with a slight abuse of notation, the first probability $\PR_f(\cdot)$ is with respect to both the sample $(X_i,Y_i)_{1 \leq i \leq n}$ drawn i.i.d.~from~\eqref{eq:model} and a new independent input $X$ drawn from $\PR_f$; and similarly for the second probability $\PR_g(\cdot)$.

We now focus on $T_1$ until the end of the proof. (The control of $T_2$ is exactly similar, by symmetry of the model and by inverting the roles of $f$ and $g$.) Denote by $\gamma_d(x)=(2\pi)^{-d/2}e^{-\|x\|^2/2}$ the density of the standard Gaussian distribution on $\mathbb{R}^d$. By Remark~\ref{rmk:truncatedSpace} (Section~\ref{sec:estimationerror}), we have, setting $\bftheta_d := (\theta_1,\ldots,\theta_d)$ and $\bfmu_d := (\mu_1,\ldots,\mu_d)$,
\[
\eta_d(X) = \frac{F_d(X)}{F_d(X)+G_d(X)} \,, \quad \textrm{where} \quad  F_d (x) = \gamma_d(x- \bftheta_d) \quad \textrm{and} \quad G_d(x) = \gamma_d(x-\bfmu_d) \,.
\]
Similarly, by~\eqref{eq:etadestim}, the estimated regression function $\hat{\eta}_d$ can be rewritten as
$$ \hat \eta_d(X) = \frac{ \hat F_d(X) }{\hat F_d(X)+ \hat G_d(X)} \,, \quad \textrm{where} \quad \hat F_d (x) = \gamma_d(x- \hat \theta) \quad \textrm{and} \quad \quad\hat G_d(x) = \gamma_d(x-\hat \mu) \,.$$
Using simple algebra, we get
\begin{eqnarray}
\lefteqn{T_1:= \mathbb{P}_f ( | \hat \eta_d(X) - \eta_d(X) | > \epsilon) } \nonumber \\
& = & \mathbb{P}_f\left(  \left| \frac{\hat F_d(X)}{\hat F_d(X) + \hat G_d(X) } - \frac{F_d(X)}{F_d(X)+G_d(X)}      \right| > \epsilon \right)  \nonumber \\
& \leq & \mathbb{P}_f\left(  \left| \frac{\hat F_d(X)- F_d(X)}{ F_d(X) +  G_d(X) } + \hat F_d(X) \left( \frac{1}{\hat F_d(X)+\hat G_d(X)} -\frac{1}{F_d(X)+G_d(X)} \right)      \right| > \epsilon \right) \nonumber  \\
& \leq & \mathbb{P}_f\left(  \left| \frac{\hat F_d(X)- F_d(X)}{ F_d(X) + G_d(X) }  \right| > \frac{\epsilon}{3} \right)  + \mathbb{P}_f \left( \left|  \hat F_d(X) \left( \frac{1}{\hat F_d(X)+\hat G_d(X)} -\frac{1}{F_d(X)+G_d(X)} \right)      \right| > \frac{2 \epsilon}{3} \right)  \nonumber \\
& =: & \mathbb{P}(A_1) + \mathbb{P}(A_2) \,.
\label{eq:pa12}
\end{eqnarray}
\ \\
\textbf{Control of $\mathbb{P}(A_1)$.} First note that
\begin{eqnarray}
 \mathbb{P}(A_1) 
 & = & \mathbb{P}_f\left( | \hat F_d(X) - F_d(X) | > \frac{\epsilon}{3} (F_d(X)+G_d(X)) \right) \nonumber \\ 
 & \leq & \mathbb{P}_f\left( |\hat F_d(X) - F_d(X) | > \frac{\epsilon}{3} F_d(X) \right) \label{eq:FdFdhat} \\
 & = & \mathbb{P}_f\left( \bigl| \hat F_d(X) /F_d(X) -1 \bigr| > \frac{\epsilon}{3}\right) \nonumber\\
 & = & \mathbb{P}_f \left( \left|  e^{\langle \bfX_d-(\hat{\theta}+ \bftheta_d)/2 , \hat{\theta} - \bftheta_d \rangle} -1    \right| > \frac{\epsilon}{3}   \right) \,. \nonumber
\end{eqnarray}
Since we have $-\log(1-u) \geq \log(1+u) \geq u/2$ for $u \in (0,1)$, some straightforward computations yield:
\begin{eqnarray*}
 \mathbb{P}(A_1) 
 & \leq & \mathbb{P}_f \left( \left| \left\langle \bfX_d-\frac{\hat{\theta}+ \bftheta_d}{2} , \hat{\theta} - \bftheta_d \right\rangle \right| > \log\left(1+\frac{\epsilon}{3}\right) \right) \\
& \leq & \mathbb{P}_f \left( \left| \left\langle \bfX_d-\frac{\hat{\theta}+ \bftheta_d}{2} , \hat{\theta} - \bftheta_d \right\rangle \right| > \frac{\epsilon}{6}\right) \\
 & = & \mathbb{P}_f \left( \left|\left\langle \bfX_d-\bftheta_d +\frac{\bftheta_d- \hat{\theta}}{2} , \hat{\theta} - \bftheta_d \right\rangle \right| > \frac{\epsilon}{6} \right) \\
 & \leq & \mathbb{P}_f \left( \left|\left\langle \bfX_d-\bftheta_d , \hat{\theta} - \bftheta_d \right\rangle \right| > \frac{\epsilon}{6} - \frac{\|\hat{\theta} - \bftheta_d\|^2}{2} \right) \,.
 \end{eqnarray*}
 Now, note from~\eqref{eq:Gaussiancoeffs}, \eqref{eq:defX0X1}--\eqref{eq:defhatthethamu}, and Remark~\ref{rmk:conditionallyGaussian} that, under $\mathbb{P}_{\otimes^n} \otimes \PR_f$ and on the event $\{N_1 > 0\}$, the random variables $\xi_j := X_{d,j}-\theta_{d,j} = \langle \varphi_j , X \rangle - \theta_j$, $1 \leq j \leq d$, and
\[
\zeta_j := \sqrt{N_1} \left(\hat \theta_{j} - \theta_{d,j}\right) = \frac{1}{\sqrt{N_1}} \sum_{i=1}^{N_1} \epsilon^1_{i,j} \,, \qquad 1 \leq j \leq d\,,
\]
are i.i.d.~$\cN(0,1)$ conditionally on $Y_{1:n} := (Y_1,\ldots,Y_n)$. (On the event $\{N_1=0\}$, we define the $\zeta_j$ so as to coincide with other independent $\cN(0,1)$ random variables $\zeta'_j$.) As a consequence, the random variables $\xi_1,\ldots,\xi_d,\zeta_1,\ldots,\zeta_d$ are i.i.d.~$\cN(0,1)$ (unconditionally).

\ \\
Note also from Hoeffding's lemma (see, \eg, \cite{BLM}) and $n/2 - \sqrt{n \log(n)/2} \geq n/4$ (because $n \geq 27$) that
\begin{equation}
 \PR\left( N_1 < \frac{n}{4} \right) \leq  \PR\left( N_1 < \frac{n}{2} - \sqrt{\frac{n \log n}{2}} \right) \leq \frac{1}{n} \;.
\label{eq:Hoeffdingn4}
\end{equation}
Therefore, we deduce that
\begin{eqnarray}
 \mathbb{P}(A_1) & \leq & \mathbb{P}_f \left( \left|\left\langle \bfX_d-\bftheta_d , \hat{\theta} - \bftheta_d \right\rangle \right| > \frac{\epsilon}{6} - \frac{\|\hat{\theta} - \bftheta_d\|^2}{2}\,, N_1 \geq \frac{n}{4} \right) + 
 \, \PR\left( N_1 < \frac{n}{4} \right) \nonumber \\
&\leq & \mathbb{P} \left( \left| \sum_{j=1}^d \xi_j \zeta_j \right| \geq \frac{\sqrt{N_1} \epsilon}{6} - \frac{\|\zeta\|^2}{2 \sqrt{N_1}} \,, N_1 \geq \frac{n}{4} \right)  + \frac{1}{n} \nonumber \\
&\leq & \mathbb{P} \left( \left| \sum_{j=1}^d \xi_j \zeta_j \right| \geq \frac{\sqrt{n} \epsilon}{12} - \frac{\|\zeta\|^2}{\sqrt{n}} \right)  + \frac{1}{n} \nonumber \\
&\leq & \mathbb{P} \left( \left| \sum_{j=1}^d \xi_j \zeta_j \right| \geq \frac{\sqrt{n} \epsilon}{24} \right) + \mathbb{P} \left( \|\zeta\|^2 > \frac{n \epsilon}{24} \right)  + \frac{1}{n} \,. \label{eq:PA1twodeviations}
\end{eqnarray}
We control the first deviation probability above. First, recalling that the $\xi_j$ and $\zeta_j$ are i.i.d.~$\cN(0,1)$, and conditioning by $(\xi_1,\ldots,\xi_d)$, we get
\[
\mathbb{P} \left( \left| \sum_{j=1}^d \xi_j \zeta_j \right| \geq \frac{\sqrt{n} \epsilon}{24} \right)  \leq \mathbb{E} \left[  \mathbb{P} \left(    \left| \sum_{j=1}^d \xi_j \zeta_j \right| \geq \frac{\sqrt{n} \epsilon}{24} \; \Big| \; \xi_1,\dots,\xi_d    \right)  \right] \leq 2 \, \mathbb{E} \left[  \exp \left(  - \frac{n \epsilon^2}{1152 \sum_{j=1}^d \xi_j^2} \right) \right]\,, \nonumber
\]
where the last inequality is because, conditionally on $(\xi_1,\ldots,\xi_d)$, the random variable $Z=\sum_{j=1}^d \xi_j \zeta_j$ is Gaussian with zero mean and variance $V=\sum_{j=1}^d \xi_j^2$ and thus satisfies $\PR(|Z| > z \, | \, \xi_1,\dots,\xi_d) \leq 2 e^{-z^2/(2 V)}$ for all $z>0$. But, distinguishing whether $\sum_{j=1}^d \xi_j^2$ is below or above $4 d \log n$, we obtain
\begin{align*}
\mathbb{P} \left( \left| \sum_{j=1}^d \xi_j \zeta_j \right| \geq \frac{\sqrt{n} \epsilon}{24} \right) & \leq 2 \exp\!\left( \!- \frac{n \epsilon^2}{4608 d \log n } \right) + 2 \PR\!\left(\sum_{j=1}^d \xi_j^2 > 4 d \log n\right) \leq 2 \exp\!\left(\! - \frac{n \epsilon^2}{4608 d \log n } \right) + \frac{2}{n} \,,
\end{align*}
where we used the concentration inequality for the $\chi^2$ statistics of \cite[Lemma~1]{Beatrice2000}
\begin{equation}
\forall x >0\,, \qquad \mathbb{P}\left( \sum_{j=1}^d \xi_j^2 > d + 2 \sqrt{d x} + 2x   \right) \leq e^{-x}
\label{eq:chi2concentration}
\end{equation}
for $x=\log(n)$, and where we noted (since $2 a b \leq a^2 + b^2$ and $\log n \geq 2$ for $n \geq 27$) that
\begin{equation}
d + 2 \sqrt{d \log n} + 2 \log n \leq 2d + 3 \log n \leq 4 d \log n \,.
\label{eq:bounddlogn}
\end{equation}
Plugging the above inequalities into~\eqref{eq:PA1twodeviations}, we finally obtain
\begin{equation}
\mathbb{P}(A_1) \leq 2 \exp \left(  - \frac{n \epsilon^2}{4608 d \log n } \right) + \mathbb{P} \left( \|\zeta\|^2 > \frac{n \epsilon}{24} \right) + \frac{3}{n} \,.
\label{eq:PA1bound}
\end{equation}

\ \\
\textbf{Control of $\mathbb{P}(A_2)$.} We have:
\begin{eqnarray*}
\mathbb{P}(A_2) 
& := & \mathbb{P}_f \left( \left|  \hat F_d(X) \left( \frac{1}{\hat F_d(X)+\hat G_d(X)} -\frac{1}{F_d(X)+G_d(X)} \right)      \right| > \frac{2\epsilon}{3} \right) \\
& = & \mathbb{P}_f \left( \left|  \hat F_d(X) \left( \frac{F_d(X)- \hat F_d(X) + G_d(X) - \hat G_d(X)}{(\hat F_d(X)+\hat G_d(X))(F_d(X)+G_d(X))} \right)      \right| > \frac{2\epsilon}{3} \right) \\
& \leq & \mathbb{P}_f \left( \left|  \hat F_d(X) (F_d(X)- \hat F_d(X) ) \right| > \frac{\epsilon}{3} (F_d(X)+ G_d(X)) (\hat F_d(X) + \hat G_d(X)) \right) \\
&  & + \, \mathbb{P}_f \left( \left|  \hat F_d(X) (G_d(X)- \hat G_d(X) ) \right| > \frac{\epsilon}{3} (F_d(X)+ G_d(X)) (\hat F_d(X) + \hat G_d(X)) \right) \\
& \leq & \mathbb{P}_f \left( \left|  \hat F_d(X) (F_d(X)- \hat F_d(X) ) \right| > \frac{\epsilon}{3} F_d(X)\hat F_d(X)  \right) \\
&  & + \, \mathbb{P}_f \left( \left|  \hat F_d(X) (G_d(X)- \hat G_d(X) ) \right| > \frac{\epsilon}{3}  G_d(X) \hat F_d(X) \right) \\
& \leq & \mathbb{P}_f \left( \left| F_d(X)- \hat F_d(X)  \right| > \frac{\epsilon}{3} F_d(X) \right)  + \mathbb{P}_f \left( \left|  G_d(X)- \hat G_d(X)  \right| > \frac{\epsilon}{3}  G_d(X)  \right) \,. \\
\end{eqnarray*}
The first term has already been studied above (see~\eqref{eq:FdFdhat} and the following inequalities) and thus satisfies the same upper bound as $\mathbb{P}(A_1)$ in~\eqref{eq:PA1bound}. As for the second term, following the same lines as those leading to~\eqref{eq:PA1twodeviations}, we can see that
\begin{align*}
& \mathbb{P}_f \left( \left|  G_d(X)- \hat G_d(X)  \right| > \frac{\epsilon}{3}  G_d(X)  \right) \\
& \qquad \leq \mathbb{P}_f \left( \left| \left\langle \bfX_d-\frac{\hat{\mu}+ \bfmu_d}{2} , \hat{\mu} - \bfmu_d \right\rangle \right| > \frac{\epsilon}{6}\right) \\
& \qquad \leq \mathbb{P}_f \left( \Big| \big\langle \bfX_d-\bftheta_d, \hat{\mu} - \bfmu_d \big\rangle \Big| > \frac{\epsilon}{6} - \left|\left\langle \bftheta_d - \bfmu_d + \frac{\bfmu_d - \hat{\mu}}{2}, \hat{\mu} - \bfmu_d \right\rangle\right|\right) \\
& \qquad \leq \mathbb{P}_f \biggl( \Big| \big\langle \bfX_d-\bftheta_d, \hat{\mu} - \bfmu_d \big\rangle \Big| > \frac{\epsilon}{6} - \underbrace{\left(\|\bftheta_d - \bfmu_d\| + \frac{\|\hat{\mu} - \bfmu_d\|}{2}\right) \|\hat{\mu} - \bfmu_d\|}_{\leq 4 \bigl(\Delta_d + 2 \sqrt{\frac{d\log(n)}{n}} \, \bigr) \sqrt{\frac{d\log(n)}{n}} \textrm{ w.p. $\geq 1-2/n$}} \, \biggr) \\
& \qquad \leq \mathbb{P}_f \left( \Big| \big\langle \bfX_d-\bftheta_d, \hat{\mu} - \bfmu_d \big\rangle \Big| > \frac{\epsilon}{12} \right) + \frac{2}{n} \,,
\end{align*}
where we used~\eqref{eq:Hoeffdingn4} and~\eqref{eq:chi2concentration}--\eqref{eq:bounddlogn} again, and where the last inequality holds true whenever
\begin{equation}
\left(\Delta_d + 2 \sqrt{\frac{d\log(n)}{n}} \, \right) \sqrt{\frac{d\log(n)}{n}} \leq \frac{\epsilon}{48} \,.
\label{eq:conditionDeltadEpsilon}
\end{equation}
Mimicking what we did to derive~\eqref{eq:PA1twodeviations}, we then get
\begin{align*}
\mathbb{P}_f\!\left( \left|  G_d(X)- \hat G_d(X)  \right| > \frac{\epsilon}{3}  G_d(X)  \right) & \leq \mathbb{P}\!\left( \left| \sum_{j=1}^d \xi_j \zeta_j \right| \geq \frac{\sqrt{n} \epsilon}{24} \right) + \frac{1}{n} + \frac{2}{n} \leq 2 \exp \left(  - \frac{n \epsilon^2}{4608 d \log n } \right) + \frac{5}{n} \,.
\end{align*}
Putting everything together, we can see that, provided~\eqref{eq:conditionDeltadEpsilon} holds,
\[
\mathbb{P}(A_2)  \leq 4 \exp \left(  - \frac{n \epsilon^2}{4608 d \log n } \right) + \mathbb{P} \left( \|\zeta\|^2 > \frac{n \epsilon}{24} \right) + \frac{8}{n} \,.
\]

\ \\
\textbf{Conclusion} Combining all results above, we get, under condition~\eqref{eq:conditionDeltadEpsilon},
\[
T_1 = \mathbb{P}(A_1) + \mathbb{P}(A_2) \leq 6 \exp \left(  - \frac{n \epsilon^2}{4608 d \log n } \right) + 2 \, \mathbb{P} \left( \|\zeta\|^2 > \frac{n \epsilon}{24} \right) + \frac{11}{n}
\]
so that (the upper bound on $T_2$ is identical by symmetry of the problem):
\[
\mathbb{P} ( | \hat \eta_d(X) - \eta_d(X) | > \epsilon) \leq 6 \exp \left(  - \frac{n \epsilon^2}{4608 d \log n } \right) + 2 \, \mathbb{P} \left( \|\zeta\|^2 > \frac{n \epsilon}{24} \right) + \frac{11}{n} \,.
\]
To conclude the proof, we note that, if~\eqref{eq:conditionDeltadEpsilon} holds true, then $n\epsilon/24 \geq 4 d \log n \geq d + 2 \sqrt{d \log n} + 2 \log n$ (by~\eqref{eq:bounddlogn}), so that $\mathbb{P}\bigl( \|\zeta\|^2 > n \epsilon/24\bigr) \leq 1/n$ by~\eqref{eq:chi2concentration}.
\end{proof}

\subsection{Proof of Theorem~\ref{prop:minmax} (excess risk of $\hat{\Phi}_{d_n}$)}
\label{s:pminmax}

In all the sequel we fix $f,g \in \Hs(R)$ and show that
\begin{equation}
\mathcal{R}_{f,g}(\hat\Phi_{d_n}) - \inf_{\Phi} \mathcal{R}_{f,g}(\Phi)  \leq \left\{\begin{array}{ll}
	c \, R^{\frac{1}{2s+1}} n^{-\frac{s}{2s+1}} \log(n) & \textrm{if } \Delta < R^{\frac{1}{2s+1}} n^{-\frac{s}{2s+1}} \log(n) \\[0.2em]
	\displaystyle \frac{c}{\Delta} R^{\frac{2}{2s+1}} n^{-\frac{2s}{2s+1}} \log^2(n) & \textrm{if } \Delta \geq R^{\frac{1}{2s+1}} n^{-\frac{s}{2s+1}} \log(n)
\end{array}\right.
\label{eq:thm1-upperboundfg}
\end{equation}
where $\Delta := \|f-g\|$. This immediately entails the inequality of the theorem (\ie, the one involving the supremum) since the right-hand side of~\eqref{eq:thm1-upperboundfg} is non-increasing in $\Delta$.

\ \\
Recall that $\Phi = \mathds{1}_{\eta \geq 1/2}$ is the Bayes (optimal) classifier and that $\Phi_{d_n}^\star$ is the Bayes classifier in the $d_n$-dimensional truncated space (see Remark~\ref{rmk:truncatedSpace} in Section~\ref{sec:estimationerror}). We decompose the excess risk into estimation and approximation errors and use Lemmas~\ref{lem:proj} and~\ref{prop:CO}: for some values of $\epsilon_1$ and $\epsilon_2$ to be determined later,
\begin{align}
& \mathcal{R}_{f,g}(\hat\Phi_{d_n}) - \inf_{\Phi} \mathcal{R}_{f,g}(\Phi) \nonumber \\
& \qquad = \mathcal{R}_{f,g}(\hat\Phi_{d_n}) - \mathcal{R}_{f,g}(\Phi_{d_n}^\star) + \mathcal{R}_{f,g}(\Phi_{d_n}^\star) - \mathcal{R}_{f,g}(\Phi^\star) \nonumber \\
& \qquad \leq 2 \epsilon_1 \left( 1\wedge \frac{10 \epsilon_1}{\Delta_{d_n}}  \right) + 6 \exp \left(  - \frac{n \epsilon_1^2}{4608 \, d_n \log n } \right) + \frac{13}{n} + 12 \epsilon_2^2 + 2 \epsilon_2 \left( 1 \wedge \frac{10\epsilon_2}{\Delta} \right) \nonumber \\
& \qquad \leq 2 \epsilon_1 \left( 1\wedge \frac{10 \epsilon_1}{\Delta_{d_n}}  \right) + \frac{19}{n} + 12 \epsilon_2^2 + 2 \epsilon_2 \left( 1 \wedge \frac{10\epsilon_2}{\Delta} \right) \,, \label{eq:thm1-boundtooptimize}
\end{align}
where $\Delta_{d_n} := \|f_{d_n}-g_{d_n}\|$, and where we assumed that $\epsilon_1^2 \geq 4608 \, d_n \log^2(n) / n$ (to be checked below).

\ \\
In all the sequel the value of the constant $N_{s,R}$ may change from line to line. Our first constraint on~$N_{s,R}$ is that $N_{s,R} \geq 1/R^2$, so that $d_n := \lfloor (R^2 n)^{\frac{1}{2s+1}} \rfloor \geq 1$ for all $n \geq N_{s,R}$. The choice of $d_n$ also guarantees the bias--variance tradeoff $R d_n^{-s} \approx \sqrt{d_n/n} \approx R^{\frac{1}{2s+1}} n^{-\frac{s}{2s+1}}$. More precisely, provided $N_{s,R}$ is chosen large enough, we get for all $n \geq N_{s,R}$ that
\begin{equation}
\sqrt{\frac{d_n}{n}} \leq R^{\frac{1}{2s+1}} n^{-\frac{s}{2s+1}} \leq R d_n^{-s} \leq 2 R^{\frac{1}{2s+1}} n^{-\frac{s}{2s+1}} \,.
\label{eq:biasvariancelink}
\end{equation}

\ \\
We now choose $\epsilon_1$ and $\epsilon_2$ so as to minimize~\eqref{eq:thm1-boundtooptimize}, while meeting the assumptions of Lemmas~\ref{lem:proj} and~\ref{prop:CO}.
\begin{itemize}
	\item We choose
	\[
	\epsilon_1 := 48 \left(\Delta_{d_n} + 2 \sqrt{\frac{d_n\log(n)}{n}} + \sqrt{2 \log(n)} \, \right) \sqrt{\frac{d_n\log(n)}{n}} \;.
	\]
	This entails that $0 < \epsilon_1 \leq 1/8$ for all $n \geq N_{s,R}$ (provided $N_{s,R}$ is chosen large enough), that Assumption~\eqref{eq:lemma2conditioneps} of Lemma~\ref{lem:proj} holds true, and that the requirement $\epsilon_1^2 \geq 4608 \, d_n \log^2(n) / n$ above is met.
	\item We choose
	\[
	\epsilon_2 := 32 R d_n^{-s} \sqrt{\log \frac{1}{32 R d_n^{-s}}} \;.
	\]
	Choosing $N_{s,R}$ large enough, we can guarantee for all $n \geq N_{s,R}$ that $0 < \epsilon_2 \leq 1/8$, as well as $\log\bigl[1/(32 R d_n^{-s})\bigr] \geq 1$ so that $\epsilon_2 \geq 32 R d_n^{-s}$ and therefore $\epsilon_2 \geq 32 R d_n^{-s} \sqrt{\log(1/\epsilon_2)}$, \ie,
	\[
	R^2 d_n^{-2s} \leq \frac{\epsilon_2^2}{512 \log\bigl(1/\epsilon_2^2\bigr)} \,.
	\]
	Now, note that $\|f-f_{d_n}\|^2 \leq R^2 d_n^{-2s}$ for all $f \in \Hs(R)$ because
	\[
	\|f-f_{d_n}\|^2 = \sum_{k=d_n+1}^{+\infty} c_k(f)^2 \leq d_n^{-2s} \sum_{k=d_n+1}^{+\infty} c_k(f)^2 k^{2s} \leq R^2 d_n^{-2s} \,.
	\]
	Combining the above inequalities implies that Assumption~\eqref{eq:lemma1conditioneps} of Lemma~\ref{prop:CO} is met. 
\end{itemize}

\ \\
Before plugging the values of $\epsilon_1$ and $\epsilon_2$ into~\eqref{eq:thm1-boundtooptimize}, we compare $\Delta_{d_n}$ with $\Delta$:
\begin{align}
\Delta_{d_n} & := \|f_{d_n}-g_{d_n}\| \geq \|f-g\| - \|f-f_{d_n}\| - \|g-g_{d_n}\| \geq \Delta - 2 R d_n^{-s} \geq \frac{\Delta}{10} \label{eq:Deltan}
\end{align}
whenever $\Delta \geq (20/9) R d_n^{-s}$. By~\eqref{eq:biasvariancelink} a sufficient condition is that $\Delta \geq (40/9)R^{\frac{1}{2s+1}} n^{-\frac{s}{2s+1}}$ or even that $\Delta \geq R^{\frac{1}{2s+1}} n^{-\frac{s}{2s+1}} \log(n)$ (provided $N_{s,R} \geq e^{40/9} \approx 85.2$). This is the threshold value we use below, since it makes the righ-hand side of~\eqref{eq:thm1-upperboundfg} continuous in $\Delta$.

\ \\
\textbf{Case 1: $\Delta < R^{\frac{1}{2s+1}} n^{-\frac{s}{2s+1}} \log(n)$}.

\ \\
We substitute the values of $\epsilon_1$ and $\epsilon_2$ into~\eqref{eq:thm1-boundtooptimize} and discard the (relatively large) terms $10 \, \epsilon_1/\Delta_{d_n}$ and $10 \, \epsilon_2/\Delta$. We obtain, noting that $12 \epsilon_2^2 \leq 12 \epsilon_2/8 \leq 2 \epsilon_2$:
\begin{align}
& \mathcal{R}_{f,g}(\hat\Phi_{d_n}) - \inf_{\Phi} \mathcal{R}_{f,g}(\Phi) \leq 2 \epsilon_1 + \frac{19}{n} + 12 \epsilon_2^2 + 2 \epsilon_2 \leq 2 \epsilon_1 + 4 \epsilon_2 + \frac{19}{n} \nonumber \\
& \qquad \leq 96 \left(\Delta_{d_n} + 2 \sqrt{\frac{d_n\log(n)}{n}} + \sqrt{2 \log(n)} \, \right) \sqrt{\frac{d_n\log(n)}{n}} + 128 R d_n^{-s} \sqrt{\log \frac{1}{32 R d_n^{-s}}} + \frac{19}{n} \nonumber \\
& \qquad \leq 96 \left(2 R + 2 R^{\frac{1}{2s+1}} n^{-\frac{s}{2s+1}} \sqrt{\log(n)} + \sqrt{2 \log(n)} \, \right) R^{\frac{1}{2s+1}} n^{-\frac{s}{2s+1}} \sqrt{\log(n)} \nonumber \\
& \qquad \qquad + 256 R^{\frac{1}{2s+1}} n^{-\frac{s}{2s+1}} \sqrt{\log \frac{n^{\frac{s}{2s+1}}}{32 R^{\frac{1}{2s+1}}}} + \frac{19}{n} \nonumber \\
& \qquad \leq c_1 R^{\frac{1}{2s+1}} n^{-\frac{s}{2s+1}} \log(n) \,, \label{eq:thm1case1}
\end{align}
where the inequality before last follows from~\eqref{eq:biasvariancelink} and from $\Delta_{d_n} \leq \Delta \leq \|f\|+\|g\| \leq 2 R$ (since $f,g \in \Hs(R)$), and where~\eqref{eq:thm1case1} holds for all $n \geq N_{s,R}$ provided the absolute constant $c_1 > 0$ and the constant $N_{s,R}$ are chosen large enough.

\ \\
\textbf{Case 2: $\Delta \geq R^{\frac{1}{2s+1}} n^{-\frac{s}{2s+1}} \log(n)$}.

\ \\
Following similar calculations, but using now the (relatively small) terms $10 \, \epsilon_1/\Delta_{d_n}$ and $10 \, \epsilon_2/\Delta$, we can see from~\eqref{eq:thm1-boundtooptimize} and then~\eqref{eq:Deltan} that, for some absolute constants $c_2,c_3 > 0$,
\begin{align}
& \mathcal{R}_{f,g}(\hat\Phi_{d_n}) - \inf_{\Phi} \mathcal{R}_{f,g}(\Phi) \leq \frac{20 \epsilon_1^2}{\Delta_{d_n}}  + \frac{19}{n} + 12 \epsilon_2^2 + \frac{20\epsilon_2^2}{\Delta} \leq \frac{200 \epsilon_1^2}{\Delta} + \frac{20\epsilon_2^2}{\Delta} + 12 \epsilon_2^2 + \frac{19}{n} \nonumber \\
& \qquad \leq c_2 R^{\frac{2}{2s+1}} n^{-\frac{2s}{2s+1}} \left(\frac{\log^2(n)}{\Delta} + \frac{\log(n)}{\Delta} + \log(n) \right) + \frac{19}{n} \nonumber \\
& \qquad \leq \frac{c_3 R^{\frac{2}{2s+1}} n^{-\frac{2s}{2s+1}} \log^2(n)}{\Delta} \,, \label{eq:thm1case2}
\end{align}
where the last two inequalities hold true for all $n \geq N_{s,R}$ provided $N_{s,R}$ is chosen large enough (\eg, $\log^2(n)/\Delta \geq \log(n)$ when $n \geq e^{2R} \geq e^{\Delta}$).

\ \\
\textbf{Conclusion:} We derive \eqref{eq:thm1-upperboundfg} by combining~\eqref{eq:thm1case1} and~\eqref{eq:thm1case2} and by choosing $c:=\max\{c_1,c_3\}$. This concludes the proof of Theorem~\ref{prop:minmax}.

\section{Proof of the minimax lower bound (Theorem \ref{thm:minimaxlowerbound})}
\label{s:minimaxlowerbound}


This section contains the proof of our minimax lower bound (Theorem~\ref{thm:minimaxlowerbound}). We will pay a specific attention to the influence of the separation distance $\Delta = \|f-g\|$ on the misclassification rate. 
We directly start with the proof in Section~\ref{sec:LBproof-main} below. We will use several key technical ingredients gathered in Section~\ref{sec:proof-lowerbound-packing}.

\subsection{Proof of Theorem~\ref{thm:minimaxlowerbound}}
\label{sec:LBproof-main}

\noindent
\underline{First case}: $\Delta < R^{1/(2s+1)} \, n^{-s/(2s+1)}$. Note that
\[
\Bigl\{ (f,g) \in \Hs(R) \times \Hs(R): \, \|f-g\| \geq \Delta \Bigr\} \supseteq \Bigl\{ (f,g) \in \Hs(R) \times \Hs(R): \, \|f-g\| \geq R^{1/(2s+1)} \, n^{-s/(2s+1)} \Bigr\} \,.
\]
Therefore, taking the supremum over all such functions, we directly obtain a lower bound on the minimax excess risk by applying the lower bound $\bigl(c e^{-2 \Delta^2}/\Delta\bigr) R^{2/(2s+1)} n^{-2s/(2s+1)}$ of the second case below with $\Delta = R^{1/(2s+1)} \, n^{-s/(2s+1)}$. This yields the desired lower bound of $c e^{-2 R^{2/(2s+1)}} R^{1/(2s+1)} n^{-s/(2s+1)}$. \\

\noindent
\underline{Second case}: $\Delta \geq R^{1/(2s+1)} \, n^{-s/(2s+1)}$. We proceed in three main steps.

\paragraphnospace{Step 1: reduction to a finite-dimensional $\mathbb{L}^1$-estimation problem, and some notation.}

\noindent
\textit{Finite-dimensional construction.}
Let $\hat{\Phi}$ be any classifier built from the sample $(X_i,Y_i)_{1 \leq i \leq n}$. As is usual when deriving nonparametric lower bounds, we restrict the supremum over all $f,g \in \cH_s(R)$ to a well-chosen finite-dimensional subset. More precisely, in what follows, we restrict our attention to functions $f:[0,1] \to \R$ and $g:[0,1] \to \R$ of the form:
\[
\forall  t \in \mathbb{R}, \qquad f(t)=f_{\theta}(t) \eqdef \sum_{j=1}^d \theta_j \varphi_j(t) \,, \quad \theta \in \Theta\,, \qquad \textrm{and} \qquad g(t) = 0 \,,
\]
for some $d \in \N^*$ and some parameter set $\Theta \subseteq \bigl\{\theta \in \R^d: \, \theta_1 = \Delta \textrm{ and } \sum_{j=2}^d \theta_j^2 j^{2s} \leq R^2 - \Delta^2\bigr\}$ to be made more precise in Step~2 below. Note that $\langle f_{\theta},\varphi_j\rangle=\theta_j$, so that the notation $\theta_j$ is consistent with that of Section~\ref{sec:pluginFinitedim}. \\

\noindent
\textit{Some notation.} 
The notation we choose for this proof differs slightly from that of the rest of the paper. We write $\PR_{\theta}$ for the joint distribution of the training and test samples $\bigl((X_i,Y_i)_{1 \leq i \leq n},(X,Y)\bigr)$ when the true parameter is $\theta$, and denote by $\E_{\theta}$ the corresponding expectation. We also denote by $Q_{\theta}$ the distribution of the process $(Z_t)_{0 \leq t \leq 1}$ defined by $d Z_t = f_{\theta}(t) d t + d W(t)$. We define the $\mathbb{L}^1$-norm of $h$ by
$$\|h\|_{L^1(Q_0)} \eqdef \int |h(x)| \dd Q_0(x) = \E\bigl[|h(W)|\bigr].$$
Finally, for $X=(X(t))_{0 \leq t \leq 1}$ solution of \eqref{eq:model}, we set
\[
\tilde{X}_j \eqdef \langle \varphi_j,X \rangle = \int_0^1 \varphi_j(t) d X(t) \;.
\]
Note that when $X$ is a standard Brownian motion on $[0,1]$, then $(\tilde{X}_j)_{j \geq 1}$, are independent standard Gaussian random variables (since $(\varphi_j)_{j \geq 1}$ is an orthonormal basis).

\ \\
\textit{Reduction to an $\mathbb{L}^1$-estimation problem.} Note that $g = 0 \in \Hs(R)$ and $\{f_{\theta}: \theta \in \Theta\} \subseteq \Hs(R)$ (see the definition in \eqref{def:Hs}), and that $\|f_{\theta}-0\| = \|\theta\| \geq \Delta$ for all $\theta \in \Theta$ (we use the notation $\|.\|$ both in $\mathbb{L}^2([0,1])$ and in $\R^d$). Therefore,
\begin{align}
\sup_{\substack{f,g \in \Hs(R) \\ \|f-g\| \geq \Delta}} \left\{ \mathcal{R}_{f,g}(\hat\Phi) - \inf_{\Phi} \mathcal{R}_{f,g}(\Phi)  \right\} & \geq \sup_{\theta \in \Theta} \left\{ \mathcal{R}_{f_{\theta},0}(\hat\Phi) - \inf_{\Phi} \mathcal{R}_{f_{\theta},0}(\Phi)  \right\}  \nonumber \\
& =  \sup_{\theta \in \Theta} \E_{\theta}\!\left[ \big|2 \eta_{\theta} (X) - 1 \big | \indicator{\hat{\Phi}(X) \neq \Phi_{\theta}(X)} \right], \label{eq:minimaxLB-Bayesformula}
\end{align}
where $\eta_{\theta}(x) = \PR_{\theta}(Y=1|X=x)$ denotes the regression function corresponding to the statistical model~\eqref{eq:model} with $f = f_{\theta}$ and $g=0$, and where $\Phi_{\theta}(x) = \indicator{\eta_{\theta}(x) \geq 1/2}$ is the associated Bayes classifier.

But, for all $\theta \in \Theta$ and any $\delta \in (0,1/4)$ (to be chosen later), we have 
\begin{align}
\E_{\theta}\!\left[ \big|2 \eta_{\theta} (X) - 1 \big | \indicator{\hat{\Phi}(X) \neq \Phi_{\theta}(X)} \right] & \geq \delta \, \PR_{\theta}\!\left( \bigl\{|2 \eta_{\theta} (X) - 1| \geq \delta\bigr\} \cap \bigl\{\hat{\Phi}(X) \neq \Phi_{\theta}(X)\bigr\} \right) \nonumber \\
& \geq \delta \left( \PR_{\theta}\bigl(\hat{\Phi}(X) \neq \Phi_{\theta}(X) \bigr) - \PR_{\theta}\bigl(|2 \eta_{\theta} (X) - 1| < \delta\bigr) \right) \nonumber \\
& \geq \delta \left( \PR_{\theta}\bigl(\hat{\Phi}(X) \neq \Phi_{\theta}(X) \bigr) - \frac{5 \delta}{\Delta} \right), \label{eq:minimaxLB-1}
\end{align}
where the last inequality follows from Proposition~\ref{thm:margin}.
Next, we use a conditional argument to handle the probability above given the training sample $(X_i,Y_i)_{1 \leq i \leq n}$: the process $X=(X(t))_{0 \leq t \leq 1}$ defined in \eqref{eq:model} is independent from the training sample and has distribution $(Q_0+Q_{\theta})/2$ under $\PR_{\theta}$ (recall that $Q_{\theta}$ denotes the distribution of the process $(Z_t)_{0 \leq t \leq 1}$ defined by $d Z_t = f_{\theta}(t) d t + d W(t)$). Therefore, for all $\theta \in \Theta$,
\begin{align}
\PR_{\theta}\Bigl( \hat{\Phi}(X) \neq \Phi_{\theta}(X) \Bigr) & = \E_{\theta} \biggl\{ \PR_{\theta}\Bigl( \hat{\Phi}(X) \neq \Phi_{\theta}(X) \, \Big| \, (X_i,Y_i)_{1 \leq i \leq n} \Bigr) \biggr\} \nonumber \\
& = \E_{\theta} \biggl\{ \int \indicator{\hat{\Phi}(x) \neq \Phi_{\theta}(x)} \frac{\dd Q_0(x) + \dd Q_{\theta}(x)}{2} \biggr\} \nonumber \\
& \geq \frac{1}{2} \, \E_{\theta}\Bigl[ \big\|\hat{\Phi}-\Phi_{\theta}\big\|_{L^1(Q_0)} \Bigr] \,, \label{eq:minimaxLB-2}
\end{align}
where the last inequality follows from the fact that $\indicator{\hat{\Phi}(x) \neq \Phi_{\theta}(x)} = \big| \hat{\Phi}(x) - \Phi_{\theta}(x) \big|$ for all continuous functions $x:[0,1]\to~\R$. Putting~\eqref{eq:minimaxLB-Bayesformula}, \eqref{eq:minimaxLB-1}, and~\eqref{eq:minimaxLB-2} together, we finally get
\begin{equation}
\sup_{\substack{f,g \in \Hs(R) \\ \|f-g\| \geq \Delta}} \left\{ \mathcal{R}_{f,g}(\hat\Phi) - \inf_{\Phi} \mathcal{R}_{f,g}(\Phi)  \right\} \geq \frac{\delta}{2} \, \left( \sup_{\theta \in \Theta}  \E_{\theta}\Bigl[ \big\|\hat{\Phi}-\Phi_{\theta}\big\|_{L^1(Q_0)} \Bigr] - \frac{10 \delta}{\Delta} \right) \,.
\label{eq:minimaxLB-3}
\end{equation}

\paragraphnospace{Step 2: a key combinatorial and geometrical argument}
In order to further bound \eqref{eq:minimaxLB-3} from below, we now specialize $\Theta$ to the set given by Lemma~\ref{lem:lowerbound-packing} in Appendix~\ref{sec:proof-lowerbound-packing}, whose proof combines Varshamov-Gilbert's lemma with simple but key geometrical arguments in dimension two. More precisely, we use Lemma~\ref{lem:lowerbound-packing} in Appendix~\ref{sec:proof-lowerbound-packing} with $\epsilon=c/\sqrt{n}$ and $d = \big\lfloor \bigl((R^2-\Delta^2) \, n\bigr)^{1/(2s+1)} \big\rfloor$, for some absolute constant $c \in (0,1]$ to be determined later. Two remarks are in order:
\begin{itemize}
	\item We have $d \geq \left((R^2-\Delta^2) n \right)^{1/(2s+1)}-1 \geq 
	32 \log(2) + 1$ by the assumption $n \geq (32 \log(2) + 2)^{2s+1}/(3 R^2/4) \geq (32 \log(2)+2)^{2s+1}/(R^2 - \Delta^2)$ since $\Delta \leq R/2$. In particular the condition $d \geq 7$ in Lemma~\ref{lem:lowerbound-packing} holds true.
	\item The condition $\Delta \geq \sqrt{d} \, \epsilon$ of Lemma~\ref{lem:lowerbound-packing} holds since by assumption on $\Delta$, we have $$\Delta \geq R^{1/(2s+1)}\, n^{-s/(2s+1)} = \sqrt{(R^2 \, n)^{1/(2s+1)}/n} \geq \sqrt{d/n} \geq \sqrt{d} \, \epsilon \,,$$ by definition of $d$ and $\epsilon$.
\end{itemize}

\noindent
We can thus apply Lemma~\ref{lem:lowerbound-packing}  and find a subset $\Theta \subseteq \{\Delta\} \times \{-\epsilon, \, \epsilon \}^{d-1} \subseteq \R^d$ of cardinality $|\Theta| \geq e^{(d-1)/8} \geq 2$ such that, for all $ \theta \neq \theta' \in \Theta$,
\begin{equation}
\label{eq:lowerbound-packingbis}
\big\|\Phi_{\theta}-\Phi_{\theta'}\big\|_{L^1(Q_0)} \geq \frac{\sqrt{d-1} \, \epsilon}{4 \pi \Delta} \, e^{-\Delta^2}\,.
\end{equation}

\ \\
Note that our construction of $\Theta$ meets our earlier requirement: for all $\theta \in \Theta$, we have $\sum_{j=2}^d  \theta_j^2 j^{2s} \leq (d-1)\epsilon^2 \, d^{2s} \leq d^{2s+1} \epsilon^2 \leq R^2 -\Delta^2$ by definition of $d \leq \bigl((R^2-\Delta^2) \, n\bigr)^{1/(2s+1)}$ and $\epsilon \leq 1/\sqrt{n}$. Therefore, $\Theta \subseteq \bigl\{ \theta \in \R^d: \,  \theta_1 = \Delta \textrm{ and } \sum_{j=2}^d  \theta_j^2 j^{2s} \leq R^2 - \Delta^2\bigr\}$ as assumed at the beginning of this proof.

\paragraphnospace{Step 3: Reduction to a testing problem with finitely-many hypotheses} We now use a classical tool in nonparametric statistics since we reduce the problem to a multiple-hypotheses testing problem. More precisely, using~\eqref{eq:minimaxLB-3} and setting 
$$\hat{\theta} \in \argmin_{\theta \in \Theta} \big\|\hat{\Phi}-\Phi_{\theta}\big\|_{L^1(Q_0)}\,,$$
we can see that
\begin{align}
\sup_{\substack{f,g \in \Hs(R) \\ \|f-g\| \geq \Delta}} \left\{ \mathcal{R}_{f,g}(\hat\Phi) - \inf_{\Phi} \mathcal{R}_{f,g}(\Phi)  \right\} & \geq \frac{\delta}{2} \left( \sup_{\theta \in \Theta}  \E_{\theta}\Bigl[ \indicator{\{\hat{\theta} \neq \theta\}} \big\|\hat{\Phi}-\Phi_{\theta}\big\|_{L^1(Q_{\mu})} \Bigr] - \frac{10 \delta}{\Delta} \right) \nonumber \\
& \geq \frac{\delta}{2} \left( \frac{\sqrt{d-1} \, \epsilon}{8 \pi \Delta} \, e^{-\Delta^2} \sup_{\theta \in \Theta} \PR_{\theta}\Bigl( \hat{\theta} \neq \theta \Bigr) - \frac{10 \delta}{\Delta} \right) \,, \label{eq:minimaxLB-4}
\end{align}
where in the last inequality we used the fact that, on the event $\{\hat{\theta} \neq \theta\}$, we necessarily have
\[
\big\|\hat{\Phi}-\Phi_{\theta}\big\|_{L^1(Q_0)} \geq \frac{\sqrt{d-1} \, \epsilon}{8 \pi \Delta} \, e^{-\Delta^2}
\]
by a combination of Inequality~\eqref{eq:lowerbound-packingbis}, the definition of $\hat{\theta}$, and the triangle inequality.

\ \\
We now lower bound the worst-case testing error $\sup_{\theta \in \Theta} \PR_{\theta}\bigl( \hat{\theta} \neq \theta \bigr)$. Since $\hat{\theta}$ only depends on the training sample $(X_i,Y_i)_{1 \leq i \leq n}$, whose distribution we denote by $P_{\theta}$, we can write $\PR_{\theta}\bigl( \hat{\theta} \neq \theta \bigr) = P_{\theta}\bigl( \hat{\theta} \neq \theta \bigr)$. We can thus use Fano's inequality (cf.~Lemma~\ref{lem:Fano} in Appendix~\ref{s:C}) with the events $A_{\theta} = \bigl\{\hat{\theta} = \theta\bigr\}$, the distributions $P_{\theta}$, $\theta \in \Theta$, and the reference distribution $\Q=P_{\theta_0}$, where $\theta_0 \eqdef (\Delta,0,\ldots,0) \in \R^d$. We obtain:
\begin{equation}
\inf_{\theta \in \Theta} P_{\theta}\bigl( \hat{\theta} = \theta \bigr) \leq \frac{1}{|\Theta|} \sum_{\theta \in \Theta} P_{\theta}\bigl( \hat{\theta} = \theta \bigr) \leq \frac{\displaystyle \frac{1}{|\Theta|} \sum_{\theta \in \Theta} \KL\bigl(P_{\theta},P_{\theta_0}\bigr) + \log 2}{\log |\Theta|} \;.
\label{eq:FanoCsq}
\end{equation}
Using the chain rule for the Kullback-Leibler divergence, and following similar computations as in Section~\ref{s:model} (application of Girsanov's formula), we can see that, for all $\theta \in \Theta$,
\begin{align*}
\KL\bigl(P_{\theta},P_{\theta_0}\bigr) & = n \left(\KL\bigl(\Ber(1/2),\Ber(1/2)\bigr) + \frac{\KL(Q_{\theta},Q_{\theta_0}) + \KL(Q_0,Q_0)}{2} \right) = \frac{n \|\theta - \theta_0\|^2}{4} = \frac{n (d-1) \epsilon^2}{4} \,,
\end{align*}
where we used the fact that $\theta \in \Theta \subseteq \{\Delta\} \times \{-\epsilon, \, \epsilon \}^{d-1}$ and $\theta_0 \eqdef (\Delta,0,\ldots,0)$. Combining \eqref{eq:FanoCsq} with the Kullback-Leibler upper bound above, and recalling that $|\Theta| \geq e^{(d-1)/8}$, we get
\[
\inf_{\theta \in \Theta} P_{\theta}\bigl( \hat{\theta} = \theta \bigr) \leq \frac{n (d-1) \epsilon^2/4 + \log 2}{(d-1)/8}  \leq 2 c^2 + \frac{1}{4} \,,
\]
where the last inequality follows from $\epsilon = c/\sqrt{n}$ and $d \geq 32 \log(2) + 1$. As a consequence, choosing $c \eqdef 1/(2 \sqrt{2})$,
\[
\sup_{\theta \in \Theta} P_{\theta}\bigl( \hat{\theta} \neq \theta \bigr) \geq 1 - 2 c^2 - \frac{1}{4} = \frac{1}{2} \;.
\]
Plugging the last lower bound into~\eqref{eq:minimaxLB-4}, we finally get
\[
\sup_{\substack{f,g \in \Hs(R) \\ \|f-g\| \geq \Delta}} \left\{ \mathcal{R}_{f,g}(\hat\Phi) - \inf_{\Phi} \mathcal{R}_{f,g}(\Phi)  \right\}
\geq \frac{5 \delta}{\Delta} \left( \frac{\sqrt{d-1} \, \epsilon}{160 \pi} \, e^{-\Delta^2} - \delta \right) = \frac{(d-1) \, \epsilon^2}{20480 \pi^2 \Delta} \, e^{-2\Delta^2}
\]
with the particular choice of $\delta = \sqrt{d-1} \, \epsilon \, e^{-\Delta^2} / (320 \pi)$. We conclude the proof by substituting the values of $\epsilon=c/\sqrt{n}$ and $d -1 = \big\lfloor \bigl((R^2-\Delta^2) \, n\bigr)^{1/(2s+1)} \big\rfloor -1 \geq (6/8)  \bigl((R^2-\Delta^2) \, n\bigr)^{1/(2s+1)}$ (since $\lfloor x \rfloor - 1 \geq 6 x/8$ for all $x \geq 7$) and by using the fact that $R^2-\Delta^2 \geq 3 R^2/4$ (since $\Delta \leq R/2$). Note also that, by the assumption $n \geq R^{1/s}$, we have $\delta < 1/4$ as required in the analysis. This concludes the proof of Theorem~\ref{thm:minimaxlowerbound}.

\subsection{A key combinatorial and geometrical lemma}
\label{sec:proof-lowerbound-packing}

In this section, we 
 provide a key combinatorial and geometrical lemma to derive the minimax lower bound of Theorem~\ref{thm:minimaxlowerbound}. Indeed, the next result guarantees the existence of a parameter set $\Theta \subset \R^d$ such that---when $\epsilon$ is chosen small enough---it is statistically hard to estimate the true value of the parameter $\theta \in \Theta$, while all Bayes classifiers $\Phi_{\theta}$ and $\Phi_{\theta'}$, $\theta \neq \theta' \in \Theta$, are sufficiently far from one another, thus leading to a large classification excess risk.

\begin{lem}
\label{lem:lowerbound-packing}
Let $d \geq 7$, $\epsilon > 0$, and  $\Delta \geq \sqrt{d} \, \epsilon$. There exists a subset $\Theta \subseteq \{\Delta\} \times \{-\epsilon, \, \epsilon \}^{d-1} \subseteq \R^d$ of cardinality $|\Theta| \geq e^{(d-1)/8} \geq 2$ such that, for all $\theta \neq \theta' \in \Theta$,
\begin{equation}
\label{eq:lowerbound-packing}
\big\|\Phi_{\theta}-\Phi_{\theta'}\big\|_{L^1(Q_0)} \geq \frac{\sqrt{d-1} \, \epsilon}{4 \pi \Delta} \, e^{-\Delta^2}\,,
\end{equation}
where $Q_0$ denotes the distribution of a standard Brownian motion $W=(W(t))_{0 \leq t \leq 1}$ on $[0,1]$, and where $\|h\|_{L^1(Q_0)} \eqdef \E\bigl[|h(W)|\bigr]$.
\end{lem}

\ \\
The proof is provided in Section~\ref{sec:proofCombinatorialLemma} below. We first state three intermediary results.

\subsubsection{Intermediary results}

The following lemma shows that, for the $d$-dimensional construction of Section~\ref{sec:LBproof-main} (Step~1), the Bayes classifier $\Phi_{\theta}$ only depends on the $d$ random variables $\tilde{X}_j \eqdef \int_0^1 \varphi_j(t) d X(t)$, $1 \leq j \leq d$, and takes the form of a simple linear classifier in $\R^d$. We recall that $(\varphi_j)_{j \geq 1}$ is any Hilbert basis of $\mathbb{L}^2([0,1])$ and that $f_{\theta} = \sum_{j=1}^d \theta_j \varphi_j$.

\begin{lem}
\label{lem:BayesClassifierFinite}
Consider the statistical construction of Section~\ref{sec:LBproof-main} (Step~1). Let $W=(W(t))_{0 \leq t \leq 1}$ be a standard Brownian motion and define $\tilde{W}_j \eqdef \int_0^1 \varphi_j(t) d W(t)$ as well as $\tilde{W} \eqdef \bigl(\tilde{W}_j\bigr)_{1 \leq j \leq d} \in \R^d$. Then, the Bayes classifier $\Phi_{\theta} = \mathds{1}_{\eta_{_{\theta}} \geq 1/2}$ satisfies
\[
\Phi_{\theta}(W) = \left\lbrace
\begin{array}{lcr}
0 & \mathrm{if} &  \| \tilde{W} - \theta \| > \| \tilde{W} \|  \\
1 & \mathrm{if} & \| \tilde{W} - \theta \| \leq \| \tilde{W} \| 
\end{array}
\right. \quad \textrm{almost surely}.
\]
\end{lem}

\begin{proof}
The result follows directly from the calculations of Section~\ref{sec:BayesRule} (application of Girsanov's formula). Indeed, using \eqref{eq:eta} and the fact that $g=0$ and $\|f_{\theta}\|=\|\theta\|$, we obtain
\begin{align*}
\eta_{\theta}(W) \geq 1/2 & \iff \int_0^1 f_{\theta}(t) d W(t)  \geq \frac{\|f_{\theta}\|^2}{2}  \\
& \iff \tilde{\theta} \cdot \tilde{W} \geq \frac{\|\theta\|^2}{2} \\
& \iff \| \tilde{W} - \theta \|^2 \leq \| \tilde{W} \|^2 \;,
\end{align*}
which concludes the proof.
\end{proof}

The above lemma shows that the Bayes classifier $\Phi_{\theta}$ corresponds to a linear classifier in $\R^d$ (after projecting onto $(\varphi_j)_{1 \leq j \leq d}$). The next lemma provides a lower bound on the angle between the hyperplanes associated with two linear classifiers $\Phi_{\theta}$ and $\Phi_{\theta'}$, for $\theta \neq \theta' \in \Theta$. This result will be crucial in our proof of the lower bound of Lemma~\ref{lem:lowerbound-packing}.

We recall that the (undirected) internal angle  between two non-zero vectors $\theta,\theta' \in \R^d$ is given by
\[
\angle(\theta, \, \theta') := \arccos\left(\frac{\langle \theta, \theta' \rangle }{\|\theta\| \, \|\theta'\|}\right)\in [0,\pi] \;;
\]
this angle is in particular well defined for all $\theta,\theta' \in \Theta$ (since $0 \notin \Theta$ by construction).

\begin{lem}
\label{lem:BayesAngle}
Let $d \geq 7$, $\epsilon > 0$, and  $\Delta \geq \sqrt{d} \, \epsilon$. Let $\Gamma \subseteq \{-1,1\}^{d-1}$ be a set provided by Varshamov-Gilbert's lemma in dimension $m=d-1$ (see, \eg, Lemma~\ref{lem:VG} in Appendix~\ref{s:C}), and define
\begin{equation}
\Theta \eqdef \bigl\{\Delta\bigr\} \times \bigl(\epsilon \Gamma\bigr) = \Bigl\{ (\Delta,\epsilon u_1, \epsilon u_2, \ldots, \epsilon u_{d-1}) : \, (u_1,\ldots,u_{d-1}) \in \Gamma \Bigr\} \subset \R^d \,.
\label{eq:def-Theta}
\end{equation}
Then, for all $\theta \neq \theta' \in \Theta$, the internal angle $\angle(\theta, \, \theta')$ between the vectors $\theta$ and $\theta'$ is bounded by
\[
\frac{\sqrt{d-1} \, \epsilon}{2\Delta} \leq \angle(\theta, \, \theta') \leq \frac{\pi}{2} \;.
\]
\end{lem}

\begin{proof} 
Let  $\theta \neq \theta' \in \Theta$. By~\eqref{eq:def-Theta} we can write $\theta  = (\Delta,\epsilon u_1, \ldots, \epsilon u_{d-1})$ and $\theta'  = (\Delta,\epsilon u'_1, \ldots, \epsilon u'_{d-1})$ with $u \neq u' \in \Gamma$. We also set $m=d-1$. We have
\begin{align}
\cos \Big(\angle(\theta, \, \theta')\Bigr)  & = \frac{\langle \theta , \theta' \rangle }{\|\theta\| \, \|\theta'\|} = \frac{\Delta^2 + \epsilon^2 \sum_{j=1}^m u_j u'_j}{\sqrt{\Delta^2 + m \epsilon^2} \, \sqrt{\Delta^2 + m \epsilon^2}} = \frac{\Delta^2 + \epsilon^2 \sum_{j=1}^m u_j u'_j}{\Delta^2 + m \epsilon^2} \;. \label{eq:angleLB-1}
\end{align}
Note that $u_j u'_j \in \{-1,1\}$ so that $\Delta^2 + \epsilon^2 \sum_{j=1}^m u_j u'_j \geq \Delta^2 - m \epsilon^2 \geq 0$ because we assumed that $\Delta \geq \sqrt{d} \, \epsilon$. Therefore, $\cos\big(\angle(\theta, \, \theta')\bigr) \geq 0$, which in turn entails that $\angle(\theta, \, \theta') \leq \pi/2$ since $\angle(\theta, \, \theta') \in [0,\pi]$ by definition.

\ \\
We now prove the lower bound on $\angle(\theta, \, \theta')$. By construction of $\Gamma$ (Lemma~\ref{lem:VG} in Appendix~\ref{s:C}), we have $u_j u'_j \in \{-1,1\}$ and $\sum_{j=1}^m \indicator{\{u_j \neq u'_j\}} \geq m/4$, so that $\sum_{j=1}^m u_j u'_j \leq -m/4 + 3m/4 = m/2$. Substituting this upper bound in~\eqref{eq:angleLB-1} yields
\[
\cos \Big(\angle(\theta, \, \theta')\Bigr)  \leq \frac{\Delta^2 + m \epsilon^2 /2}{\Delta^2 + m \epsilon^2} = 1 - \frac{m \epsilon^2 /2}{\Delta^2 + m \epsilon^2} \;.
\]
Using the former result $\cos\big(\angle(\theta, \, \theta')\bigr) \geq 0$ and the last inequality above, we obtain
\[
\sin^2\!\Big(\angle(\theta, \, \theta')\Bigr)  =  1 - \cos^2\!\Big(\angle(\theta, \, \theta')\Bigr)  \geq 1 - \cos\!\Big(\angle(\theta, \, \theta')\Bigr) \geq \frac{m \epsilon^2 /2}{\Delta^2 + m \epsilon^2} \geq \frac{m \epsilon^2}{4 \Delta^2} \,,
\]
where we again used $m = d-1 \leq d$ and our assumption on $\Delta$: $\sqrt{m}\, \epsilon \leq \sqrt{d}\, \epsilon \leq \Delta $. We conclude the proof by noting that $\angle(\theta, \, \theta') \geq \sin \bigl(\angle(\theta, \, \theta')\bigr) = \sqrt{\sin^2\bigl(\angle(\theta, \, \theta')\bigr)}$ since $\angle(\theta, \, \theta') \in [0,\pi]$: 
$$
\angle(\theta, \, \theta') \geq \frac{\sqrt{m} \epsilon}{2 \Delta} = \frac{\sqrt{d-1} \epsilon}{2 \Delta} \;.
$$
\end{proof}

Our third and last lemma in this subsection provides a lower bound on the Gaussian measure of a double cone in dimension~2. We say that $\cC \subset \R^2$ is an \emph{open double cone with apex $z \in \R^2$} if it is of the form
\[
\cC = \Bigl\{ z + a u + b v : \, (a,b) \in \R^{\star 2}_+ \cup\, \R^{\star 2}_- \Bigr\}
\]
for some linearly independent vectors $u, v \in \R^2$. It is clear that there is not a one-to-one correspondence between $(u,v)$ and $\cC$ (several pairs $(u,v)$ correspond to the same $\cC$). However, the value of the internal angle $\angle(u,v) := \arccos\bigl(\langle u , v\rangle/ (\|u\| \, \|v\|)\bigr) \in (0,\pi)$ between $u$ and $v$ is the same for all pairs $(u,v)$ that correspond to $\cC$. We thus call $\angle(u,v)$ the \emph{angle of the open double cone $\cC$}.

\begin{lem}
\label{lem:GaussianMassCone}
Let $\cC \subset \R^2$ be an open double cone with apex $z \in \R^2$ and angle $\cA \in (0,\pi)$. Then, the measure of $\cC$ with respect to the standard Gaussian distribution $\gamma_2 = \cN(0,\Id_{2 \times 2})$ on $\R^2$ is lower bounded by
\[
\gamma_2(\cC) \geq \frac{\cA}{2 \pi} \, e^{-\| z \|^2} \,.
\]
\end{lem}

\ \\
We emphasize that rather intuitively, the above lower bound is proportional to the angle $\cA$ and decreases exponentially fast with $\|z\|^2$. (The constant of $1$ appearing in the exponential could certainly be optimized, but this one is sufficient for our purposes.)

\begin{proof}
We carry out a change of variables by a translation around $z$: writing $\cC - z = \bigl\{x - z: \, x \in \cC \bigr\}$ and using the inequality $\|z+u\|^2 \leq 2 \|z\|^2 + 2 \|u\|^2$, we get
\begin{align*}
\gamma_2(\cC) & = \frac{1}{2 \pi}\int_{\cC} e^{-\|x\|^2/2} \, \dd x = \frac{1}{2 \pi}\int_{\cC-z} e^{-\|z+u\|^2 /2} \, \dd u \geq \frac{e^{-\|z\|^2}}{2 \pi}\int_{\cC-z} e^{-\|u\|^2} \dd u \\
& = \frac{e^{-\|z\|^2}}{2 \pi} \, 2 \int_0^{\cA} \left( \int_0^{+\infty} r e^{-r^2} \dd r \right) \dd \alpha = \frac{e^{-\|z\|^2}}{2 \pi} \, \cA \,,
\end{align*}
where the second line is obtained by parameterizing $\cC-z$
with polar coordinates and by noting that $\cC-z$ is an open double cone of angle $\cA$ pointed at the origin. This concludes the proof.
\end{proof}

\subsubsection{Proof of Lemma~\ref{lem:lowerbound-packing}}
\label{sec:proofCombinatorialLemma}

We now prove Lemma~\ref{lem:lowerbound-packing} using the intermediary results of the previous subsection. We use the same notation as in Section~\ref{sec:LBproof-main}.
 Let $\Gamma \subseteq \{-1,1\}^{d-1}$ be a set provided by Varshamov-Gilbert's lemma in dimension $m=d-1$ (cf.~Lemma~\ref{lem:VG} in Appendix~\ref{s:C}). Next we show that the set
\begin{equation*}
\Theta \eqdef \bigl\{\Delta\bigr\} \times \bigl(\epsilon \Gamma\bigr) = \Bigl\{ (\Delta,\epsilon u_1, \epsilon u_2, \ldots, \epsilon u_{d-1}) : \, (u_1,\ldots,u_{d-1}) \in \Gamma \Bigr\} \subset \R^d
\end{equation*}
satisfies the statement of Lemma~\ref{lem:lowerbound-packing}. We can already see that its cardinality is $|\Theta| = |\Gamma| \geq e^{m/8} \geq e^{(d-1)/8}$. It remains to prove that, for all $\theta \neq \theta' \in \Theta$,
\begin{equation}
\big\|\Phi_{\theta}-\Phi_{\theta'}\big\|_{L^1(Q_0)} \geq \frac{\sqrt{d-1} \, \epsilon}{4 \pi \Delta} \, e^{-\Delta^2}\,,
\label{eq:lowerbound-packing2}
\end{equation}
where $Q_0$ denotes the distribution of a standard Brownian motion $W=(W(t))_{0 \leq t \leq 1}$ on $[0,1]$, and where $\|h\|_{L^1(Q_0)} \eqdef \E\bigl[|h(W)|\bigr]$.

\begin{proof}[Proof of~\eqref{eq:lowerbound-packing2}]
Let $\theta \neq \theta' \in \Theta$. Let $W=(W(t))_{0 \leq t \leq 1}$ be a standard Brownian motion on some probability space $(\Omega,\mathcal{F},\PR)$. Noting that $\big| \Phi_{\theta}(W) - \Phi_{\theta'}(W) \big|=\indicator{\Phi_{\theta}(W) \neq \Phi_{\theta'}(W)}$ a.s., we have
\begin{align}
\big\|\Phi_{\theta}-\Phi_{\theta'}\big\|_{L^1(Q_0)}  & = \PR\bigl(\Phi_{\theta}(W) \neq \Phi_{\theta'}(W) \bigr) \nonumber \\
& = \PR\Bigl( \bigl\{\| \tilde{W} - \theta \| \leq \| \tilde{W} \| < \| \tilde{W} - \theta' \| \bigr\} \cup \bigl\{\| \tilde{W} - \theta' \| \leq \| \tilde{W} \| < \| \tilde{W} - \theta \| \bigr\} \Bigr) \nonumber \\
& \geq \PR\Bigl( \, \underbrace{\bigl\{\| \tilde{W} - \theta \| < \| \tilde{W} \| < \| \tilde{W} - \theta' \| \bigr\} \cup \bigl\{\| \tilde{W} - \theta' \| < \| \tilde{W} \| < \| \tilde{W} - \theta \| \bigr\}}_{=:A} \, \Bigr) \,, \nonumber
\end{align}
where the line before last follows from Lemma~\ref{lem:BayesClassifierFinite}, and where we recall that $\tilde{W} \eqdef \bigl(\tilde{W}_j\bigr)_{1 \leq j \leq d} \in \R^d$ with $\tilde{W}_j \eqdef \int_0^1 \varphi_j(t) d W(t)$. In order to bound $\PR(A)$ from below, we project (orthogonally) all points in $\R^d$ onto the unique plane $\cP$ that contains $0$ and the non-colinear vectors $\theta$ and $\theta'$ (note from Lemma~\ref{lem:BayesAngle} that $0 < \angle(\theta,\theta') \leq  \pi/2 < \pi$). As shown in Figure~\ref{fig:2Dprojection}, we define $z \in \cP$ as the intersection between the perpendicular bisectors $\cD$ and $\cD'$ of the segments $[0,\theta]$ and $[0,\theta']$ on the plane $\cP$. Writing $r_{-\pi/2}$ for the rotation of angle $-\pi/2$ on the plane $\cP$, we also consider the unit vectors $u = r_{-\pi/2}\bigl(\theta/\|\theta\|\bigr)$ and $v = r_{-\pi/2}\bigl(\theta'/\|\theta'\|\bigr)$ that support the lines $\cD$ and $\cD'$ respectively. 
\begin{figure}[h!]
\begin{center}
\begin{tikzpicture}[scale=1] 
\filldraw [gray!25,draw=black] (0,0.525)--(-2,-0.55)--(-3,2.2)--cycle; 
\filldraw [gray!25,draw=black] (0,0.525)--(2,-0.55)--(3,2.2)--cycle; 
\draw [->,-latex,dashed] (0,-2)--++(0,4) node [above] {$y$}; 
\draw [->,-latex] (0,-2)--++(-2,4) node [above] {$\theta$}; 
\draw [->,-latex] (0,-2)--++(2,4) node [above] {$\theta'$}; 
\draw [ |<->|] (-2,2) -- (0,2) node [midway, above] {$L$}; 
\draw [ |<->|] (2,2) -- (0,2) node [midway, above] {$L$}; 
\node at (0,-2) [below] {$0$}; 
\node at (-1,0) [below ] {$\frac{\theta}{2}$};
\node at (1,0) [below ] {$\frac{\theta'}{2}$};
\draw (-1,0) to (5,3.3);
\draw (-1,0) to (-2.5,-0.825);
\draw (1,0) to (-5,3.3);
\draw (1,0) to (2.5,-0.825);
\node at (0,0.2) [right] {$z$};
\draw (-0.59,0.18) arc (22:109:0.5);
\draw (1.20,0.38) arc (65:145:0.5);
\node at (1.03,0.38) [above] {\small $\pi/2$}; 
\node at (-1.03,0.38) [above] {\small $\pi/2$}; 
\node at (2,1) [below] {$\mathcal{C}$}; 
\node at (3.4,2.8) [below] {$\mathcal{D}$}; 
\node at (2.5,-0.8) [below] {$\mathcal{D}'$}; 
\draw [red, very thick, -stealth] (0,0.55) --++ (0.7,0.36) node [above] {$u$};
\draw [red, very thick, -stealth] (0,0.55) --++ (0.7,-0.36) node [below] {$v$};
\end{tikzpicture}
\end{center}
\caption{\label{fig:2Dprojection} The main objects of interest on the plane $\cP$.}
\end{figure}
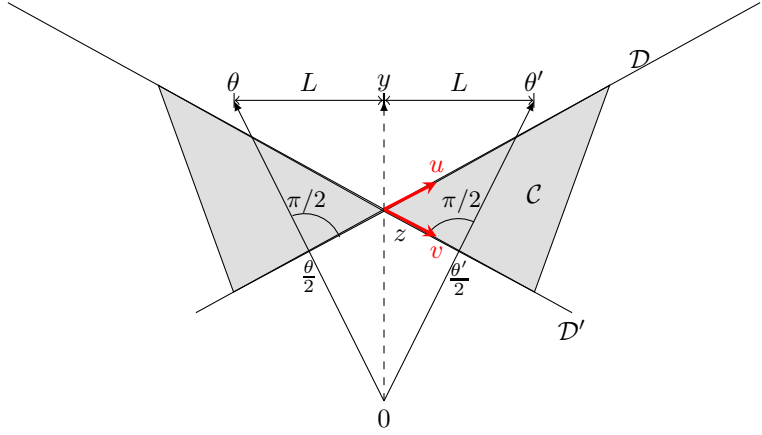
%
%

Writing $\tilde{W}_{\cP}$ for the orthogonal projection of $\tilde{W} \in \R^d$ onto $\cP$, we can see that
\[
\PR(A) = \PR\Bigl(\tilde{W}_{\cP} \in \cC \Bigr) \quad \textrm{with} \quad \cC \eqdef \Bigl\{ z + a u + b v : \, (a,b) \in \R_+^{*2} \cup\, \R_-^{*2} \Bigr\} \,.
\]

Let $(e_1,e_2)$ be any orthonormal basis of $\cP$. Decomposing any $w \in \cP$ as $w = w^1 e_1 + w^2 e_2$ (and similarly for $u$ and $v$), we can see that
\[
w \in \cC \iff (w^1,w^2) \in \underbrace{\Bigl\{ (z^1,z^2) + a (u^1,u^2) + b (v^1,v^2) : \, (a,b) \in \R_+^{*2} \cup\, \R_-^{*2} \Bigr\}}_{=:\tilde{\cC}} \,.
\]
Therefore,
\[
\PR(A) =
\PR\Bigl(\bigl(\tilde{W}_{\cP}^1,\tilde{W}_{\cP}^2\bigr) \in \tilde{\cC} \, \Bigr) = \gamma_2\bigl(\tilde{\cC} \,\bigr) \,,
\]
where $\gamma_2 = \cN(0,\Id_{2 \times 2})$ denotes the standard Gaussian distribution on $\R^2$. The last equality holds true because $W=(W(t))_{0 \leq t \leq 1}$ is a standard Brownian motion so that the $\tilde{W}_j = \int_0^1 \varphi_j(t) dW(t)$, $1 \leq j \leq d$, are independent $\cN(0,1)$ random variables (because the $\varphi_j$ are orthonormal), so that $\bigl(\tilde{W}^1,\tilde{W}^2\bigr)$ is a standard two-dimensional Gaussian vector (because $e_1$ and $e_2$ are orthonormal).

Now, we note that the subset $\tilde{\cC} \subset \R^2$ is an open double cone with apex $(z_1,z_2)$. Since $(e_1,e_2)$ is an orthonormal basis of $\cP$, the angle of $\tilde{\cC}$ is equal to $\angle(u,v) = \angle\bigl(r_{-\pi/2}\bigl(\theta/\|\theta\|\bigr),r_{-\pi/2}\bigl(\theta'/\|\theta'\|\bigr)\bigr) = \angle(\theta,\theta')$. Therefore, applying Lemma~\ref{lem:GaussianMassCone} and then Lemma~\ref{lem:BayesAngle},
\begin{align}
\PR(A) \
& \geq \frac{\angle(\theta,\theta')}{2 \pi} \, e^{-(z_1^2+z_2^2)} \geq \frac{e^{-(z_1^2+z_2^2)} \sqrt{d-1} \, \epsilon}{4 \pi \Delta} \,. \label{eq:lowerbound-packing4}
\end{align}
We conclude the proof by upper bounding $z_1^2 + z_2^2 = \|z\|^2$ as follows. First note from Figure~\ref{fig:2Dprojection} that
\[
\cos\left(\frac{\angle(\theta, \, \theta')}{2}\right) = \frac{\|\theta\|/2}{\|z\|} \qquad \textrm{so that} \qquad \|z\| = \frac{\|\theta\|}{2 \cos\left(\frac{\angle(\theta, \, \theta')}{2}\right)} \;.
\]
But, from the inequality $0 \leq \angle(\theta, \, \theta')/2 \leq \pi/4$ (see Lemma~\ref{lem:BayesAngle}) we get that $\cos\bigl(\angle(\theta, \, \theta')/2\bigr) \geq 1/\sqrt{2}$, so that $\|z\| \leq \|\theta\|/\sqrt{2}$, \ie,
\[
z_1^2 + z_2^2 \leq \frac{\|\theta\|^2}{2} = \frac{\Delta^2+(d-1) \epsilon^2}{2} \leq \Delta^2
\]
by the assumption $\Delta \geq \sqrt{d} \epsilon$. Combining $\|z\|^2 \leq \Delta^2$ with Equation \eqref{eq:lowerbound-packing4}  concludes the proof.
\end{proof}


\subsubsection{Two well-known lemmas}
\label{s:C}

The next combinatorial result is known as Varshamov-Gilbert's lemma. It provides a lower bound on the packing entropy of the $m$-dimensional hypercube $\{-1,1\}^m$ endowed with the Hamming metric, at scale~$m/4$. This result indicates that among the $2^m$ corners of $\{-1,1\}^m$, exponentionally many of them are almost opposite from one another. A proof can be found, \eg, in \cite[Lemma~4.7]{Massart03StFlour}.

\begin{lem}[Varshamov-Gilbert's lemma]
\label{lem:VG}
Let $m \geq 1$. There exists a subset $\Gamma \subseteq \{-1,1\}^m$ of cardinality $|\Gamma| \geq e^{m/8}$ such that
\[
\forall x \neq y \in \Gamma, \quad \sum_{j=1}^m \indicator{\{x_j \neq y_j\}} > \frac{m}{4} \,.
\]
\end{lem}

The next lemma is a well-known version of Fano's inequality that follows, \eg, from \cite[Chapter~VII, Lemma~1.1]{IbHa-81-StatisticalEstimation} or \cite[Theorem~2.11.1]{CoTh06} (see also Proposition~1 in the recent survey \cite{GeMeSt-17-Fano}).

We recall that the Kullback-Leibler divergence $\KL(\PR,\mathbb{Q})$ between two probability distributions $\PR$ and~$\mathbb{Q}$ on the same measurable space $(E,\mathcal{B})$ is defined by
\begin{numcases}{\KL(\PR,\mathbb{Q}) \eqdef}
\nonumber
\int_E \ln \left(\frac{\dd \PR}{\dd \mathbb{Q}}\right) \dd \PR & if $\PR$ is absolutely continuous with respect to $\mathbb{Q}$; \\
\nonumber
+\infty & otherwise.
\end{numcases}

\begin{lem}[Fano's inequality]
\label{lem:Fano}
Let $(E,\mathcal{B})$ be any measurable space and $N \geq 2$. Let $(A_1,\ldots,A_N)$ be a measurable partition of $(E,\mathcal{B})$ and $(\PR_1,\ldots,\PR_N)$ a family of probability distributions on $(E,\mathcal{B})$. Then,
\[
\frac{1}{N} \sum_{i=1}^N \PR_i(A_i) \leq \frac{\displaystyle  \inf_{\Q} \frac{1}{N} \sum_{i=1}^N \KL(\PR_i,\Q) + \log 2}{\log N} \;,
\]
where the infimum is over all probability distributions $\Q$ on $(E,\mathcal{B})$.
\end{lem}

%
%

\section{Truncated nearest neighbor strategy (Theorem \ref{thm:lwr_bound_knn})}

This appendix section gathers the proof of the lower bound of the nearest neighbor method used with a sample-splitting thresholding strategy, \ie, half of the learning sample is used to choose a thresholding dimension $\hat{d}_n$ and then the nearest neighbor classifier is computed on the remaining part of the samples. Therefore, $\hat{d}_n$ is choosen independently from the second part of the samples. 

\subsection{Smoothness of thee Gaussian translation model}
\label{proof:gaussian_alpha}

This paragraph is devoted to the computation of the smoothness index $\beta_d$ involved in the Gaussian translation model in dimension $d \in \mathbb{N}^\star$ (see, \eg, Equation \ref{eq:smooth}). Below, 
$\gamma$ will refer to the density of the $d$-dimensional standard Gaussian random variable and we omit the dependency in $d$ to alleviate the notations.

\begin{proof}[Proof of Proposition \ref{prop:gaussian_alpha}]

According to the definition of the smoothness parameter given in Equation \ref{eq:smooth}, we compute the average value of $\eta$ on a ball $B(x,r)$ and compare it to $\eta(x)$:
\begin{eqnarray}
\lefteqn{\eta(B(x,r))-\eta(x)}\nonumber \\ 
&= &\frac{1}{\mu(B(x,r)} \int_{B(x,r)} \eta(s) d\mu(s) - \frac{\gamma(x)}{\gamma(x)+\gamma(x-m)},\nonumber\\
& = & \frac{2}{\int_{B(x,r)}  \gamma(s) + \gamma(s-m) ds} \int_{B(x,r)} \frac{\gamma(s)}{\gamma(s)+\gamma(s-m)} \frac{1}{2} [\gamma(s)+\gamma(s-m)] ds  - \frac{\gamma(x)}{\gamma(x)+\gamma(x-m)},\nonumber\\
& = & \frac{\gamma(B(x,r))}{\gamma(B(x,r))+\gamma(B(x-m,r))} - \frac{\gamma(x)}{\gamma(x)+\gamma(x-m)}, \nonumber\\
& = & \frac{[\gamma(x)+\gamma(x-m)] \gamma(B(x,r))-\gamma(x) [\gamma(B(x,r))+\gamma(B(x-m,r))]}{[\gamma(x)+\gamma(x-m)] [\gamma(B(x,r))+\gamma(B(x-m,r))]}, \nonumber\\
& = & \frac{\gamma(x-m) \gamma(B(x,r)) - \gamma(x)\gamma(B(x-m,r))}{[\gamma(x)+\gamma(x-m)] [\gamma(B(x,r))+\gamma(B(x-m,r))]}. 
\label{eq:diff_eta}
\end{eqnarray}
It is then necessary to compare $\gamma(B(x,r))$ with $\gamma(x) \lambda(B_r)$ where $\lambda(B_r)$ is the Lebesgue measure of the centered ball of radius $r$ in $\mathbb{R}^d$. For this purpose, we can use the well known convexity inequality on Gaussian measures of shifted balls:
\begin{equation}\label{eq:gaussian_shift}
\exp(-\|x\|^2/2) \gamma(B(0,r)) \leq \gamma(B(x,r)) \leq \gamma(B(0,r)).
\end{equation}
In particular, we have (see \cite{Kuelbs94}) when $r \longrightarrow 0$ that
$$
\gamma(B(x,r)) \sim \exp(-\|x\|^2/2) \gamma(B(0,r)),
$$
but the r.h.s. of \eqref{eq:gaussian_shift} is tight only for $x$ close to $0$.
Expanding the denominator of \eqref{eq:diff_eta}, we obtain that
\begin{eqnarray}
\lefteqn{\left| \eta(B(x,r))-\eta(x) \right|} \nonumber \\
& =& \frac{\left| \gamma(x-m) \gamma(B(x,r)) - \gamma(x) \gamma(B(x-m,r))\right|}{\gamma(x) \gamma(B(x,r))+\gamma(x)\gamma(B(x-m,r))+\gamma(x-m)\gamma(B(x,r))+\gamma(x-m)\gamma(B(x-m,r))} \nonumber\\
& \leq & \frac{\left| \gamma(x-m) \gamma(B(x,r)) - \gamma(x) \gamma(B(x-m,r))\right|}{\gamma(x)\gamma(B(x-m,r))+\gamma(x-m)\gamma(B(x,r))}.
\label{eq:up_KLL}
\end{eqnarray}
Concerning the numerator, a simple change of variable leads to
\begin{eqnarray*}
\lefteqn{\gamma(x-m) \gamma(B(x,r)) - \gamma(x) \gamma(B(x-m,r)) }\\
& = & (2\pi)^{-d}\int_{B(0,r)} \left\lbrace
 e^{-\|x-m\|^2/2} e^{-\|x-s\|^2/2} - e^{-\|x\|^2/2} e^{-\|x-m-s\|^2/2} \right\rbrace ds.
\end{eqnarray*}
For all $x\in \mathbb{R}^d$ and $s\in B(0,r)$, the term inside the integral above  may be written as
$$
e^{-\|x-m\|^2/2} e^{-\|x-s\|^2/2} - e^{-\|x\|^2/2} e^{-\|x-m-s\|^2/2} = e^{-\|x-m\|^2/2-\|x\|^2/2} e^{-\|s\|^2/2}  \left[ e^{\langle x,s\rangle} - e^{\langle x-m,s\rangle}\right].$$
We can use the following upper bound for any real value $a$:
$$
|e^a-1-a|\leq \frac{a^{2}e^{|a|}}{2},
$$
with $a=\langle x,s\rangle$ and $a=\langle x-m,s\rangle$ and deduce that
$$
|e^{\langle x,s\rangle} - e^{\langle x-m,s\rangle}- \langle m,s\rangle| \leq \frac{s^2}{2} \left( \|x-m\|^2 e^{|\langle x-m,s\rangle|}+\|x\|^2e^{|\langle x,s \rangle|} \right).
$$
Therefore, we obtain
\begin{eqnarray*}
\lefteqn{\left| \gamma(x-m) \gamma(B(x,r)) - \gamma(x) \gamma(B(x-m,r))\right|}\\  
&\leq& \gamma(x) \gamma(x-m)
  \int_{B(0,r)} e^{-\|s\|^2/2} \langle m,s\rangle ds  \\ 
  & & +\frac{r^2}{2} \gamma(x) \gamma(x-m)\left[ \|x-m\|^2 \int_{B(0,r)} e^{-\frac{\|s\|^2}{2} } e^{|\langle x-m,s\rangle|} ds
  +\|x\|^2 \int_{B(0,r)} e^{-\frac{\|s\|^2}{2} } e^{|\langle x,s\rangle|} ds \right] \\
  & = & 
  \frac{r^2}{2} \gamma(x) \gamma(x-m)\left[ \|x-m\|^2 \int_{B(0,r)} e^{-\frac{\|s\|^2}{2} } e^{|\langle x-m,s\rangle|} ds
  +\|x\|^2 \int_{B(0,r)} e^{-\frac{\|s\|^2}{2} } e^{|\langle x,s\rangle|} ds \right]  \\
  & \leq & \frac{r^2}{2} \gamma(x)\gamma(x-m)  \|x-m\|^2 \left( \int_{B(0,r)} e^{-\frac{\|s\|^2}{2} } e^{\langle x-m,s\rangle} ds + \int_{B(0,r)} e^{-\frac{\|s\|^2}{2} } e^{-\langle x-m,s\rangle} ds \right) \\
  & & + \frac{r^2}{2} \gamma(x)\gamma(x-m)  \|x\|^2 \left( \int_{B(0,r)} e^{-\frac{\|s\|^2}{2} } e^{\langle x,s\rangle} ds + \int_{B(0,r)} e^{-\frac{\|s\|^2}{2} } e^{-\langle x,s\rangle} ds \right)\\
  & = & \frac{r^2}{2} \left[ \|x-m\|^2 \gamma(x) [\gamma(B(x-m,r))+\gamma(B(m-x,r))]
  +
  \|x\|^2 \gamma(x-m) [\gamma(B(x,r))+\gamma(B(-x,r))]\right] \\
  & = & r^2 \left[ \|x-m\|^2\gamma(x-m) \gamma(B(x,r)) + \|x\|^2 \gamma(x) \gamma(B(x-m,r))\right],
\end{eqnarray*}
where the last line comes from the symmetry of the Gaussian distribution.
Using this last inequality in Inequality (\ref{eq:up_KLL}) yields:
\begin{equation}
\label{eq:tec1}
\left| \eta(B(x,r)) - \eta(x)\right|\leq  r^2 \left[ \|x-m\|^2 +\|x\|^2 \right].
\end{equation}
Now, we should remark that
$$
\gamma(B(0,r)) = \int_{B(0,r)} \frac{e^{-|u|^2/2} }{\sqrt{2 \pi}^d} du
\geq e^{-r^2/2} (2\pi)^{-d/2} \lambda(B(0,r)) \geq e^{-r^2/2} (2\pi)^{-d/2} r^d  \frac{\pi^{d/2}}{\Gamma(d/2+1)},
$$
where we used the direct computation of the Lebesgue volume of the unit ball in $\R^d$
$$
\lambda(B(0,1)) = \frac{\pi^{d/2}}{\Gamma(d/2+1)}.
$$
Therefore, we obtain that
$$
r^2 \leq \left( \frac{\gamma(B(0,r)) e^{r^2/2} (2\pi)^{d/2} \Gamma(d/2+1)}{\pi^{d/2}}\right)^{2/d} = 2 e^{r^2/d} \Gamma(d/2+1)^{2/d} \gamma(B(0,r))^{2/d}.
$$
Then, Equation \eqref{eq:gaussian_shift} on the volume of shifted balls entails
\begin{eqnarray*}
\forall x \in \R^d \quad \forall r >0 \qquad 
r^2 & \leq  & 2 e^{r^2/d} \Gamma(d/2+1)^{2/d} \left( \frac{\gamma(B(x,r)) e^{\|x\|^2/2} + \gamma(B(x-m,r)) e^{\|x-m\|^2/2}}{2}  \right)^{2/d} \\
& \leq & 2 e^{r^2/d} \Gamma(d/2+1)^{2/d}\left[ \gamma(x)^{-1}+\gamma(x-m)^{-1}\right]^{2/d}  \mu(B(x,r))^{2/d}.
\end{eqnarray*}
Using the Stirling formula, we have
$$
\Gamma(d/2+1) \leq 2 \sqrt{2\pi} (d/2+1)^{d/2+1/2} e^{-d/2-1}.
$$
We then plug-in this upper bound in the previous inequality and we deduce that:
\begin{eqnarray*}
r^2 & \leq & 2 e^{r^2/d} \frac{d}{2} \left( 2 \sqrt{2\pi} (1+2/d)^{d/2+1/2} e^{-d/2-1} \right)^{2/d} \left[ \gamma(x)^{-1}+\gamma(x-m)^{-1}\right]^{2/d}  \mu(B(x,r))^{2/d}\\
&
\leq & d e^{r^2/d} \left[ \gamma(x)^{-1}+\gamma(x-m)^{-1}\right]^{2/d}  \mu(B(x,r))^{2/d} 
\sup_{d'\geq 1} \left\{ \left( 2 \sqrt{2\pi} (1+2/d')^{d'/2+1/2} e^{-d'/2-1} \right)^{2/d'}\right\}.
\end{eqnarray*}
Some straightforward algebra yields:
$$
\sup_{d'\geq 1} \left\{ \left( 2 \sqrt{2\pi} (1+2/d')^{d'/2+1/2} e^{-d'/2-1} \right)^{2/d'}\right\} \leq 72 \pi e^{-3} \leq 12,
$$
which entails that:
$$
\left| \eta(B(x,r)) - \eta(x)\right|\leq  12 d e^{r^2/d} \left[ \|x-m\|^2 +\|x\|^2 \right] \left[ \gamma(x)^{-1}+\gamma(x-m)^{-1}\right]^{2/d}  \mu(B(x,r))^{2/d}.
$$
\end{proof}

\subsection{Analysis of the Nearest Neighbor classifier in finite dimension}
\label{proof:knngauss}
Below, $\Phi_{k,n}$ refers to the $k$ nearest neighbor classifier given a $n$ sample $\mathcal{D}_n := (X_1,Y_1),\ldots,(X_n,Y_n)$ in $\mathbb{R}^d$ with a Gaussian translation model.

\begin{proof}[Proof of Proposition \ref{theo:knngauss}]  We begin with a classical decomposition of the excess risk, we  have:
$$
\mathcal{R}_{f,g}(\Phi_{k,n,d})-\mathcal{R}_{f,g}(\Phi_d^{\star})= \mathbb{E} \left[ |2 \eta_d(X)-1| \mathds{1}_{\Phi_{k,n,d}(X) \neq \Phi_d^{\star}(X)} \right].
$$
Consider a small $\epsilon$, whose value will be fixed later on. For any $\delta>0$, we use the simple lower bound
\begin{eqnarray*}
\mathcal{R}_{f,g}(\Phi_{k,n,d})-\mathcal{R}_{f,g}(\Phi_d^{\star})&\geq&  \mathbb{E} \left[|2 \eta_d(X)-1| 
\mathds{1}_{\delta\epsilon<|\eta(X)-1/2|<\epsilon}
\mathds{1}_{\Phi_{k,n,d}(X) \neq \Phi_d^{\star}(X)} \right], \\
 &\geq& \delta\epsilon
\mathbb{E} \left[
\mathds{1}_{\delta\epsilon<|\eta(X)-1/2|<\epsilon}
\mathds{1}_{\Phi_{k,n,d}(X) \neq \Phi_d^{\star}(X)} \right], \\
& \geq & \delta\epsilon
\mathbb{E}_{X} \left[
\mathds{1}_{\delta\epsilon<|\eta(X)-1/2|<\epsilon}
\mathbb{E}_{\otimes^n} \left[
\mathds{1}_{\Phi_{k,n}(X) \neq \Phi_d^{\star}(X)} \right]\right],\\
& \geq & \delta\epsilon
\mathbb{E}_{X} \left[
\mathds{1}_{\delta\epsilon<|\eta(X)-1/2|<\epsilon}
\mathbb{E}_{\otimes^n} \left[
\mathds{1}_{\Phi_{k,n}(X) \neq \Phi_d^{\star}(X)} \right] \mathds{1}_{\lbrace \| X \| \leq R_d \rbrace}  \right],
\end{eqnarray*}
where $R_d := \tau\sqrt{d}$ for some $\tau >0$.  Proposition \ref{prop:gaussian_alpha} gives $\beta_d=2/d$ in our situation. From Proposition \ref{prop:gaussian_alpha}, the value of $L_R$ given in \eqref{eq:value_LR}, and the choice of $R = R_d$, we know that a $\tau >0$ exists such that $L_{R_d} = d$. It is important to notice that $R$ is independent of $n$.

We now use Lemma 5, Lemma 17 and Lemma 18 of \cite{chaudhuri}: for any $(\beta_d,L_R)$-smooth distribution (see the dependency on $\beta_d$ in Equation \ref{eq:smooth}), then a constant $\kappa>0$ exists such that for any $k$ and $n$:
$$
\mathbb{P}_{\otimes^n}\left[\Phi_{k,n}(X) \neq \Phi_d^{\star}(X) \, \left\vert \, |\eta(X)-1/2| \leq \frac{1}{\sqrt{k}} - L_{R_d} \left(\frac{k+\sqrt{k}+1}{n}\right)^{\beta_d}\right.\right] \geq \kappa.
$$
According to our choice of $k_n$ and $R_d$, we then have for any $\delta>0$:
\begin{eqnarray}%
\mathbb{E} \mathcal{R}(\Phi_{k_n,n,d})-\mathcal{R}(\Phi_d^{\star}) & \geq & \kappa \delta \epsilon \mathbb{E}_{X} \left[
\mathds{1}_{\delta \epsilon <|\eta(X)-1/2|<\epsilon} 
\mathds{1}_{|\eta(X)-1/2|<\frac{1}{\sqrt{k_n}} - L_R \left(\frac{k_n+\sqrt{k_n}+1}{n}\right)^{\beta} }\mathds{1}_{\lbrace \| X \| \leq R_d \rbrace} \right]\nonumber\\
& \geq &  \kappa \delta \epsilon \mathbb{E}_{X} \left[
\mathds{1}_{\delta \epsilon <|\eta(X)-1/2|<\epsilon} 
\mathds{1}_{|\eta(X)-1/2|< \left( \frac{k_n}{n}\right)^{2/d} \left[2d - d \left(1+k_n^{-1/2}+k_n^{-1}\right)^{2/d}  \right] }\mathds{1}_{\lbrace \| X \| \leq R_d \rbrace} \right] \nonumber\\
& \geq & \kappa \delta \epsilon \mathbb{E}_{X} \left[
\mathds{1}_{\delta \epsilon <|\eta(X)-1/2|<\epsilon} 
\mathds{1}_{|\eta(X)-1/2|< \frac{d}{2} \left( \frac{k_n}{n} \right)^{2/d}}\mathds{1}_{\lbrace \| X \| \leq R_d \rbrace} \right],
\label{eq:lower_bound}
\end{eqnarray}
where we used that $k\leq K_n$. To obtain the best achievable lower bound in (\ref{eq:lower_bound}), $\epsilon$ has to be chosen as large as possible. We are driven to  the choice ($\epsilon$ depends on $n$ and $d$):
$$\epsilon_n  = \frac{1}{2} d\left( \frac{k_n}{n}\right)^{2/d} .$$
Then one has for any value of $\delta$ smaller than $1$:
\begin{eqnarray*}
\mathcal{R}_{f,g}(\Phi_{k,n})-\mathcal{R}(\Phi_d^{\star}) 
& \geq & c_{\delta} \epsilon_n \mathbb{E}_{X} \left[ \mathds{1}_{\delta \epsilon_n <|\eta(X)-1/2|<\epsilon_n} \mathds{1}_{\lbrace \| X \| \leq R_d \rbrace} \right], \\
& \geq &c_{\delta} \epsilon_n \mathbb{P}_X \left(\lbrace \delta \epsilon_n <|\eta(X)-1/2|<\epsilon_n \rbrace \cap \lbrace \| X \| \leq R_d \rbrace\right) 
\end{eqnarray*}
Again, we shall use the margin property of the Gaussian translation model: Theorem \ref{thm:crown} in Appendix \ref{s:A} shows that a $\delta$ exists (independent on $n$) such that
$$
\mu \left( \delta t \leq \left| \eta(X) - \frac{1}{2} \right| \leq t   \right) \geq  \check{c}_\delta \, t,
$$
where $\check{c}$ is a small enough positive constant. In the same time, there exists a constant $C_\tau$ such that
$$ \mathbb{P}(\| X\| \leq \tau \sqrt{d}) \geq C_\tau.$$
The last bound of the excess risk above together with the previous inequality lead to a lower bound of the order $\epsilon_n^{2}$: a constant $C_1$ independent on $n$ and $d$ exists such that
$$
\mathbb{E} \mathcal{R}(\Phi_{k,n,d})-\mathcal{R}(\Phi_d^{\star}) \geq  C_1 d^2 \left( \frac{k}{n}\right)^{4/d} \geq \frac{C_1}{k} 
$$
We stress that this lower bound is uniform for any $k \leq K_n$ which leads to the desired result.  The upper bound involved in the statement of Proposition \ref{theo:knngauss} is a simple consequence of Theorem 4.3 of \cite{GKM16}.
\end{proof}

\subsection{Proof of Theorem \ref{thm:lwr_bound_knn}}
\label{proof:lwr_bound_knn}

\subsubsection{Technical result}

Below, we establish a complementary result with a lower bound on the probability involved in the margin condition. This will make it possible to derive a lower bound of the nearest neighbour classifier.

\begin{prop}\label{prop:mino_proba}
Let $X$ distributed according to the model (\ref{eq:model}) and for any fixed $\Delta=\|f-g\|_2$, then:
$$
\forall \epsilon < 1/4 \qquad
\mathbb{P}\left(\left| \eta(X)-\frac{1}{2} \right|\leq \epsilon \right) \geq (2\pi)^{-1/2} \left[  \frac{\epsilon}{\Delta}  e^{-(1+\Delta/2)^2/2}  \wedge \frac{e^{-1/2}}{2} \right] .
$$
\end{prop}
\begin{proof}
To alleviate the notations, we skip the dependency on $X$ and write $\eta-1/2=\frac{q_f-q_g}{2(q_f+q_g)}$. We then repeat the arguments used above:
\begin{eqnarray*}
\PR\left( \left| \eta-\frac{1}{2}\right|\leq \epsilon \right) & = & 
\PR\left( \frac{|q_f-q_g|}{2(q_f+q_g)} \leq \epsilon \right) \\
& = & \PR\left( \frac{q_f-q_g}{2(q_f+q_g)} \leq \epsilon  \, , \, q_f>q_g \right) + 
\PR\left( \frac{q_g-q_f}{2(q_f+q_g)} \leq \epsilon  \, , \, q_f<q_g \right)
\\
& \geq &\PR \left(\frac{q_f-q_g}{2 q_f} \leq \epsilon  \, , \, q_f>q_g \right) + 
\PR\left( \frac{q_g-q_f}{2q_g} \leq \epsilon  \, , \, q_f<q_g \right) \\
& = & \PR \left( 0\leq  1- \frac{q_g}{q_f} \leq 2 \epsilon  \right) + 
\PR\left( 0\leq 1- \frac{q_f}{q_g} \leq \epsilon  \right) \\
& = & \PR \left( \log(1-2\epsilon) \leq \log\left(\frac{q_g}{q_f}\right) \leq 0  \right) + 
\PR \left( \log(1-2\epsilon) \leq \log\left(\frac{q_f}{q_g}\right) \leq 0  \right)
\end{eqnarray*}
We compute a lower bound of the first bound (the second term being handled similarly.
For $\epsilon<1/4$, it can be checked that $\log(1-2 \epsilon) < - \epsilon$. Therefore, we have
$$
\PR \left( \log(1-2\epsilon) \leq \log\left(\frac{q_g}{q_f}\right) \leq 0  \right) \geq 
\PR \left(  - \epsilon \leq \log\left(\frac{q_g}{q_f}\right) \leq 0  \right) 
$$
Using again the conditional distribution of $X \vert Y$ and that $Y$ is distributed according to a Bernoulli distribution $\mathcal{B}(1/2)$, we have
$$
\PR \left(  - \epsilon \leq \log\left(\frac{q_g}{q_f}\right) \leq 0  \right) 
= \frac{1}{2} \PR \left( - \epsilon \leq \frac{\Delta^2}{2} + \Delta \xi \leq 0 \right) + \frac{1}{2}
\PR \left( - \epsilon \leq -\frac{\Delta^2}{2} + \Delta \xi \leq 0 \right),
$$
where $\Delta = \|f-g\|_2$ and $\xi$ is distributed according to $\mathcal{N}(0,1)$. We can conclude that

$$
\PR \left(\left| \eta-\frac{1}{2}\right|\leq \epsilon \right) \geq \frac{1}{2} \int_{-\frac{\epsilon}{\Delta}-\frac{\Delta}{2}}^{-\Delta/2} \frac{e^{-t^2/2}}{\sqrt{2\pi}} dt + \frac{1}{2} \int_{-\frac{\epsilon}{\Delta}+\frac{\Delta}{2}}^{\Delta/2} \frac{e^{-t^2/2}}{\sqrt{2\pi}} dt.
$$
Then, we split our study into two cases:
\begin{itemize}
\item If $\epsilon \leq \Delta$, then $\forall t \in [-\frac{\epsilon}{\Delta}-\frac{\Delta}{2},\frac{\Delta}{2}]$ and $\frac{e^{-t^2/2}}{\sqrt{2 \pi}} \geq \frac{e^{-(1+\Delta/2)^2/2}}{\sqrt{2 \pi}}$
and in this case:
$$
\PR \left(\left| \eta-\frac{1}{2}\right|\leq \epsilon \right) \geq \frac{e^{-(1+\Delta/2)^2/2}}{\sqrt{2 \pi}} \frac{\epsilon}{\Delta}
$$
\item If $\epsilon > \Delta$, 
\begin{eqnarray*}
\PR \left(\left| \eta-\frac{1}{2}\right|\leq \epsilon \right)& \geq & \frac{1}{2}  \int_{-\frac{\epsilon}{\Delta}}^{-\Delta/2} \frac{e^{-t^2/2}}{\sqrt{2\pi}} dt + \frac{1}{2} \int_{-\frac{\epsilon}{\Delta}}^{0} \frac{e^{-t^2/2}}{\sqrt{2\pi}} dt\\
& \geq & \int_{-\frac{\epsilon}{\Delta}}^{-\Delta/2} \frac{e^{-t^2/2}}{\sqrt{2\pi}} dt \\
& \geq & (2\pi)^{-1/2} \left[ \int_{-1}^0  e^{-t^2/2} dt - \frac{\Delta}{2} \right]\\
& \geq & \frac{e^{-1/2}}{2 \sqrt{2 \pi}},
\end{eqnarray*}
where the last bound comes from the fact that $\int_{-1}^0  e^{-t^2/2} dt \geq e^{-1/2}$ while $\Delta<\epsilon<1/4<e^{-1/2}$.
\end{itemize}
This ends the proof of the Proposition.
\end{proof}

A key consequence   is the lower bound of the area of the crown $ \delta \epsilon \leq  | \eta-1/2|\leq \epsilon$ for $\delta$ small enough.

\begin{prop}\label{thm:crown}
Let $X$ given by (\ref{eq:model}) and for any fixed $\Delta=\|f-g\|_2$, if we set  $\delta =  \frac{e^{-(1+\Delta/2)^2/2}}{2 \sqrt{2\pi}}$, then:
$$
\forall \epsilon \leq \frac{1}{4} \wedge \Delta \qquad 
\mathbb{P}\left( \delta \epsilon \leq \left| \eta(X)-\frac{1}{2} \right|\leq \epsilon \right) \geq 
\delta \frac{\epsilon}{\Delta}.$$

\end{prop}
\begin{proof}

For a given $c>0$, we introduce $\delta =  \frac{e^{-(1+\Delta/2)^2/2}}{c \sqrt{2\pi}}$ and use the decomposition
\begin{eqnarray*}
\mathbb{P}\left( \delta \epsilon \leq \left| \eta(X)-\frac{1}{2} \right|\leq \epsilon \right) &= &
\mathbb{P}\left(\left| \eta(X)-\frac{1}{2} \right|\leq \epsilon \right) - \PR\left(\left| \eta(X)-\frac{1}{2} \right|\leq \delta \epsilon \right)\\
& \geq & c \delta \frac{\epsilon}{\Delta} -  \PR\left(\left| \eta(X)-\frac{1}{2} \right|\leq \delta \epsilon \right), \\
\end{eqnarray*}
where the last line comes from Proposition \ref{prop:mino_proba}. Now, we use Proposition \ref{thm:margin} to conclude that
$$
\mathbb{P}\left( \delta \epsilon \leq \left| \eta(X)-\frac{1}{2} \right|\leq \epsilon \right) \geq  (c-1) \delta \frac{\epsilon}{\Delta}.
$$
We now choose $c=2$ and obtain the desired result.
\end{proof}

\begin{rem}
Proposition \ref{thm:crown} states that when $\Delta$ is small, the measure of the uncertainty area for the classification ($\eta \simeq 1/2$) has an important mass although this measure decreases linearly with the inverse of $\Delta$. This result is intuitive and  translates the fact that for large values of $\Delta$, the classification problem is easy (the two classes are well separated) and there is a steep transition from $\{\eta>1/2\}$ to $\{\eta<1/2\}$.
\end{rem}
\subsubsection{Logarithmic rate of Nearest Neighbor rule}

This last paragraph is devoted to the proof of Theorem \ref{thm:lwr_bound_knn}, which shows that a sample splitting strategy used with the NN rule is not efficient with a logarithmic decrease of the misclassification rate.

\begin{proof}[Proof of Theorem \ref{thm:lwr_bound_knn}]

Since the truncation is chosen once for all at the beginning of the classification process with a sample-splitting strategy, our elementary starting point is given by:
$$
\mathcal{R}_{f,g}(\hat{\Phi}_{kNN}^{\hat{d}})-\mathcal{R}_{f,g}(\Phi^{\star})  \geq \min_{d \in \mathbb{N}}   \mathcal{R}_{f,g}(\hat{\Phi}_{kNN}^{d})-\mathcal{R}_{f,g}(\Phi^{\star}).
$$
For any frequency threshold $d \in \mathbb{N}$, we decompose the excess risk as:
\begin{equation}\label{eq:bayes_bayesd}
\mathcal{R}_{f,g}(\Phi_{k,n,d})-\mathcal{R}_{f,g}(\Phi^{\star}) = \mathcal{R}_{f,g}(\Phi_{k,n,d})-\mathcal{R}_{f,g}(\Phi^\star_d)+ \mathcal{R}_{f,g}(\Phi_d^\star) - \mathcal{R}_{f,g}(\Phi^{\star}),
\end{equation}
where $\Phi^{\star}_d$ is the Bayes classification rule with the Gaussian $d$-dimensional model that involves the  first $d$ frequencies.
Proposition \ref{theo:knngauss} shows that if $\Delta^2=\|f-g\|_2^2$, then a constant $c_{\Delta,1}$ exists such that:
\begin{equation}\label{eq:knnbye}
\mathcal{R}_{f,g}(\Phi_{k,n})-\mathcal{R}_{f,g}(\Phi^\star_d) \geq c_{\Delta,1} n^{-\frac{4}{d+4}}.
\end{equation}
We now focus on the second term of (\ref{eq:bayes_bayesd}). Since $Y$ is distributed according to a Bernoulli distribution $\mathcal{B}(1/2)$, we have:
$$
 \mathcal{R}_{f,g}(\Phi^\star_d) - \mathcal{R}_{f,g}(\Phi^{\star}) = \frac{1}{2}\left( \mathbb{P}_{f} [\Phi^\star_d=1] - \mathbb{P}_{f} [\Phi^\star=1]\right) +  \frac{1}{2}\left( \mathbb{P}_{g} [\Phi^\star_d=0] - \mathbb{P}_{g} [\Phi^\star=0]\right).
$$
We compute the first term (the second term is handled similarly). Let $f,g$ be fixed function belonging to $\Hs(R)$ which will be made precise latter on. We define $\Delta_d^2=\|g-f\|_{d,2}^2$ the $L^2$ norm of $g-f$ restricted to the first $d$ coefficients. If $\xi$ is a standard Gaussian random variable, we have:
$$
\mathbb{P}_{f} [\Phi^\star_d(X)=1] = \mathbb{P}_f \left[ \langle X-f,g-f\rangle_d> \frac{\|g-f\|_{d,2}^2}{2}  \right]= \mathbb{P}\left( \xi \Delta_d > \frac{\Delta_d^2}{2}\right)
$$
In the meantime, the second probability can be computed as
$$
\mathbb{P}_{f} [\Phi^\star(X)=1]=\mathbb{P}_f \left[ \langle X-f,g-f\rangle> \frac{\|g-f\|_{2}^2}{2}  \right]
= \mathbb{P}\left( \xi \Delta > \frac{\Delta^2}{2}\right).
$$
Hence, we deduce that
$$
\mathbb{P}_{f} [\Phi^\star_d(X)=1] - \mathbb{P}_{f} [\Phi^\star(X)=1]  = \int_{\Delta_d/2}^{\Delta} \gamma(s) ds \geq \gamma(\Delta) \frac{\Delta-\Delta_d}{2} = \gamma(\Delta) \frac{\Delta^2-\Delta_d^2}{2(\Delta+\Delta_d)} \geq \frac{\Delta^2-\Delta_d^2}{4 \Delta}\gamma(\Delta).
$$
We can then find $f$ and $g$ such that $\Delta^2 < 1$ and $\Delta^2-\Delta_d \sim d^{-2s}$ because $f$ and $g$ shall belong to the Sobolev space $\mathcal{H}_s(R)$. Hence, we deduce the following lower bound on the excess risk between the truncated Bayes rule and the non parametric Bayes rule: a constant $c_{\Delta,2}$ exists such that
\begin{equation}\label{eq:bayesbie}
 \mathbb{P}_{f} [\Phi^\star_d=1] - \mathbb{P}_{f} [\Phi^\star=1] \geq c_{\Delta,2} \, d^{-2s}.
\end{equation}
Gathering Equations (\ref{eq:knnbye}) and (\ref{eq:bayesbie}), we deduce that
$$
   \mathcal{R}_{f,g}(\hat{\Phi}_{k,n,\hat d})-\mathcal{R}_{f,g}(\Phi^{\star})  \geq c_{\Delta,3} \min_{d \in \mathbb{N}^\star}
\left[d^{-2s} +n^{-\frac{4}{4+d}}\right].
$$
We then optimize our lower bound with respect to $d$ and we obtain the conclusion of the proof.
\end{proof}



\bibliographystyle{plain}
\bibliography{references}

\end{document}